\documentclass[11pt]{article}

\usepackage[compress]{natbib}
\usepackage{amsmath,amssymb,amsthm,amsfonts}
\usepackage[margin=1in]{geometry}
\usepackage{setspace}
\usepackage{subcaption}
\usepackage{bm}
\usepackage[dvipsnames]{xcolor}
\usepackage{float}
\usepackage[ruled,linesnumbered]{algorithm2e}
\usepackage{hyperref}
\hypersetup{
  colorlinks=true,
  citecolor=MidnightBlue,
  linkcolor=red!80!black,
  urlcolor=MidnightBlue,
  hypertexnames=false
}
\usepackage{caption}
\usepackage{url}
\usepackage{booktabs}
\usepackage{array}
\usepackage{nicefrac}
\usepackage{graphicx}
\usepackage{thm-restate}
\usepackage[export]{adjustbox}

\usepackage{tabularx,colortbl}  

\usepackage[T1]{fontenc}
\usepackage{microtype}
\usepackage{mathpazo}
\usepackage[scaled]{helvet}
\usepackage{courier}
\normalfont

\usepackage{threeparttable}
\allowdisplaybreaks[4]
\usepackage[shortlabels]{enumitem}

\usepackage{tikz}
\usetikzlibrary{fit,arrows.meta,positioning}
\tikzset{
  highlight/.style={
    rectangle,
    blend mode=multiply,
    draw=blue!90!black,
    thick,
    rounded corners=0.3mm,
    inner sep=0.5pt
  }
}

\usepackage{tcolorbox}

\usepackage{mathtools}
\usepackage{pifont}
\usepackage[titletoc]{appendix}
\usepackage{thmtools}

\declaretheoremstyle[parent=section]{definitionwithend}

\declaretheorem[style=definitionwithend]{theorem}
\declaretheorem[style=definitionwithend]{proposition}
\declaretheorem[style=definitionwithend]{definition}

\declaretheorem[style=definitionwithend]{remark}

\declaretheorem[style=definitionwithend]{lemma}

\newcommand{\R}{\mathbb R}
\newcommand{\B}{\mathbb B}
\newcommand{\N}{\mathbb N}

\newcommand{\bz}{\mathbf 0}
\newcommand{\x}{\bm x}

\newcommand{\y}{\bm y}
\newcommand{\z}{\bm z}
\newcommand{\uvec}{\bm u}
\newcommand{\vvec}{\bm v}
\renewcommand{\u}{\bm u}
\renewcommand{\v}{\bm v}
\newcommand{\w}{\bm w}
\renewcommand{\a}{\bm a}
\newcommand{\s}{\bm s}
\newcommand{\omegavec}{\bm\omega}
\newcommand{\ee}{\bm e}
\newcommand{\bI}{\mathbf I}
\newcommand{\bA}{\mathbf A}
\newcommand{\bB}{\mathbf B}
\newcommand{\bD}{\mathbf D}
\newcommand{\bM}{\mathbf M}
\newcommand{\bU}{\mathbf U}
\newcommand{\bV}{\mathbf V}
\newcommand{\bW}{\mathbf W}
\newcommand{\cA}{\mathcal A}

\newcommand{\cF}{\mathcal F}

\newcommand{\cO}{\mathcal O}
\newcommand{\cT}{\mathcal T}

\newcommand{\cX}{\mathcal X}
\newcommand{\cY}{\mathcal Y}
\def\ord{\bm{\pi}}
\newcommand{\sA}{\mathsf A}

\newcommand{\fbar}{\bar f}
\newcommand{\supp}{\operatorname{supp}}
\newcommand{\diam}{\operatorname{diam}}
\DeclareMathOperator{\dist}{dist}
\DeclareMathOperator{\proj}{proj}

\newcommand{\prox}{\operatorname{prox}}
\newcommand{\argminop}{\operatorname*{argmin}}

\newcommand{\OS}{\mathrm{\scriptscriptstyle  OS}}
\newcommand{\GS}{\mathrm{\scriptscriptstyle  GS}}
\newcommand{\ind}{\mathbf 1}
\DeclarePairedDelimiter{\norm}{\lVert}{\rVert}
\DeclarePairedDelimiterX{\ip}[2]{\langle}{\rangle}{#1,#2}

\providecommand{\FOAM}{\textnormal{\textsc{FOAM}}}

\providecommand{\TrackedFOAM}{\textnormal{\textsc{Tracked-FOAM}}}

\title{Optimal Deterministic First-Order Oracle Complexity for
Nonconvex-Concave Minimax Optimization}
\date{\today}
\author{%
Siyu Pan\thanks{%
Sauder School of Business, The University of British Columbia.
Email: \texttt{siyu.pan@sauder.ubc.ca}.}
\qquad
Taoli Zheng\thanks{%
RIKEN AIP.
Email: \texttt{taoli.zheng@riken.jp}.}
\qquad
Jiajin Li\thanks{%
Sauder School of Business, The University of British Columbia.
Email: \texttt{jiajin.li@sauder.ubc.ca}.}
}
\begin{document}
\maketitle

\begin{abstract}
We study the deterministic first-order oracle complexity of smooth nonconvex-concave minimax optimization over a bounded convex dual domain under two notions of stationarity: optimization stationarity (OS) and game stationarity (GS).
OS concerns stationarity of the value function, viewed as a
structured nonsmooth objective obtained by maximizing over the
dual variable. GS concerns the joint first-order optimality conditions for both primal minimization and dual maximization, each with the other variable
held fixed. Existing upper bounds on the oracle complexity of first-order methods for finding $\epsilon$-OS and $\epsilon$-GS points have different dependence on $\epsilon$, but whether this difference reflects an intrinsic separation in computational hardness remains unresolved.
We resolve this question by establishing worst-case lower bounds for arbitrary deterministic first-order methods:
$\Omega(\epsilon^{-3})$ oracle queries for finding an
$\epsilon$-OS point and $\Omega(\epsilon^{-5/2})$ for finding an $\epsilon$-GS point, with the remaining problem parameters fixed. These bounds match existing upper bounds up to logarithmic factors. To close the remaining logarithmic gap, we introduce \texttt{Tracked-FOAM}, a new first-order algorithm whose oracle
complexity matches our lower bounds for both OS and GS.
Together, these results establish an intrinsic separation
between OS and GS in the optimal accuracy dependence
of deterministic first-order oracle complexity.
\end{abstract}


\section{Introduction}
\label{sec:introduction}
\begingroup
In this paper, we investigate the deterministic first-order oracle complexity of the smooth
nonconvex-concave (NC-C) minimax problem
\begin{equation}
    \min_{\x\in\cX}\max_{\y\in\cY} f(\x;\y),
    \tag{P}\label{eq:prob}
\end{equation}
where $f$ is jointly $\ell$-smooth, possibly nonconvex in $\x$, and concave in $\y$.
The sets $\cX\subseteq\mathbb R^m$ and $\cY\subseteq\mathbb R^n$ are nonempty and convex,
with $\cX$ closed and $\cY$ compact with diameter at most $D_{\cY}$.

Two notions of stationarity are commonly used for
Problem~\eqref{eq:prob}
\citep{jin2020local,lin2020near,li2025nonsmooth,li2026smoothing}. The first, optimization stationarity (OS), views the problem as minimizing the value function
$\Phi(\x):=\max_{\y\in\cY}f(\x;\y)+\iota_{\cX}(\x)$,
where $\iota_{\cX}$ is the indicator function of $\cX$.
Although $f$ is smooth, $\Phi$ can be nonsmooth, and OS measures stationarity of this structured nonsmooth objective through its Moreau envelope~\citep{moreau1965proximite,davis2019stochastic}.
The second, game stationarity (GS), concerns the joint first-order optimality conditions for primal minimization and dual maximization, each with the other variable held fixed.
The precise definitions of $\epsilon$-OS and $\epsilon$-GS
are given in Section~\ref{sec:preliminaries}.

Existing upper bounds on the oracle complexity of first-order
methods have different accuracy dependence for  $\epsilon$-OS
and $\epsilon$-GS points.
With the remaining problem parameters fixed,
\textsc{Prox-DIAG} \citep{thekumparampil2019efficient},
primal-dual smoothing \citep{zhao2024primal},
\textsc{Minimax-PPA} \citep{lin2020near}, and
\textsc{FNE Search} \citep{ostrovskii2021efficient}
attain $\cO(\epsilon^{-3}\log^2(1/\epsilon))$
oracle complexity for finding an $\epsilon$-OS point.
For finding an $\epsilon$-GS point,
\textsc{Minimax-PPA} \citep{lin2020near} and
\textsc{FNE Search} \citep{ostrovskii2021efficient}
attain $\cO(\epsilon^{-5/2}\log^2(1/\epsilon))$
oracle complexity. More recently, \textsc{Perturbed Smoothed FOAM} \citep{li2026smoothing} improves both bounds by reducing
the logarithmic factor from $\log^2(1/\epsilon)$
to $\log(1/\epsilon)$. Table~\ref{tab:ncc-complexity} summarizes these bounds and their dependence on the problem parameters.

% First-order methods for Problem~\eqref{eq:prob} include
% proximal-point methods \citep{rafique2022weakly},
% two-timescale gradient descent-ascent \citep{lin2020gradient},
%  smoothing methods
% \citep{zhang2020single,zheng2023universal}, etc.
% With the remaining problem parameters fixed,
% \textsc{Prox-DIAG}, primal-dual smoothing,
% \textsc{Minimax-PPA}, and \textsc{FNE Search}
% attain $\cO(\epsilon^{-3}\log^2(1/\epsilon))$
% oracle complexity for finding an $\epsilon$-OS point
% \citep{thekumparampil2019efficient,zhao2024primal,
% lin2020near,ostrovskii2021efficient}.
% For $\epsilon$-GS, \textsc{Minimax-PPA} and
% \textsc{FNE Search} achieve the sharper accuracy dependence
% $\cO(\epsilon^{-5/2}\log^2(1/\epsilon))$.
% More recently, \textsc{Perturbed Smoothed FOAM}
% \citep{li2026smoothing} improves both bounds by reducing
% the logarithmic factor from $\log^2(1/\epsilon)$
% to $\log(1/\epsilon)$.
% Table~\ref{tab:ncc-complexity} summarizes these bounds
% and their dependence on the problem parameters.

It remains unclear, however, whether the different dependence
on $\epsilon$ for OS and GS reflects an intrinsic separation
in computational hardness, and whether the remaining logarithmic
factors can be removed. This motivates the following question:
\begin{quote}
\emph{What are the optimal deterministic first-order oracle
complexities of finding $\epsilon$-OS and $\epsilon$-GS points,
and does the choice of stationarity notion fundamentally
change these complexities?}
\end{quote} 

We resolve these questions by establishing matching lower and
upper bounds for both OS and GS across all parameter regimes.
Our lower bounds hold for arbitrary deterministic first-order
methods.
Let $\x_0\in\cX$ be the initial point, and define the initial gap
of the value function as
$\Delta:=\Phi(\x_0)-\inf_{\x\in\cX}\Phi(\x)$.
In the principal regime reported in
Table~\ref{tab:ncc-complexity}, the optimal oracle complexities are
\[
    \underbrace{
    \Theta\!\left(\frac{\ell^2D_{\cY}\Delta}{\epsilon^3}\right)
    }_{\epsilon\text{-OS}}
    \qquad\text{and}\qquad
    \underbrace{
    \Theta\!\left(
        \frac{\ell^{3/2}D_{\cY}^{1/2}\Delta}{\epsilon^{5/2}}
    \right)
    }_{\epsilon\text{-GS}}.
\]
These rates establish that the different dependence on $\epsilon$
reflects an intrinsic separation in computational hardness
between the two stationarity notions.

Our lower bounds rely on structurally different hard instances.
For the  OS lower bound, we start from
nonconvex-strongly-concave instances whose dual strong-concavity
parameter vanishes as $\epsilon\to0$.
By contrast, the GS instances are not strongly concave in
the dual variable.
Both constructions build on the zero-chain framework
\citep{nesterov2018lectures,li2021lower}.
To obtain lower bounds for arbitrary deterministic first-order
methods, we develop two direct construction lemmas, which may
be of independent interest.
These lemmas use auxiliary variables and orthogonal rotations
to construct a hard instance directly for each algorithm,
bypassing the auxiliary zero-respecting algorithms used in
the finite-horizon resisting-oracle argument of
\citet{carmon2019lower} and the preliminary analysis of
linear-span methods in \citet{ouyang2021lower}.

To attain matching upper bounds, we develop \TrackedFOAM{}.
Following \citet{li2026smoothing}, we combine primal
smoothing and dual perturbation to form a sequence of
smooth strongly-convex-strongly-concave (SC-SC)
subproblems, which we approximately solve using the
accelerated \FOAM{} method of \citet{kovalev2022first}.
In the method of \citet{li2026smoothing}, each subproblem
is solved to a prescribed $\epsilon$-dependent accuracy,
incurring a logarithmic cost at every outer iteration.
In contrast, \TrackedFOAM{} retains the full accelerated
\FOAM{} state across successive subproblems.
We establish a global Lyapunov descent estimate that
jointly accounts for progress of the outer method and
the inexact errors in solving the SC-SC subproblems.
This estimate shows that each inner solve need only reduce the
\FOAM{} Lyapunov function by a fixed fraction.
This removes the need to prescribe an inner accuracy that
depends on $\epsilon$.
Consequently, the oracle complexity of \TrackedFOAM{}
matches our lower bounds for both OS and GS
across all parameter regimes.

% \TrackedFOAM{} retains the full accelerated \FOAM{} state
% across successive subproblems.
% We establish a global Lyapunov descent estimate that accounts
% jointly for progress of the outer method and the error from
% approximately solving the SC-SC subproblems.
% This estimate shows that a constant-factor \Siyu{AI feeling} contraction of the
% \FOAM{} Lyapunov function at each outer iteration suffices,
% eliminating the need for an $\epsilon$-dependent inner stopping
% tolerance.
% With parameter choices tailored to each stationarity criterion,
% the resulting oracle complexities match our lower bounds
% for both OS and GS without logarithmic factors.\Siyu{We didn't exactly mention different regimes, shall we mention these?}

Together, our results establish an intrinsic separation in optimal
accuracy dependence between finding stationary points of the
structured nonsmooth value function and satisfying the first-order equilibrium conditions of the original minimax problem.
Our matching lower and upper bounds give a complete
characterization of deterministic first-order oracle complexity under both stationarity notions across all parameter regimes.

\begin{table}[t]
\centering
\small
\begin{threeparttable}
\caption{Comparison of first-order oracle complexity bounds
for OS and GS in their respective principal regimes:
$0<\epsilon\lesssim\min\{\ell D_{\cY},\sqrt{\ell\Delta}\}$
for OS and
$0<\epsilon\lesssim\min\{\ell D_{\cY},\Delta/D_{\cY}\}$
for GS.
Our upper and lower bounds match in all problem parameters
under both stationarity notions.}
\label{tab:ncc-complexity}

\setlength{\tabcolsep}{4pt}
\renewcommand{\arraystretch}{1.2}
\renewcommand{\tabularxcolumn}[1]{m{#1}}
\newcommand{\tblcite}[1]{\par\vspace{1pt}{\footnotesize\citep{#1}}}
\newcommand{\tblOS}[2]{%
  $\displaystyle #1\!\left(
  \frac{\ell^2D_{\cY}#2}{\epsilon^3}
  \log^2\!\frac{1}{\epsilon}\right)$}
\newcommand{\tblGS}[2]{%
  $\displaystyle #1\!\left(
  \frac{\ell^{3/2}D_{\cY}^{1/2}#2}{\epsilon^{5/2}}
  \log^2\!\frac{1}{\epsilon}\right)$}

\begin{tabularx}{\linewidth}{
  >{\raggedright\arraybackslash}m{0.28\linewidth}
  >{\centering\arraybackslash}X
  >{\centering\arraybackslash}X}
\toprule
\textbf{Method}
& \textbf{Optimization stationarity}\newline (OS)
& \textbf{Game stationarity}\newline (GS) \\
\midrule
\addlinespace[4pt]

Prox-DIAG\tnote{a}
\tblcite{thekumparampil2019efficient}
& \tblOS{\cO}{\Delta}
& \textemdash \\
\addlinespace[9pt]

Primal-dual smoothing\tnote{a}
\tblcite{zhao2024primal}
& \tblOS{\cO}{\Delta}
& \textemdash \\
\addlinespace[9pt]

Minimax-PPA
\tblcite{lin2020near}
& \tblOS{\cO}{\Delta}
& \tblGS{\cO}{\Delta} \\
\addlinespace[9pt]

FNE Search
\tblcite{ostrovskii2021efficient}
& \tblOS{\cO}{\Delta}
& \tblGS{\cO}{\Delta} \\
\addlinespace[9pt]

Perturbed Smoothed FOAM\tnote{b}
\tblcite{li2026smoothing}
& $\displaystyle \cO\!\left(
  \frac{\ell^2D_{\cY}\Delta_{p_0}}{\epsilon^3}
  \log\!\frac{1}{\epsilon}\right)$
& $\displaystyle \cO\!\left(
  \frac{\ell^{3/2}D_{\cY}^{1/2}\Delta_{p_0}}{\epsilon^{5/2}}
  \log\!\frac{1}{\epsilon}\right)$ \\
\addlinespace[5pt]

\midrule
\rowcolor[gray]{0.95}
\rule{0pt}{4.8ex}\textbf{\TrackedFOAM{}}
% \par{\footnotesize\textbf{This paper}}
& $\displaystyle \cO\!\left(
  \frac{\ell^2D_{\cY}\Delta}{\epsilon^3}\right)$
& $\displaystyle \cO\!\left(
  \frac{\ell^{3/2}D_{\cY}^{1/2}\Delta}{\epsilon^{5/2}}\right)$
\\[6pt]
\rowcolor[gray]{0.95}
\rule{0pt}{4.8ex}{\bf  Our lower bound}
% \par{\footnotesize\textbf{This paper}}
& $\displaystyle \Omega\!\left(
  \frac{\ell^2D_{\cY}\Delta}{\epsilon^3}\right)$
& $\displaystyle \Omega\!\left(
  \frac{\ell^{3/2}D_{\cY}^{1/2}\Delta}{\epsilon^{5/2}}\right)$
\\[6pt]
\bottomrule
\end{tabularx}

\begin{tablenotes}[flushleft]
\footnotesize
\setlength{\itemsep}{2pt}
\item[a] Prox-DIAG assumes $\cX=\mathbb R^m$; primal-dual smoothing
is specialized to the Euclidean setting. A dash means that the cited
analysis does not directly state a GS guarantee.
% \item[b] A projected correction converts the original stationarity
% guarantee to the normal-cone GS criterion without changing the displayed rate.
\item[b] $\Delta_{p_0}$ is the initial gap of the regularized envelope
in \citet{li2026smoothing}.
\end{tablenotes}
\end{threeparttable}
\end{table}

\paragraph{Concurrent works.} 
For upper bounds, \citet{ke2026efficient} study a more general
NC-C minimax formulation that allows convex nonsmooth regularizers
in place of constraint indicators.
Assuming that the primal minimum is attained, they develop
deterministic first-order algorithms, including parameter-free
variants, with leading GS oracle complexity
$\cO(\epsilon^{-5/2})$.
For the smooth constrained setting considered here,
\citet{zhang2026matching} develop a single-loop method with
leading oracle complexities $\cO(\epsilon^{-3})$ for OS and
$\cO(\epsilon^{-5/2})$ for GS.
Their full bounds additionally depend on initial gradient norms,
which need not be controlled by $\ell,D_{\cY},\Delta$.
Our upper bounds depend only on $\ell,D_{\cY},\Delta,\epsilon$
and match our lower bounds throughout the full parameter range.
In particular, our GS bound improves the parameter dependence
of both works in regimes with small initial gaps.
\TrackedFOAM{} also returns deterministic outputs under both
stationarity notions without assuming attainment of the primal minimum.

For lower bounds, \citet{zhang2026matching} and
\citet{wu2026lower} only establish OS lower bounds for
zero-respecting first-order methods in the regime
$\epsilon\lesssim\min\{\sqrt{\ell\Delta},\ell D_{\cY}\}$.
None of these concurrent works establishes a GS lower bound.
Our lower bounds cover both OS and GS across all parameter
regimes and apply to arbitrary deterministic first-order
methods, even for $C^\infty$ objectives.
% \paragraph{Concurrent work.}
% {
% For upper bounds, \citet{ke2026efficient} develop
% deterministic methods for nonsmooth composite problems whose leading
% GS complexity term scales as $\cO(\epsilon^{-5/2})$.
% They also provide variants that do not require prior knowledge of the
% problem parameters. Their stated assumptions include attainment of
% the primal minimum.
% \citet{zhang2026matching} develop a single-loop method with leading
% complexity terms of order $\cO(\epsilon^{-3})$ for OS and
% $\cO(\epsilon^{-5/2})$ for GS. 
% For smooth constrained problems, our GS bound improves the parameter
% dependence in the bounds of \citet{ke2026efficient} and
% \citet[Corollary~4.4]{zhang2026matching} in regimes with small
% initial gaps.
% The full OS and GS bounds of \citet{zhang2026matching} additionally
% contain initialization costs depending on the norms of the initial
% gradients, which are not bounded by $\ell,D_{\cY},\Delta$ alone.

% For lower bounds, \citet{zhang2026matching,wu2026lower} establish
% OS bounds for zero-respecting methods in the regime
% $\epsilon\lesssim\min\{\sqrt{\ell\Delta},\ell D_{\cY}\}$.
% The hard instances in these works are not twice differentiable
% everywhere, whereas ours are $C^\infty$.
% Our results characterize the optimal deterministic oracle complexity
% for both OS and GS over all positive parameter values.
% The lower bounds cover arbitrary deterministic algorithms, and
% \TrackedFOAM{} attains matching bounds with deterministic outputs
% and no additional dependence on initial gradient norms.}

\paragraph{Notation.}
We use $\norm{\cdot}$ for Euclidean and induced operator norms,
$\norm{\cdot}_{\rm F}$ for the Frobenius norm,
$\B_D^d:=\{\z\in\R^d\mid\norm{\z}\le D/2\}$ for the Euclidean ball
of radius $D/2$, $[T]:=\{1,\ldots,T\}$, $[0]:=\varnothing$, and
$\N_0:=\N\cup\{0\}$. For $d\in\N_0$, we write $\bz_d$ for
the zero vector in $\R^d$. In particular, $\bz_0$ denotes the unique
vector in the zero-dimensional space, called the empty vector. Thus $\R^0=\B_D^0=\{\bz_0\}$. Empty sums are zero.
We use bold letters to denote vectors.
The semicolon in $f(\x;\y)$ separates the minimization and
maximization variables. We write $\operatorname{O}(d):=\{\bU\in
\R^{d\times d}\mid\bU^\top\bU=\bI_d\}$ and
$\dist(\x,\mathcal S):=\inf_{\z\in\mathcal S}\norm{\x-\z}$.
We also write
$\diam(\mathcal S):=\sup_{\x,\z\in\mathcal S}\norm{\x-\z}$
for the diameter of a nonempty set $\mathcal S$.
For a scalar function $h:\mathcal Z\to\R$, write
$\norm{h}_\infty:=\sup_{\z\in\mathcal Z}|h(\z)|$.
For symmetric matrices $\bA,\bB$, the notation $\bA\preceq\bB$
means that $\bB-\bA$ is positive semidefinite.
A coordinate ordering $\ord=(\pi_1,\ldots,\pi_d)$ is a permutation
of $[d]$. For $\z\in\R^d$, $\supp(\z)$ denotes the set of
indices of its nonzero coordinates. We write $\z_{\ord}$ for
the vector obtained by arranging the coordinates of $\z$ according
to $\ord$, and $\supp_{\ord}(\z)$ for the set of positions of
nonzero coordinates in this reordered vector.
We identify tuples of
vectors with their concatenation. We write $\bI_d$ for the $d\times d$ identity matrix and
$\ee_i$ for the $i$-th standard basis vector.
For a nonempty closed convex set $\mathcal C$, let
$\proj_{\mathcal C}$ denote the Euclidean projection onto
$\mathcal C$, and let $\iota_{\mathcal C}$ be its indicator
function, equal to zero on $\mathcal C$ and $+\infty$ elsewhere.
We write $\partial h$ for the convex subdifferential of a convex
function $h$, and $N_{\mathcal C}:=\partial\iota_{\mathcal C}$
for the normal cone mapping of $\mathcal C$.
For $h:\R^d\to\R\cup\{+\infty\}$ and $\lambda>0$, we write
\(
 \prox_{\lambda h}(\x)
 :=\argminop_{\z\in\R^d}
 \left\{h(\z)+\norm{\x-\z}^2/(2\lambda)\right\}
\)
whenever the minimum is attained at a unique point.
Unless otherwise stated, the notation
$\cO$, $\Omega$, $\Theta$, $\lesssim$, and $\asymp$
hides positive numerical constants independent of the problem parameters.

\paragraph{Organization.}
Section~\ref{sec:preliminaries} introduces the problem setup
and oracle model, and Section~\ref{subsec:main-results}
presents our main results.
Sections~\ref{sec:construction} and~\ref{gs:sec:construction}
present the OS and GS lower bound constructions, respectively,
whose properties are verified in
Sections~\ref{sec:verification} and~\ref{gs:sec:verification}.
Sections~\ref{sec:upper} and~\ref{sec:upper-proof}
present and analyze \TrackedFOAM{}, respectively.
The appendices contain additional technical results and proofs.
% {\color{red}\paragraph{Organization.}
% Sections~\ref{sec:preliminaries} and \ref{subsec:main-results} give the
% model and results. \textcolor{purple}{Section~\ref{sec:deterministic-extension}
% gives the common extension argument for the lower bounds.}
% Sections~\ref{sec:construction} and
% \ref{gs:sec:construction} construct the lower-bound instances,
% verified in Sections~\ref{sec:verification} and
% \ref{gs:sec:verification}.
% Sections~\ref{sec:upper} and~\ref{sec:upper-proof} give and analyze
% \TrackedFOAM{}; the appendices contain the technical ingredients.} 
\endgroup

\section{Preliminaries}
\label{sec:preliminaries}
To study Problem~\eqref{eq:prob}, we introduce the following class of NC-C functions.
\begin{definition}[NC-C function class]
\label{def:function-class}
 Given $\ell,D_{\cY},\Delta>0$, let
$\cF(\ell,D_{\cY},\Delta)$ be the union, over all $m,n\in\N_0$
and all closed convex sets
$\cX\subseteq\R^m$ and $\cY\subseteq\R^n$ satisfying
\(
  (\bz_m,\bz_n)\in\cX\times\cY
  \mbox{ and }
  \diam(\cY)\le D_{\cY},
\)
of the class of functions $f:\cX\times\cY\to\R$ satisfying the
following properties. For each
 such function, define $\Phi(\x):=\max_{\y\in \cY}f(\x;\y)+\iota_{\cX}(\x)$.
\begin{enumerate}[(i)]
    \item ($\ell$-smoothness) $f$ is jointly $\ell$-smooth on
    $\cX\times\cY$, i.e., for all
    $(\x_1,\y_1),(\x_2,\y_2)\in \cX\times\cY$,
    \[
        \norm*{\nabla f(\x_1;\y_1)-\nabla f(\x_2;\y_2)}
        \le \ell\norm*{(\x_1,\y_1)-(\x_2,\y_2)}.
    \]
    \item (Dual concavity) For every fixed $\x\in \cX$,
    $f(\x;\cdot)$ is concave on $\cY$.
    \item (Initial gap) The initialization satisfies $
        \Phi(\bz)-\inf_{\x\in \cX}\Phi(\x)\le \Delta.$
\end{enumerate}
\end{definition} 
% We consider two standard notions of stationarity for
% Problem~\eqref{eq:prob}: optimization stationarity (OS)
% and game stationarity (GS); see, e.g.,
% \citep[Definition~18]{lin2020near} and
% \citep[Definition~2]{li2025nonsmooth}.
% OS concerns stationarity of the value function $\Phi$.
% Since $f(\cdot;\y)+\iota_{\cX}(\cdot)$ is
% $\ell$-weakly convex for every $\y\in\cY$,
% so is $\Phi$.
% Consequently, the Moreau envelope of $\Phi$
% with parameter $1/(2\ell)$ is continuously
% differentiable, and OS is measured by the norm
% of its gradient.
% In contrast, GS measures the first-order
% optimality residuals of the two players,
% each with the other variable held fixed.
% We formalize these notions below.\Siyu{Does the explanation overlap with this in intro? They are almost talking about the same thing. }

We consider two standard notions of stationarity for
Problem~\eqref{eq:prob}: optimization stationarity (OS)
and game stationarity (GS); see, e.g.,
\citep[Definition~18]{lin2020near} and
\citep[Definition~2]{li2025nonsmooth}. See also \citet[Section~4.3]{lin2020gradient}
for a comparison of related stationarity notions.
To formulate the OS notion, note that
$f(\cdot;\y)+\iota_{\cX}(\cdot)$ is
$\ell$-weakly convex for every $\y\in\cY$,
and hence so is $\Phi$.
Consequently, its Moreau envelope with parameter
$1/(2\ell)$ is continuously differentiable. 
We give the precise definitions below.
\begin{definition}[Stationary points]
\label{def:stationary}
Let $\epsilon>0$ be given. 
\begin{enumerate}[(i)]
    \item A point $\x\in\cX$ is an
$\epsilon$-optimization-stationary point ($\epsilon$-OS) of
Problem~\eqref{eq:prob} if
\[
  \norm*{\nabla\Phi_{\frac{1}{2\ell}}(\x)}\le\epsilon,
\]
where $
  \Phi_{\frac{1}{2\ell}}(\x)
  :=
  \min_{\z\in\R^m}
  \{
    \Phi(\z)+\ell\norm{\x-\z}^2
  \}
$
is the Moreau envelope of $\Phi$ with parameter
$1/(2\ell)$.
\item The pair $(\x,\y)\in \cX \times \cY$ is an $\epsilon$-game-stationary point ($\epsilon$-GS) of Problem \eqref{eq:prob} if 
			\[
\dist(\bz,\nabla_{\x}f(\x;\y)+\partial\iota_{\cX}(\x))\leq \epsilon \quad \text{and} \quad \dist(\bz,-\nabla_{\y}f(\x;\y)+\partial\iota_{\cY}(\y))\leq \epsilon.
			\]
\end{enumerate}
\end{definition}

% For every fixed $\y\in\cY$, the function $f(\cdot;\y)+\iota_{\cX}(\cdot)$ is
% $\ell$-weakly convex. Consequently, $\Phi$ is $\ell$-weakly convex. The following
% standard identities show that Definition~\ref{def:stationary} is
% well defined.

% \begin{lemma}[Moreau envelope identities; c.f. Lemma~2.2, \citep{davis2019stochastic}]
% \label{lem:moreau-identities}
% Let $\Phi:\R^m\to\R\cup\{+\infty\}$ be proper, lower semicontinuous,
% and $\ell$-weakly convex.  For every $\x\in\R^m$, define
% \[
%   \prox_{\frac{\Phi}{2\ell}}(\x)
%   :=
%   \argminop_{\z\in\R^m}
%   \left\{
%     \Phi(\z)+\ell\norm*{\x-\z}^2
%   \right\}.
% \]
% Then $\uvec$ is unique, the envelope
% $\Phi_{\frac{1}{2\ell}}$ is continuously differentiable, and
% \[
%   \nabla\Phi_{\frac{1}{2\ell}}(\x)
%   =
%   2\ell(\x-\uvec)
%   \in\nabla\Phi(\uvec).
% \]
% \end{lemma}

For smooth minimization, a first-order oracle returns the function
value and gradient at each query. A deterministic algorithm chooses
each new query from the previous oracle replies, and oracle complexity
measures the number of queries needed to reach a prescribed
stationarity tolerance in the worst case.

We consider the saddle-point counterpart of this framework for
Problem~\eqref{eq:prob}, following the minimax oracle formulations in
\citet{li2021lower,pan2026lower}. In our setting, the oracle returns
the function value and both partial gradients, and complexity is measured
under OS or GS. We formalize these definitions below.
% We next introduce the first-order saddle oracle (cf. \citet[Definition 2.3]{pan2026lower}).
\begin{definition}[First-order saddle oracle]
\label{def:first-order-oracle}
For a differentiable function $f:\cX\times\cY\to\R$, the first-order
saddle oracle at $(\x,\y)\in\cX\times\cY$ returns
\[
        \mathbb O_f(\x;\y)
        :=
        \bigl(
        f(\x;\y),
        \nabla_{\x} f(\x;\y),
        \nabla_{\y} f(\x;\y)
        \bigr).
\]
\end{definition}
For a first-order algorithm $\sA$ and a differentiable saddle function
$f:\cX\times\cY\to\R$, let $\sA_{\x}^{(t)}[f]\in\cX$ and
$\sA_{\y}^{(t)}[f]\in\cY$ denote, respectively, the primal and dual
components of its $t$-th feasible oracle query.
% {
% We consider deterministic first-order algorithms that access $f$
% through $\mathbb O_f$, and denote their class by $\cA_{\det}$
% (cf. \citet[Definition 2.5]{pan2026lower}).
% }
\begin{definition}[Deterministic first-order algorithm]
\label{def:deterministic-algorithm}
A first-order algorithm $\sA$ is deterministic if there exists a sequence
of maps $\{\Gamma_t\}_{t\geq1}$ such that, for every such $f$, it
starts from $(\sA_{\x}^{(0)}[f],\sA_{\y}^{(0)}[f])=\bz$ and its query
sequence satisfies
\[
 \bigl(\sA_{\x}^{(t)}[f],\sA_{\y}^{(t)}[f]\bigr)
 =
 \Gamma_{t}
 \left(
   \Bigl(
       \mathbb O_f\bigl(\sA_{\x}^{(s)}[f];\sA_{\y}^{(s)}[f]\bigr)
   \Bigr)_{s<t}
 \right),\qquad \forall t\geq1.
\]
We denote the class of such algorithms by $\cA_{\det}$.
\end{definition}

We measure the number of oracle queries in the worst case as follows.
\begin{definition}[Oracle Complexity]
\label{def:complexity-algorithm-class}
Fix $\epsilon>0$. Let $\cA\subseteq\cA_{\det}$ be any nonempty subclass of deterministic first-order algorithms, and let $\cF$ be a nonempty class of such saddle functions. The OS and GS oracle complexities of $\cA$ over $\cF$ are defined respectively by
\[
\begin{aligned}
        \cT^{\OS}_\epsilon(\cA,\cF)
        :=\ &
        \inf_{\sA\in\cA}
        \sup_{f\in\cF}
         \inf
        \left\{
        t\in\N_0\;\middle|\;
        \sA_{\x}^{(t)}[f]\;
        \text{is an }\epsilon\text{-OS for }f
        \right\},\\
        \cT^{\GS}_\epsilon(\cA,\cF)
        :=\ &
        \inf_{\sA\in\cA}
        \sup_{f\in\cF}
         \inf
        \left\{
        t\in\N_0\;\middle|\;
        (\sA_{\x}^{(t)}[f],\sA_{\y}^{(t)}[f])\;
        \text{is an }\epsilon\text{-GS for }f
        \right\}.
\end{aligned}
\]
\end{definition}

We next introduce the zero-chain property and rotation notation
used in our lower-bound analysis.
We recall the standard zero-chain definition
\citep[Definition~3]{carmon2019lower}
and its formulations for minimax problems
[\citealp[Definition~6]{li2021lower};
 \citealp[Definition~2.7]{pan2026lower}].
\begin{definition}[First-order zero-chain]
\label{def:zero-chain}
Let $d\in\N_0$, let $h:\mathcal Z\subseteq\R^d\to\R$
be differentiable, and let $\ord=(\pi_1,\ldots,\pi_d)$
be a permutation of $[d]$.
We call $h$ a first-order zero-chain with respect to $\ord$
if, for every $i\in[d]$ and every $\z\in\mathcal Z$,
\[
    \supp_{\ord}(\z)\subseteq[i-1]
    \quad\Longrightarrow\quad
    \supp_{\ord}(\nabla h(\z))\subseteq[i].
\]
\end{definition}
We also introduce notation for orthogonal rotations of the
primal and dual variables. Let $\cX\subseteq\R^m$ and $\cY\subseteq\R^n$, with $m,n\in\N_0$.
For a smooth function $f:\cX\times\cY\to\R$ and matrices
$\bU\in\operatorname O(m)$ and $\bV\in\operatorname O(n)$
satisfying $\bU\cX=\cX$ and $\bV\cY=\cY$, define the rotated function
$f_{\bU,\bV}:\cX\times\cY\to\R$ by
\begin{equation*}
  f_{\bU,\bV}(\x;\y):=f(\bU^\top\x;\bV^\top\y).
\end{equation*}
Let
$\Phi_{\bU,\bV}(\x):=\max_{\y\in\cY}f_{\bU,\bV}(\x;\y)+\iota_{\cX}(\x)$
denote its value function.

\section{Main Results}
\label{subsec:main-results}
We establish matching lower and upper bounds on the deterministic
first-order oracle complexity under OS and GS across all parameter
regimes.
We first state the complexity bounds and then present two direct
extension lemmas used in the lower bound analysis. 
\subsection{Lower bound results}
We state the deterministic first-order lower bounds for achieving both $\epsilon$-OS and $\epsilon$-GS points. 
\begingroup

\begin{theorem}[Deterministic OS lower bound]
\label{thm:deterministic}
There is a numerical constant $c_{\mathrm{OS}}>0$ such that, for every
$\ell,D_{\cY},\Delta>0$ and every
$0<\epsilon<\sqrt{{8\ell\Delta}/{3}}$, one has
\[
  \cT_{\epsilon}^{\OS}\bigl(
      \cA_{\rm det},\cF(\ell,D_{\cY},\Delta)
  \bigr)
  \ge c_{\mathrm{OS}}\frac{\ell\Delta}{\epsilon^2}
       \left(1+\frac{\ell D_{\cY}}{\epsilon}\right).
\]
\end{theorem}
\begin{remark}[Zero-query regime for OS]
\label{lem:os-origin-threshold}
For every $f\in\cF(\ell,D_{\cY},\Delta)$, the initial gap bound gives
$\|\nabla\Phi_{\frac{1}{2\ell}}(\bz)\|^2
  \le  8\ell\Delta/3.$
Indeed, let $\uvec:=\prox_{\frac{\Phi}{2\ell}}(\bz)$.
The function $\Phi(\x)+\ell\norm{\x}^2$ is
$\ell$-strongly convex and minimized at $\uvec$, so
\[
  \Delta
  \ge \Phi(\bz)-\Phi(\uvec)
  \ge \frac{3\ell}{2}\norm{\uvec}^2
  = \frac{3}{8\ell}
    \norm*{\nabla\Phi_{\frac{1}{2\ell}}(\bz)}^2,
\]
where the equality uses
$\nabla\Phi_{\frac{1}{2\ell}}(\bz)=-2\ell\uvec$.
Consequently, whenever $\epsilon\ge\sqrt{8\ell\Delta/3}$,
returning the origin requires no oracle queries and guarantees
an $\epsilon$-OS point. Thus,
$
  \cT_\epsilon^{\OS}
  \bigl(\cA_{\rm det},\cF(\ell,D_{\cY},\Delta)\bigr)=0.
$ 
\end{remark}
\endgroup

\begingroup

\begin{theorem}[Deterministic GS lower bound]
\label{gs:thm:deterministic}
There is a numerical constant $c_{\rm GS}>0$ such that, for every
$\ell,D_{\cY},\Delta,\epsilon>0$,
\begin{equation}
\begin{aligned}
 \cT^{\GS}_\epsilon\bigl(\cA_{\rm det},\cF(\ell,D_{\cY},\Delta)\bigr)\ge\ & c_{\rm GS}\sqrt{
 \left(1+\frac{\ell\Delta}{\epsilon^2}\right)
 \left(1+\frac{\ell D_{\cY}}{\epsilon}\right)
 \left(1+\frac{\ell\Delta}{\epsilon^2}
         +\frac{\ell D_{\cY}}{\epsilon}\right)}.
\end{aligned}
\label{gs:eq:uniform-lower}
\end{equation}
\end{theorem}
\endgroup

% \begin{remark}[Arbitrary deterministic output]
% \label{rem:arbitrary-output}
% Theorems~\ref{thm:deterministic} and~\ref{gs:thm:deterministic}
% also give lower bounds for deterministic outputs computed from the
% final transcript but not previously queried, up to an additive loss
% of one oracle call.
% For OS, append the output primal point, paired with the feasible
% dual origin, as the next query.
% For GS, append the output primal-dual pair.
% The resulting algorithm makes at most one additional oracle call,
% so the lower bounds for queried points apply.
% \end{remark}

\subsection{Matching upper bounds}
% \label{subsec:main-upper} 
The following theorems give our matching upper bounds.
The algorithm \TrackedFOAM{} is developed in
Section~\ref{sec:upper}.
\begin{theorem}[OS oracle complexity of \TrackedFOAM{}]
\label{thm:upper}
Let $\ell,D_{\cY},\Delta>0$ and
$f\in\cF(\ell,D_{\cY},\Delta)$.
For every $\epsilon>0$, an $\epsilon$-OS point can be obtained using
\begin{equation}
\begin{cases}
  0,
  & \epsilon\ge\sqrt{8\ell\Delta/3},\\[2mm]
  \displaystyle
  \cO\!\left(
    \left(1+\frac{\ell\Delta}{\epsilon^2}\right)
    \left(1+\frac{\ell D_{\cY}}{\epsilon}\right)
  \right),
  & 0<\epsilon<\sqrt{8\ell\Delta/3},
\end{cases}
\label{eq:uniform-rate}
\end{equation}
first-order saddle-oracle calls.
In the first case, we return the origin.
In the second, we run
Algorithm~\ref{alg:tracked-foam-overview},
initialized at $\z^0=\bz$.
\end{theorem}
% \begin{theorem}[OS oracle complexity of \TrackedFOAM{}]
% \label{thm:upper}
% Let $\ell,D_{\cY},\Delta>0$ and
% $f\in\cF(\ell,D_{\cY},\Delta)$.
% For any $\epsilon>0$,
% Algorithm~\ref{alg:tracked-foam-overview}, initialized at $\z^0=\bz$,
% returns an $\epsilon$-OS point using
% \begin{equation}
%   \cO\left(
%     \left(1+\frac{\ell\Delta}{\epsilon^2}\right)
%     \left(1+\frac{\ell D_{\cY}}{\epsilon}\right)
%   \right)
%   \label{eq:uniform-rate}
% \end{equation}
% first-order saddle-oracle calls.
% % Every saddle-oracle query is feasible.
% \end{theorem}
\begin{theorem}[GS oracle complexity of \TrackedFOAM{}]
\label{thm:upper-gs}
Let $\ell,D_{\cY},\Delta>0$ and
$f\in\cF(\ell,D_{\cY},\Delta)$.
For every $\epsilon>0$, the GS variant of
Algorithm~\ref{alg:tracked-foam-overview} described in
Section~\ref{sec:gs-output}, initialized at $\z^0=\bz$,
returns an $\epsilon$-GS point using
\[
 \cO\!\left(
 \sqrt{
 \left(1+\frac{\ell\Delta}{\epsilon^2}\right)
 \left(1+\frac{\ell D_{\cY}}{\epsilon}\right)
 \left(1+\frac{\ell\Delta}{\epsilon^2}
         +\frac{\ell D_{\cY}}{\epsilon}\right)}
 \right)
\]
first-order saddle-oracle calls.
\end{theorem}
The OS upper bound matches Theorem~\ref{thm:deterministic}
because $1+\ell\Delta/\epsilon^2\asymp\ell\Delta/\epsilon^2$
whenever $\epsilon<\sqrt{8\ell\Delta/3}$.
The GS upper and lower bounds match for all positive parameter
values.

Figure~\ref{fig:os-gs-complexity-regimes}  summarizes the OS and GS complexity bounds in terms of $\epsilon^2/(\ell\Delta)$
(horizontal axis) and $\epsilon/(\ell D_{\cY})$ (vertical axis).
Each region is labeled with the corresponding oracle complexity.
For OS, the region $\epsilon^2/(\ell\Delta)\ge 8/3$ requires
no oracle queries.
In the GS panel, the diagonal $\epsilon=\Delta/D_{\cY}$
separates the two complexity regimes in the lower-left region.
\begin{figure}[htbp]
\centering
\begin{adjustbox}{max width=0.78\linewidth}
\begin{tikzpicture}[x=1cm,y=1cm,>=Stealth,font=\scriptsize,
 every node/.style={inner sep=1pt}]
 % Horizontal units: 2.08 cm for OS and 3.4 cm for GS.
 % In both panels, the vertical ratio 1 is at y=4.2.
 \pgfmathsetmacro{\oscut}{2.08*8/3}
 \begin{scope}
 \node[font=\small\bfseries] at (3.3,7.05) {OS};
 \draw[dashed] (0,4.2) -- (\oscut,4.2);
 \draw[dashed] (\oscut,0) -- (\oscut,6.4);
 \draw[->] (0,0) -- (6.85,0);
 \node[anchor=east] at (6.85,-0.7)
   {$\displaystyle\frac{\epsilon^2}{\ell\Delta}$};
 \draw[->] (0,0) -- (0,6.65)
   node[above=4pt] {$\displaystyle\frac{\epsilon}{\ell D_{\cY}}$};
 \node[below left=3pt] at (0,0) {$0$};
 \draw (-0.06,4.2) -- (0.06,4.2);
 \node[left=4pt] at (0,4.2) {$1$};
 \draw (\oscut,-0.08) -- (\oscut,0.08);
 \node[below=4pt] at (\oscut,0) {$\frac83$};
 \node at (2.75,2.1)
   {$\displaystyle\Theta\!\left(\frac{\ell^2D_{\cY}\Delta}{\epsilon^3}\right)$};
 \node at (2.75,5.3)
   {$\displaystyle\Theta\!\left(\frac{\ell\Delta}{\epsilon^2}\right)$};
 \node at (6.12,3.2) {$0$};
 \end{scope}

 \begin{scope}[xshift=8cm]
 \node[font=\small\bfseries] at (3.3,7.05) {GS};
 \draw[dashed] (3.4,0) -- (3.4,6.4);
 \draw[dashed] (0,4.2) -- (6.6,4.2);
 % The two ratios agree on the diagonal segment.
 \draw[dashed] (0,0) -- (3.4,4.2);
 \node[anchor=west] (gsboundary) at (3.8,3.2)
   {$\displaystyle\epsilon=\frac{\Delta}{D_{\cY}}$};
 \draw[->,thin] (gsboundary.west) -- (2.8,{4.2*2.8/3.4});
 \draw[->] (0,0) -- (6.85,0);
 \node[anchor=east] at (6.85,-0.7)
   {$\displaystyle\frac{\epsilon^2}{\ell\Delta}$};
 \draw[->] (0,0) -- (0,6.65)
   node[above=4pt] {$\displaystyle\frac{\epsilon}{\ell D_{\cY}}$};
 \node[below left=3pt] at (0,0) {$0$};
 \draw (3.4,-0.06) -- (3.4,0.06);
 \node[below=4pt] at (3.4,0) {$1$};
 \draw (-0.06,4.2) -- (0.06,4.2);
 \node[left=4pt] at (0,4.2) {$1$};
 \node at (1.25,3.65)
   {$\displaystyle\Theta\!\left(\frac{\ell^{3/2}D_{\cY}^{1/2}\Delta}{\epsilon^{5/2}}\right)$};
 \node at (2.1,1.05)
   {$\displaystyle\Theta\!\left(\frac{\ell^{3/2} D_{\cY}\Delta^{1/2}}{\epsilon^2}\right)$};
 \node at (1.7,5.3)
   {$\displaystyle\Theta\!\left(\frac{\ell\Delta}{\epsilon^2}\right)$};
 \node at (5.05,1.65)
   {$\displaystyle\Theta\!\left(\frac{\ell D_{\cY}}{\epsilon}\right)$};
 \node at (5.05,5.3) {$\Theta(1)$};
 \end{scope}
\end{tikzpicture}
\end{adjustbox}
\caption{Optimal deterministic first-order oracle complexities for
OS (left) and GS (right) in different parameter regimes.}
\label{fig:os-gs-complexity-regimes}
\end{figure}
\subsection{Direct constructions for deterministic lower bounds}
\label{subsec:deterministic-extensions}
We now state two lemmas for constructing hard instances
against arbitrary deterministic first-order methods,
using the zero-chain property and the corresponding
stationarity obstruction.
These lemmas may be of independent interest.
\begin{lemma}[Direct construction for OS]
\label{lem:deterministic-construction-os}
Let $\ell,D_{\cY},\Delta,\epsilon>0$, $m,n\in\N_0$,
and $L\in[m+n]$.
Let $f\in\cF(\ell,D_{\cY},\Delta)$ be defined on
$\R^m\times\B_{D_{\cY}}^n$, and let $\ord$ be a coordinate
ordering of the joint variable $(\x,\y)$.
Suppose that:
\begin{enumerate}[(i)]
    \item $f$ is a first-order zero-chain with respect to $\ord$;
    \item every $(\x,\y)\in\R^m\times\B_{D_{\cY}}^n$ satisfies
    \[
      \supp_{\ord}(\x,\y)\subseteq[L-1]
      \quad\Longrightarrow\quad
       \x \text{ is not an }\epsilon\text{-OS point of }f.
    \]
\end{enumerate}
Define
$\widehat f_{\OS}:
\R^{m+L}\times\B_{D_{\cY}}^{n+L}\to\R$
by
$
  \widehat f_{\OS}(\x,\bm\xi;\y,\bm\eta):=f(\x;\y),
$ 
where $\bm\xi,\bm\eta\in\R^L$ are auxiliary variables.
Then, for every $\sA\in\cA_{\rm det}$, there exist
$\bU\in\operatorname O(m+L)$ and
$\bV\in\operatorname O(n+L)$ such that
\[
  (\widehat f_{\OS})_{\bU,\bV}\in\cF(\ell,D_{\cY},\Delta)
  \qquad\text{and}\qquad
  \cT^{\OS}_{\epsilon}\bigl(\{\sA\},\{(\widehat f_{\OS})_{\bU,\bV}\}\bigr)\ge L.
\]
\end{lemma}
\begin{lemma}[Direct construction for GS]
\label{lem:deterministic-construction-gs}
Let $\ell,D_{\cY},\Delta,\epsilon>0$, $m,n,k\in\N_0$,
and $L\in[m+n]$.
Let $f\in\cF(\ell,D_{\cY},\Delta)$ be defined on
$\R^m\times[(\R_+^k\times\R^n)\cap\B_{D_{\cY}}^{k+n}]$,
and let $\ord$ be a coordinate ordering of the joint variable
$(\x,\w)$.
Suppose that:
\begin{enumerate}[(i)]
    \item for every feasible $\bm c$, the function
    $f(\cdot;\bm c,\cdot)$ is a first-order zero-chain
    with respect to $\ord$;
    \item every feasible pair $(\x,\y)$, with
    $\y=(\bm c,\w)$, satisfies
    \[
      \supp_{\ord}(\x,\w)\subseteq[L-1]
      \quad\Longrightarrow\quad
      (\x,\y)\text{ is not a }2\epsilon\text{-GS point of }f.
    \]
\end{enumerate}
Define $\widehat f_{\GS}$ on
$\R^{m+L}\times
[(\R_+^k\times\R^{n+L})\cap\B_{D_{\cY}}^{k+n+L}]$
by
$
  \widehat f_{\GS}(\x,\bm\xi;\bm c,\w,\bm\eta)
  :=f(\x;\bm c,\w)-\frac{\ell}{2}\norm{\bm\eta}^2,
$ 
where $\bm\xi,\bm\eta\in\R^L$ are auxiliary variables. Then, for every $\sA\in\cA_{\rm det}$, there exist
$\bU\in\operatorname O(m+L)$ and
$\bV\in\operatorname O(n+L)$ such that
\[
   (\widehat f_{\GS})_{\bU,\operatorname{diag}(\bI_k,\bV)}
  \in\cF(\ell,D_{\cY},\Delta)
  \qquad\text{and}\qquad
  \cT^{\GS}_{\epsilon}\!\left(
    \{\sA\},
    \left\{
      (\widehat f_{\GS})_{\bU,\operatorname{diag}(\bI_k,\bV)}
    \right\}
  \right)\ge L.
\]
\end{lemma}
Both lemmas are proved in
Appendix~\ref{app:deterministic-extension-proof}.
% We use Lemma~\ref{lem:deterministic-construction-os} for $f_2$ and $f_3$
% in Section~\ref{sec:construction}, and
% Lemma~\ref{lem:deterministic-construction-gs} for $g_2$ and $g_3$
% in Section~\ref{gs:sec:construction}.

\section{The Construction of OS Hard Instances}
\label{sec:construction}

\subsection{Overview of the construction}
\label{subsec:os-construction-overview}

We construct three hard instances $f_1,f_2,f_3$ to prove
Theorem~\ref{thm:deterministic}. The lower bound has the form
\[
 \Omega\left(
   \frac{\ell\Delta}{\epsilon^2}
   \max\left\{1,\frac{\ell D_{\cY}}{\epsilon}\right\}
 \right).
\]
For $c_0\sqrt{\ell\Delta}<\epsilon<\sqrt{8\ell\Delta/3}$,
where $c_0\le1/2$ is the numerical constant in
Theorem~\ref{thm:zr-lower}, the factor $\ell\Delta/\epsilon^2$
is of constant order. The scalar quadratic $f_1$ gives a constant
lower bound when $\ell D_{\cY}/\epsilon$ is small.
For larger values of this ratio, $f_2$ uses a dual chain
on an unconstrained primal domain to obtain the required
diameter dependence.

For $\epsilon\le c_0\sqrt{\ell\Delta}$, the multistage instance
$f_3$ combines inner dual chains with outer state couplings to
obtain both factors.

For $\epsilon\ge\sqrt{8\ell\Delta/3}$, the origin is already an
$\epsilon$-OS point by Remark~\ref{lem:os-origin-threshold}.
Table~\ref{tab:os-regimes} summarizes these regimes.
Here $d_0>0$ is the numerical constant defined in the
proof of Theorem~\ref{thm:zr-lower}.

\begin{table}[h]
\centering
\small

\caption{Regimes, deterministic OS lower bounds, and hard instances.}
\label{tab:os-regimes}
\renewcommand{\arraystretch}{1.4}
\setlength{\tabcolsep}{4pt}
\begin{adjustbox}{max width=\textwidth}
\begin{tabular}{@{}lll@{}}
\toprule
Regimes & OS lower bound & Hard instances\\
\midrule
$c_0\sqrt{\ell\Delta}<\epsilon<\sqrt{\frac{8\ell\Delta}{3}},\quad\epsilon>\frac{\ell D_{\cY}}{16}$
& $\Omega(1)$ (Lemma~\ref{lem:os-endpoint})
& $f_1$ (Eq.~\eqref{eq:os-scalar-endpoint-instance})\\
$c_0\sqrt{\ell\Delta}<\epsilon<\sqrt{\frac{8\ell\Delta}{3}},\quad\epsilon\le\frac{\ell D_{\cY}}{16}$
& $\Omega\left(\frac{\ell D_{\cY}}{\epsilon}\right)$ (Lemma~\ref{lem:os-endpoint})
& $f_2$ (Eq.~\eqref{eq:os-endpoint-core})\\
$0<\epsilon\le c_0\sqrt{\ell\Delta},\quad\epsilon>d_0\ell D_{\cY}$
& $\Omega\left(\frac{\ell\Delta}{\epsilon^2}\right)$ (Theorem~\ref{thm:zr-lower})
& $f_3$, $N=0$ (Eq.~\eqref{eq:os-scaled-instance})\\
$0<\epsilon\le\min\left\{c_0\sqrt{\ell\Delta},d_0\ell D_{\cY}\right\}$
& $\Omega\left(\frac{\ell^2D_{\cY}\Delta}{\epsilon^3}\right)$ (Theorem~\ref{thm:zr-lower})
& $f_3$, $N\ge1$ (Eq.~\eqref{eq:os-scaled-instance})\\
$\epsilon\ge\sqrt{\frac{8\ell\Delta}{3}}$
& $\bz$ is $\epsilon$-OS (Remark~\ref{lem:os-origin-threshold})
& \textemdash\\
\bottomrule
\end{tabular}
\end{adjustbox}
\end{table}

For the constant lower bound, fix $\ell,D_{\cY},\Delta>0$ and
$0<\epsilon<\sqrt{8\ell\Delta/3}$, and set
\begin{equation}
 R:=\min\left\{\frac{\epsilon}{\ell},
                    \sqrt{\frac{2\Delta}{3\ell}}\right\}.
 \label{eq:os-endpoint-diameter}
\end{equation}
When $\ell D_{\cY}/\epsilon<16$, define
\begin{equation}
 f_1(a;\bz_0):=-\ell Ra-\frac{\ell}{2}a^2,
 \qquad
  (a,\bz_0)\in[-R,R]\times\R^0.
 \label{eq:os-scalar-endpoint-instance}
\end{equation}
Let $\Phi_1$ denote its value function. Its initial gap is
$3\ell R^2/2\le\Delta$, and its Moreau gradient at the
origin has norm $2\ell R>\epsilon$.
Thus $f_1\in\cF(\ell,D_{\cY},\Delta)$ gives the constant lower bound.
\vspace{-2mm}
\subsection{Construction of \texorpdfstring{$f_2$}{f2}}
% \subsection{Construction of $f_2$}
\label{subsec:os-endpoint-construction}
For the same $\ell,D_{\cY},\Delta,\epsilon$ and $R$,
assume $\ell D_{\cY}/\epsilon\ge16$.
The instance $f_2$ uses an unconstrained primal domain and a dual
chain. The parameter $R$ sets its scale. Put
\begin{equation}
 L_2:=\left\lfloor\frac{D_{\cY}}{16R}\right\rfloor,
 \qquad \delta:=R-\frac{\epsilon}{2\ell}.
 \label{eq:os-endpoint-parameters}
\end{equation}
Then $L_2\ge1$, $0<\delta\le R/2$, and
$3\ell R^2/2\le\Delta$.
Since $\epsilon/(2\ell)<R\le\epsilon/\ell$, we also have
$L_2=\Theta(\ell D_{\cY}/\epsilon)$.
Define
\begin{equation}
 \begin{aligned}
 \Psi(t):=\ &
 \begin{cases}
 \frac{1}{2}(t+\delta)^2\exp\!\left(-\frac{2\delta}{t+\delta}\right),
     &t>-\delta,\\
 0,&t\le-\delta,
 \end{cases}\qquad Q(t):=\Psi(t)-\Psi(t-R).
 \end{aligned}
 \label{eq:os-endpoint-scalars}
\end{equation}
These functions are infinitely differentiable, and
$0\le Q''\le1$ and $0\le Q'\le R$.

For $a\in\R$ and $\u=(u_1,\ldots,u_{L_2})\in\B_{D_{\cY}}^{L_2}$,
write $b:=u_{L_2}/\sqrt{L_2}$ and define
\begin{equation}
 f_2(a;\u):=\ell\left[
 (R-a)b-Q(b-a-R)
 -\frac{1}{8}\sum_{i=1}^{L_2-1}
     \left(u_i-u_{i+1}-\frac{4R}{\sqrt{L_2}}\right)^2
 \right].
 \label{eq:os-endpoint-core}
\end{equation}
Define its value function by
\[
 \Phi_2(a):=\max_{\u\in\B_{D_{\cY}}^{L_2}}f_2(a;\u).
\]

\begin{lemma}[Properties of $f_2$]
\label{lem:os-endpoint-properties}
Let $f_2$ be defined by~\eqref{eq:os-endpoint-core}.
\begin{enumerate}[label=\textup{(\roman*)},leftmargin=*]
\item $f_2\in\cF(\ell,D_{\cY},\Delta)$.
\item The function $f_2$ is a first-order zero-chain
in the order $(u_1,\ldots,u_{L_2},a)$.
\item The value function satisfies
\begin{equation}
 \norm{\nabla[\Phi_2]_{1/(2\ell)}(0)}>\epsilon.
 \label{eq:os-endpoint-obstruction}
\end{equation}
\end{enumerate}
\end{lemma}

Combining the bounds from $f_1$ and $f_2$ gives
the following result.

\begin{lemma}[OS lower bound from $f_1$ and $f_2$]
\label{lem:os-endpoint}
For every $\ell,D_{\cY},\Delta>0$ and
$0<\epsilon<\sqrt{8\ell\Delta/3}$, one has
\[
  \cT_\epsilon^{\OS}(\cA_{\rm det},\cF(\ell,D_{\cY},\Delta))
  \ge\frac{1}{100}\left(1+\frac{\ell D_{\cY}}{\epsilon}\right).
\]
\end{lemma}
Both lemmas are proved in Appendix~\ref{app:os-endpoint}.

\subsection{Construction of \texorpdfstring{$f_3$}{f3}}
\label{subsec:os-multistage-construction} 
The unscaled function $\bar f$ underlying $f_3$ combines
$M$ stages, each with three primal variables $(a_i,b_i,s_i)$
and $N$ dual variables $\y^{(i)}$. The inner chain connects
$a_i$ to $b_i$, and $s_i$ links consecutive stages.
The variables follow the order
\[
 s_{i-1}\longrightarrow a_i\longrightarrow y_1^{(i)}
 \longrightarrow\cdots\longrightarrow y_N^{(i)}
 \longrightarrow b_i\longrightarrow s_i.
\]
For $N=0$, the dual coordinates are omitted.

\subsubsection{The inner chain}
\label{subsec:inner-block}
The inner chain has primal variables
$a,b\in\R$ and a dual variable $\w\in\R^N$, where $N\in\N_0$.
The construction adapts the regularized quadratic chain of
\citet[Section~5.1]{li2021lower}, which builds on the classical
construction of \citet{nesterov2018lectures}.
For $N\ge2$, define
\begin{equation}
\label{eq:A}
\bA_N:=\left(
\begin{array}{ccccc}
1 & -1 &   &   &   \\
-1 & 2 & -1 &   &   \\
   & -1 & \ddots & \ddots &   \\
   &   & \ddots & 2 & -1 \\
   &   &   & -1 & 1
\end{array}
\right)\in\R^{N\times N}.
\end{equation}
For $N=1$, set $\bA_1=[0]$.
For $N\ge1$, using $\bA_N$ from~\eqref{eq:A}, set
\[
 \bM_N:=\frac{1}{N^2}\bI_N+\bA_N,
 \qquad
 \bB_N:=\bM_N^{-1}.
\]
Since $\bA_N$ is positive semidefinite,
$\bM_N\succeq N^{-2}\bI_N\succ0$, so $\bB_N$ is well defined.
Set
\[
  k_N:=(\bB_N)_{1N},
  \qquad\mbox{and}\qquad
  c_N:=\frac{(\bB_N)_{11}}{k_N}.
\]
Lemma~\ref{lem:inner-interface} shows that $k_N>0$ for every
$N\ge1$, so $c_N$ is also well defined.

For $N\ge1$, define the inner dual chain by
\[
  H(a,b;\w)
  :=-\frac{1}{2}\w^\top\bM_N\w
    +k_N^{-\frac{1}{2}}\ip*{a\ee_1-b\ee_N}{\w}
    +\frac{2-c_N}{2}(a^2+b^2).
\]
The function $H(a,b;\cdot)$ is
$N^{-2}$-strongly concave and has the unique maximizer
\[
  \w^\star(a,b)
  =k_N^{-\frac{1}{2}}\bB_N(a\ee_1-b\ee_N).
\]
For $N=0$, define the direct coupling and its maximizer by
\[
 H(a,b;\bz_0):=a^2-ab+b^2,\qquad
 \w^\star(a,b):=\bz_0.
\]
For every $N\in\N_0$, the normalization gives the same
unconstrained inner maximum:
\begin{equation}
  Q(a,b):=\max_{\w\in\R^N}H(a,b;\w)=a^2-ab+b^2.
  \label{eq:effective-link}
\end{equation}
The maximizer satisfies
\[
  \norm*{\w^\star(a,b)}\le c_{\y}N\sqrt{a^2+b^2},
\]
where $c_{\y}>0$ is a numerical constant independent of $N$.
This estimate will be used to control the dual maximizers
when the chains are combined over a bounded dual domain.
Moreover, $H$ has a joint smoothness constant independent of $N$
and is a zero-chain with coordinate ordering $\ord$ given by
$(a,b,\w)_{\ord}=(a,(w_j)_{j=1}^N,b)$.
These properties are proved in Lemma~\ref{lem:inner-interface}
in Appendix~\ref{app:inner-interface-proof}.

\subsubsection{The composite hard instance  \texorpdfstring{$\bar f$}{f}}
\label{subsec:outer-chain}
We combine the inner chains using three coupling functions
and a state regularizer.
The entrance and exit couplings connect each chain to its
adjacent states, while the state coupling and regularizer
control the state configuration when the
gradient of the value function is small.
A smooth activation function switches the couplings according
to the states, and bounded extensions of the identity ensure
global smoothness.

Define the smooth function
\begin{equation}
\begin{aligned}
p(t):=\ &
\begin{cases}
0, & t\le0,\\[2pt]
\displaystyle
\int_0^t
\frac{\exp(-1/v)}
     {\exp(-1/v)+\exp(-1/(1-v))}\,{\rm d}v,
& 0<t<1,\\[10pt]
t-\frac{1}{2}, & t\ge1,
\end{cases}\qquad q(t):=p'\!\left(\frac{5t-1}{4}\right).
\end{aligned}
\label{eq:base-ramp}
\end{equation}
The function $q$ smoothly transitions from $0$ to $1$, with
$q(t)=0$ for $t\le1/5$ and $q(t)=1$ for $t\ge1$.

Define two smooth bounded extensions of the identity by
\[
\begin{aligned}
e_{\s}(t):=\ & p(t+3)-p(t-2)-\frac{5}{2},\qquad
e_{\bm{\nu}}(t):=p(t+22)-p(t-21)-\frac{43}{2}.
\end{aligned}
\]
These satisfy $e_{\s}(t)=t$ for $|t|\le2$ and
$e_{\bm{\nu}}(t)=t$ for $|t|\le21$, and are constant for
sufficiently large positive or negative inputs.
The functions $q$, $e_{\s}$, and $e_{\bm{\nu}}$ are infinitely
differentiable, with compactly supported first derivatives.
Their required bounds are given in
Lemmas~\ref{lem:outer-gate} and~\ref{lem:identity-extensions}
in Appendices~\ref{app:outer-gate-proof}
and~\ref{app:identity-extensions-proof}.

For consecutive state variables $u,v$ and connector variables $a,b$,
define the entrance, exit, and state coupling functions by
\begin{align}
C_{\rm en}(u,a,v)
:=\ & -4q(u)\bigl(1-q(v)\bigr)e_{\bm{\nu}}(a),
\label{eq:activation-component}\\
C_{\rm ex}(u,b,v)
:=\ & -q(u)e_{\s}(v)\bigl(1-q(v)\bigr)e_{\bm{\nu}}(b),
\label{eq:writeback-component}\\
C_{\rm st}(u,v)
:=\ & 24q(v)\bigl(1-e_{\s}(u)\bigr)\bigl(1-q(u)\bigr).
\label{eq:ordering-component}
\end{align}
Set
\[
\begin{aligned}
c_R:=25+\norm*{\nabla_uC_{\rm st}}_{\infty}
 +\norm*{\nabla_vC_{\rm st}}_{\infty}+\norm*{\nabla_uC_{\rm en}}_{\infty}
 +\norm*{\nabla_vC_{\rm en}}_{\infty}+\norm*{\nabla_uC_{\rm ex}}_{\infty}
 +\norm*{\nabla_vC_{\rm ex}}_{\infty},
\end{aligned}
\]
which is a numerical constant, and define the state regularizer
\begin{equation}
R(t):=\frac{12}{5}p(-10t)
-\frac{c_R+1}{10}\bigl[p(10t-1)-p(10t-10)\bigr].
\label{eq:phase-potential}
\end{equation}
Call an integer pair $(M,N)$ admissible if
$M\ge1$ and $N\in\N_0$.
Fix an admissible pair $(M,N)$ and $D>0$.
Write the primal variable as
$\x=(\s,\bm{\nu})\in\R^{3M}$, where
$\s=(s_1,\ldots,s_M)$ and
$\bm{\nu}=(a_1,b_1,\ldots,a_M,b_M)$, and set $s_0:=1$.
Let $\bar n:=MN$ and $\cY_D:=\B_D^{\bar n}$, with dual variable
$\y=(\y^{(1)},\ldots,\y^{(M)})\in\cY_D$ and
$\y^{(i)}\in\R^N$.
Define the unscaled hard instance on
$\R^{3M}\times\cY_D$ by
\begin{equation}
\begin{aligned}
\fbar(\s,\bm{\nu};\y)
:=\ & \sum_{i=1}^{M}\Bigl[
  C_{\rm en}(s_{i-1},a_i,s_i)
   +H(a_i,b_i;\y^{(i)})
   +C_{\rm ex}(s_{i-1},b_i,s_i)+C_{\rm st}(s_{i-1},s_i)+R(s_i)
\Bigr].
\end{aligned}
\label{eq:hard-instance}
\end{equation}

The coordinate ordering $\ord_3$ is specified by
\begin{equation}
  (\x,\y)_{\ord_3}
  =\bigl(a_i,(y_j^{(i)})_{j=1}^N,b_i,s_i\bigr)_{i=1}^M.
\label{eq:joint-coordinate-order}
\end{equation}

Denote its value function by
\(
 \bar\Phi_3(\s,\bm{\nu})
 :=\max_{\y\in\cY_D}\fbar(\s,\bm{\nu};\y).
\) The following lemma collects the properties of the unscaled hard instance $\bar f$. 

\begin{lemma}[Properties of $\fbar$]
\label{lem:os-instance-properties}
There are numerical constants $\ell_0,g_0,c_\Delta,c_D>0$,
independent of $M,N,D$, such that, for every admissible pair $(M,N)$
and every $D>0$, the function $\fbar$ has the following properties.
\begin{enumerate}[label=\textup{(\roman*)},leftmargin=*]
\item $\fbar\in\cF(\ell_0,D,c_\Delta M)$.
\item $\fbar$ is a first-order zero-chain in the order
\eqref{eq:joint-coordinate-order}, with $s_M$ at position
$L_3=M(N+3)$.
\item The value function $\bar\Phi_3$ is differentiable, and whenever
$N\le c_DD$,
\[
 s_M\le\frac{1}{5}
 \quad\Longrightarrow\quad
 \norm{\nabla\bar\Phi_3(\s,\bm\nu)}\ge g_0.
\]
\end{enumerate}
\end{lemma}

For an admissible pair $(M,N)$ and $\lambda>0$, set
$D:=D_{\cY}/\lambda$ in $\bar f$ and define
\begin{equation}
  f_3(\x;\y)
  :=\frac{\ell\lambda^2}{\ell_0}
    \fbar\!\left(\frac{\x}{\lambda};
                 \frac{\y}{\lambda}\right)
  \label{eq:os-scaled-instance}
\end{equation}
on $\R^{3M}\times\B_{D_{\cY}}^{\bar n}$. Let $\Phi_3$ denote its
value function. For sufficiently small $\epsilon$,
the parameter choices in~\eqref{eq:lower-bound-parameters} satisfy
$\lambda=\Theta(\epsilon/\ell)$,
$M=\Theta(\ell\Delta/\epsilon^2)$, and
$N+3=\Theta(\max\{1,\ell D_{\cY}/\epsilon\})$.
Section~\ref{subsec:proof-zr} transfers the gradient obstruction to the
Moreau envelope and applies Lemma~\ref{lem:deterministic-construction-os}
to obtain the following bound.
\begin{theorem}[OS lower bound from $f_3$]
\label{thm:zr-lower}
There are numerical constants $c_0\in(0,1/2]$ and $c_1>0$ such that,
for every $\ell,D_{\cY},\Delta,\epsilon>0$ satisfying
$\epsilon\le c_0\sqrt{\ell\Delta}$, one has
\[
 \cT_{\epsilon}^{\OS}(\cA_{\rm det},\cF(\ell,D_{\cY},\Delta))
 \ge c_1\frac{\ell\Delta}{\epsilon^2}
       \max\left\{1,\frac{\ell D_{\cY}}{\epsilon}\right\}.
\]
\end{theorem}

Together with Lemma~\ref{lem:os-endpoint}, this result
gives Theorem~\ref{thm:deterministic} over the full nontrivial
accuracy range.

\section{The Construction of GS Hard Instances}
\label{gs:sec:construction}
\begingroup
\subsection{Overview of the construction}
\label{gs:subsec:construction-overview}
We construct three hard instances $g_1,g_2,g_3$ to prove
Theorem~\ref{gs:thm:deterministic}. Up to numerical factors, the
lower bound in that theorem is the maximum of four contributions:
\[
 \max\left\{
 1,\frac{\ell\Delta}{\epsilon^2},
 \frac{\ell D_{\cY}}{\epsilon}
       \sqrt{1+\frac{\ell\Delta}{\epsilon^2}},
 \frac{\ell^{3/2}D_{\cY}^{1/2}\Delta}{\epsilon^{5/2}}
 \right\}.
\]
The linear dual instance $g_1$ supplies the constant term. The family $g_2$ supplies the next two terms.
For $N=1$, its direct primal couplings give the cost of smooth
minimization without a restriction on the dual diameter.
For $N\ge2$, its bilinear inner chains give the additional
cost depending on the dual diameter, including the cost that
persists at zero initial gap.

The remaining contribution
$\ell^{3/2}D_{\cY}^{1/2}\Delta/\epsilon^{5/2}$ is supplied by
$g_3$ when
$\epsilon\le\min\{c_3\ell D_{\cY},\allowbreak\Delta/(c_4D_{\cY})\}$.
Outside this range, it is controlled by the bounds from $g_1$
and $g_2$, as shown in Section~\ref{gs:subsec:scaling}.
The family $g_3$ uses a dual controller and one or more
connected batches to obtain this contribution.

We use $M$ for the number of stages in each batch, $N$ for the length
of each inner chain, and $K$ for the number of batches in $g_3$.
We define $\widetilde g_2,\widetilde g_3$ and their scaled
versions $g_2,g_3$ below.
Section~\ref{gs:subsec:scaling} chooses the parameters.

Table~\ref{tab:gs-regimes} partitions all positive parameter choices
into four disjoint regimes. In each regime, the indicated instance
attains the full lower bound of Theorem~\ref{gs:thm:deterministic}.
Here, $c_3\in(0,1)$ and $c_2,c_4\ge1$ are numerical
constants specified in~\eqref{gs:eq:regime-constants}.

\begin{table}[h]
\centering
\small
\caption{Regimes, deterministic GS lower bounds, and hard instances.}
\label{tab:gs-regimes}
\renewcommand{\arraystretch}{1.4}
\setlength{\tabcolsep}{4pt}
\begin{adjustbox}{max width=\textwidth}
\begin{tabular}{@{}lll@{}}
\toprule
Regimes & GS lower bound & Hard instances\\
\midrule
$\epsilon>\min\left\{c_3\ell D_{\cY},\frac{\Delta}{c_4D_{\cY}}\right\},\quad\ell D_{\cY}<c_2\sqrt{\epsilon^2+\ell\Delta},\quad\ell\Delta<c_2\epsilon^2$ & $\Omega(1)$ (Lemma~\ref{gs:lem:fixed-controller-lower}) & $g_1$ (Eq.~\eqref{gs:eq:scaled-instance-1})\\
$\epsilon>\min\left\{c_3\ell D_{\cY},\frac{\Delta}{c_4D_{\cY}}\right\},\quad\ell D_{\cY}<c_2\sqrt{\epsilon^2+\ell\Delta},\quad\ell\Delta\ge c_2\epsilon^2$
& $\Omega\left(\frac{\ell\Delta}{\epsilon^2}\right)$
(Lemma~\ref{gs:lem:fixed-controller-lower})
& $g_2$, $N=1$ (Eq.~\eqref{gs:eq:scaled-instance-2})\\
$\epsilon>\min\left\{c_3\ell D_{\cY},\frac{\Delta}{c_4D_{\cY}}\right\},\quad\ell D_{\cY}\ge c_2\sqrt{\epsilon^2+\ell\Delta}$
& $\Omega\left(\frac{\ell D_{\cY}}{\epsilon}\sqrt{1+\frac{\ell\Delta}{\epsilon^2}}\right)$
(Lemma~\ref{gs:lem:fixed-controller-lower})
& $g_2$, $N\ge2$ (Eq.~\eqref{gs:eq:scaled-instance-2})\\
$0<\epsilon\le\min\left\{c_3\ell D_{\cY},\frac{\Delta}{c_4D_{\cY}}\right\}$
& $\Omega\left(\frac{\ell^{3/2}D_{\cY}^{1/2}\Delta}{\epsilon^{5/2}}\right)$
(Theorem~\ref{gs:thm:deterministic})
& $g_3$ (Eq.~\eqref{gs:eq:scaled-instance-3})\\
\bottomrule
\end{tabular}
\end{adjustbox}
\end{table}

For the first regime, define
\begin{equation}
 g_1(x;v):=2\epsilon v,
 \qquad (x,v)\in\R\times
 \left[-\frac{D_{\cY}}{2},\frac{D_{\cY}}{2}\right].
 \label{gs:eq:scaled-instance-1}
\end{equation}
Its value function is the constant $\epsilon D_{\cY}$, so its
initial gap is zero, while its initial dual residual is $2\epsilon$.
Thus $g_1\in\cF(\ell,D_{\cY},\Delta)$ gives the constant lower bound.

\subsection{The inner chain and outer terms}
\label{gs:subsec:inner-block}
We first define the inner chain and outer terms used to
construct $g_2$ and $g_3$.
The inner chain delays coordinate discovery, while the
outer terms connect successive stages and control the
state variables.

For $N\ge2$, define
\begin{equation}
 \bD_N:=
 \begin{pmatrix}
  \frac{1}{\sqrt N} &        &        &        & \\
  -1               & 1      &        &        & \\
                   & \ddots & \ddots &        & \\
                   &        & -1     & 1      & \\
                   &        &        & -\frac{1}{\sqrt N}
                                             & \frac{1}{\sqrt N}
 \end{pmatrix}
 \in\R^{N\times N}.
 \label{gs:eq:weighted-matrix}
\end{equation}
For $\z,\y\in\R^N$ and $t\in\R$, set
\begin{equation}
 I_N(\z,t;\y)
 :=\left\langle\bD_N^\top\z-\frac{t}{\sqrt N}\ee_N,\y\right\rangle,
 \qquad N\ge2.
 \label{gs:eq:inner-chain}
\end{equation}
For $N=1$, define
\[
 I_1(z,t;y):=\frac{1}{2}(z-t)^2,\qquad z,t,y\in\R,
\]
which is independent of $y$.
Lemma~\ref{gs:mem:matrix-bounds} shows that $I_1$ is jointly
$2$-smooth and, for each fixed $y$, a zero-chain in the order $(z,t)$.
For $N\ge2$, it shows that $I_N$ is jointly
$3$-smooth and a first-order zero-chain with coordinate
ordering $\ord$ given by
\[
 (\z,t,\y)_{\ord}=(z_1,y_1,\ldots,z_N,y_N,t).
\]
The lemma also records the derivative identities used below.

The entrance and exit couplings connect each inner chain's
first primal coordinate to the states before and after that
chain, respectively.
The state regularizer $R_1$ controls negative state values
and transitions between inactive and active states.
For the construction of $g_3$ with multiple batches, the state
coupling orders the batch terminal states, while the tail
regularizer $R_2$ separates their low and high ranges in the
analysis of small primal gradients.
We first define the smooth scalar functions used in these terms.
Fix $c_s:=64$ and, using $p$ from~\eqref{eq:base-ramp}, set
\[
 \begin{aligned}
 e_1(t):=\ & \frac{c_s}{16}\Biggl[
 15\left(
   p\!\left(\frac{t}{c_s}+4\right)
   -p\!\left(\frac{t}{c_s}-3\right)
 \right)+p\!\left(\frac{t}{c_s}+20\right)
  -p\!\left(\frac{t}{c_s}-19\right)
 \Biggr].
 \end{aligned}
\]
This function satisfies $0\le e_1\le9c_s$ and
$e_1(t)=t+9c_s/2$ for $|t|\le3c_s$, and is constant on each tail.
Set $e_0:=e_1(0)=9c_s/2$.

We also define
\[
 e_2(t):=\left[1-p'\!\left(\frac{t}{c_s}-1\right)\right]^2.
\]
The function $e_2$ smoothly decreases from one to zero and is
constant on each tail.
The required bounds for $e_1$ and $e_2$ are given in
Lemma~\ref{gs:lem:affine-extension}.

For state variables $r,s,t$ and an entrance coordinate $a$, define
the entrance, exit, and state coupling functions by
\begin{equation}
 \begin{aligned}
 \widetilde C_{\rm en}(t,a)
 :=\ & 8\left[-a+\bigl(1-e_2(3c_s-4t)\bigr)e_1(a)\right],\\
 \widetilde C_{\rm ex}(t,a)
 :=\ & 8\bigl(1-e_2(t)\bigr)[t+e_1(a-t)],\\
 \widetilde C_{\rm st}(r,s)
 :=\ & (1-e_{\s}(r))(1-q(r))q(s).
 \end{aligned}
 \label{gs:eq:entrance-coupling}
\end{equation}
Here $e_{\s}$ and $q$ are defined in Section~\ref{subsec:outer-chain}.
The state coupling satisfies $\widetilde C_{\rm st}\ge0$,
$\nabla_r\widetilde C_{\rm st}\le0$, and
$\nabla_s\widetilde C_{\rm st}\ge0$, with bounded derivatives.

Fix $c_S:=192$ and define the state regularizer by
\begin{equation}
 R_1(t):=-2c_Sc_s\left[
  p\!\left(\frac{4t}{c_s}-1\right)
  -4p\!\left(\frac{t}{c_s}-1\right)
 \right]+18c_sp\!\left(-\frac{4t+c_s}{3c_s}\right).
 \label{gs:eq:state-regularizer}
\end{equation}
Set $S_0:=9c_Sc_s$.

For $D>0$ and a sufficiently large numerical constant $R_0>0$,
define the tail regularizer by
\begin{equation}
 R_2(t):=-R_0\left[
 p(t)-p(t-D)+\frac{D^2}{25}\left(
 p\!\left(\frac{5t}{D}\right)
 -p\!\left(\frac{5t}{D}-5\right)\right)\right].
 \label{gs:batch:tail-ramp}
\end{equation}
The properties of $R_1$ and $R_2$ used below are given in
Lemma~\ref{gs:lem:regularizers} in
Appendix~\ref{gs:app:derivative-identities}.

\subsection{Construction of \texorpdfstring{$g_2$}{g2}}
\label{gs:subsec:composite-instance}
We first collect the entrance and exit couplings into a common potential.
Fix $M\ge1$, let
$\s=(s_1,\ldots,s_M)\in\R^M$ and
$\a=(a_1,\ldots,a_M)\in\R^M$ be the state and entrance vectors,
respectively, and set $s_0:=c_s/2$.
For an activation parameter $\eta_0\in[0,1]$, define
\begin{equation}
 P(\s,\a,\eta_0)\!:=
 \sum_{i=1}^M\!\left[
 \widetilde C_{\rm en}(s_{i-1},a_i)
 +\widetilde C_{\rm ex}(s_i,a_i)\right]\!+8(1-\eta_0)e_1(a_1)-36c_s\left[M-\eta_0+1-e_2(s_M)\right].
 \label{gs:batch:potential}
\end{equation}
The two terms after the sum serve different purposes:
the first lets $\eta_0$ control activation of the first stage,
while the second makes $P$ zero at the origin for every $\eta_0$.
This form will be used to connect batches in the construction
of $g_3$.

Fix $D>0$ and $N\ge1$, and fix $s_M=0$, so that
$\s=(s_1,\ldots,s_{M-1},0)$.
For each stage $i\in[M]$, take $\z^{(i)},\y^{(i)}\in\R^N$,
and set $\z=(\z^{(1)},\ldots,\z^{(M)})$ and
$\z_1:=(z_1^{(1)},\ldots,z_1^{(M)})\in\R^M$.
With primal and dual variables
\[
 \x=(s_1,\ldots,s_{M-1},\z)\in\R^{MN+M-1},\qquad
 \y=(\y^{(1)},\ldots,\y^{(M)})\in\B_D^{MN},
\]
define
\begin{equation}
 \widetilde g_2(\x;\y)
 :=P(\s,\z_1,1)
   +\sum_{i=1}^M\left[I_N(\z^{(i)},s_i;\y^{(i)})+R_1(s_i)\right]+\frac{3D}{2}.
 \label{gs:fixed:instance}
\end{equation}
For $N=1$, the function is independent of $\y$, and we also write
$\widetilde g_2(\x):=\widetilde g_2(\x;\y)$.
The direct quadratic couplings give the zero-chain ordering
\begin{equation}
 \x_{\ord}=(z_1^{(1)},s_1,z_1^{(2)},\ldots,s_{M-1},z_1^{(M)}).
 \label{gs:fixed:primal-order}
\end{equation}
For $N\ge2$, the bilinear inner chains delay discovery between
consecutive states and give the coordinate ordering $\ord$
specified by
\begin{equation}
 (\x,\y)_{\ord}=\left(
 \bigl((z_h^{(i)},y_h^{(i)})_{h=1}^N,s_i\bigr)_{i=1}^{M-1},
 (z_h^{(M)},y_h^{(M)})_{h=1}^N
 \right),
 \label{gs:fixed:order}
\end{equation}
with the fixed state $s_M$ omitted in the last stage.

Write
$\widetilde\Phi_2(\x):=\max_{\y\in\B_D^{MN}}\widetilde g_2(\x;\y)$
for its value function.
When $N=1$, this value function is simply $\widetilde g_2(\x)$.

\begin{lemma}[Properties of $\widetilde g_2$]
\label{gs:lem:fixed-controller-instance}
Let $M,N\ge1$ be integers and $D>0$. When $N\ge2$, assume
$D\ge16\sqrt2N\sqrt M$.
\begin{enumerate}[label=\textup{(\roman*)},leftmargin=*]
\item 
$\widetilde g_2\in\cF(c_L,D,c_{\Phi}M)$,
where $c_L,c_{\Phi}\ge1$ are numerical constants independent of $M,N,D$.
\item For $N=1$, the function $\widetilde g_2(\x)$ is a first-order
zero-chain in the order~\eqref{gs:fixed:primal-order}.
For $N\ge2$, the function $\widetilde g_2(\x;\y)$ is a first-order
zero-chain in the order~\eqref{gs:fixed:order}.
\item At every $2$-GS point, the last coordinate in the ordering
specified in \textup{(ii)} is nonzero.
\end{enumerate}
\end{lemma}

The proof is given in Section~\ref{gs:sec:fixed-controller}.

For $h>0$, set $D:=D_{\cY}/h$ in $\widetilde g_2$ and define
\begin{equation}
 g_2(\x;\y):=\frac{\ell h^2}{c_L}
 \widetilde g_2\left(\frac{\x}{h};\frac{\y}{h}\right),
 \qquad
 (\x,\y)\in\R^{MN+M-1}\times\B_{D_{\cY}}^{MN}.
 \label{gs:eq:scaled-instance-2}
\end{equation}
Section~\ref{gs:subsec:scaling} chooses $h,M,N$
to obtain the required lower bounds.

\subsection{Construction of \texorpdfstring{$g_3$}{g3}}
\label{gs:subsec:batch-instances}
To define one batch of $\widetilde g_3$, we allow $s_M$ to vary
and introduce a dual controller $u$ that multiplies $P$ by $2u/D$,
replacing the constant $3D/2$ by $3u$.
For $N\ge2$, let $\w=(\w^{(1)},\ldots,\w^{(M)})$ with
$\w^{(i)}\in\R^N$ for $i\in[M]$, and write the primal and dual variables as
\[
 \x=(\s,\z)\in\R^{MN+M},\qquad
 \y=(u,\w)\in\R^{MN+1}.
\]
For an activation parameter $\eta_0\in[0,1]$, define the batch function
on these ambient spaces by
\begin{equation}
 \mathcal G(\x,\eta_0;\y)
 :=\frac{2u}{D}P(\s,\z_1,\eta_0)
   +\sum_{i=1}^M\left[I_N(\z^{(i)},s_i;\w^{(i)})+R_1(s_i)\right]+3u.
 \label{gs:one:instance}
\end{equation}
For one batch, take $\eta_0=1$ and the dual domain
\[
 \left\{(u,\w)\in\R\times\R^{MN}
 \;\middle|\;u\ge0,\ \norm{(u,\w)}\le\frac{D}{2}\right\}.
\]

At the dual origin, the terms involving $P$ and the inner
chains vanish, leaving only the state regularizers, which
are bounded below.
The resulting bound on the initial gap is independent of $N$.
The term $3u$ prevents $u=0$ from yielding small GS residuals
while $s_M=0$.
To keep the smoothness bound uniform, we limit the number
of stages in each batch and connect batches to handle
larger initial gaps.

For $K\ge1$, take $K$ copies of $\mathcal G$ with $N\ge2$,
each with $M$ stages.
For each batch $j\in[K]$, let
$\s^{(j)}=(s_1^{(j)},\ldots,s_M^{(j)})\in\R^M$ be its state vector
and take stage variables $\z^{(j,i)},\w^{(j,i)}\in\R^N$ for $i\in[M]$.
Write $\z^{(j)}=(\z^{(j,1)},\ldots,\z^{(j,M)})$ and
$\w^{(j)}=(\w^{(j,1)},\ldots,\w^{(j,M)})$.
With a controller $u_j\ge0$, the primal and dual variables of batch $j$ are
\[
 \x^{(j)}=(\s^{(j)},\z^{(j)})\in\R^{MN+M},\qquad
 \y^{(j)}=(u_j,\w^{(j)})\in\R^{MN+1}.
\]
Set $\x=(\x^{(1)},\ldots,\x^{(K)})$ and
$\y=(\y^{(1)},\ldots,\y^{(K)})$, and let
$\u:=(u_1,\ldots,u_K)\in\R^K$ collect the controllers.
All dual variables are constrained jointly to the domain
\[
 \widehat{\cY}_D
 :=\left \{\y=(\y^{(1)},\ldots,\y^{(K)})\;\middle|\;
             \u\in\R^K_+,\ \norm{\y}\le\frac{D}{2}\right \},
\]
which has diameter $D$ and contains the origin.
Set $s_M^{(0)}:=D$ and use the boundary convention
$\eta_0^{(K+1)}:=0$. For each batch $j\in[K]$, define its input
activation and weight by
\[
 \eta_0^{(j)}:=q\left(\frac{s_M^{(j-1)}}{D}\right),\qquad
 w_j:=\eta_0^{(j)}\bigl(1-\eta_0^{(j+1)}\bigr).
\]
The final state $s_M^{(K)}$ remains a primal variable.
Using $\widetilde C_{\rm st}$ and $R_2$ from
Section~\ref{gs:subsec:inner-block}, define $\widetilde g_3$ by
\begin{equation}
 \begin{aligned}
 \widetilde g_3(\x;\y):=\ &
 \sum_{j=1}^K\left[
 \mathcal G(\x^{(j)},\eta_0^{(j)};\y^{(j)})
 -(1-w_j)\left(3u_j+\frac{1}{2}\norm{\y^{(j)}}^2\right)\right]\\
 &+\sum_{j=1}^{K-1}\left[
 R_2(s_M^{(j)})
 +D^2\widetilde C_{\rm st}\left(\frac{s_M^{(j-1)}}{D},\frac{s_M^{(j)}}{D}\right)\right].
 \end{aligned}
 \label{gs:batch:game}
\end{equation}
When $K=1$, the construction reduces to
$\widetilde g_3(\x;\y)=\mathcal G(\x,1;\y)$.
Each batch after the first is activated by the preceding
terminal state.
The weight $w_j$ retains the controller term $3u_j$ when
$\eta_0^{(j)}=1$ and $\eta_0^{(j+1)}=0$.
When $w_j=0$, this term is removed and a quadratic penalty
is applied to the batch's dual variables.
These penalties control the combined dual norm of such
batches in the GS analysis, independently of $K$.
Lemma~\ref{gs:lem:batched-instance} shows that both the gap
bound and chain length grow linearly with $K$, while the
smoothness bound and dual diameter remain independent of $K$.

Write $\w=(\w^{(1)},\ldots,\w^{(K)})$ and define the coordinate
ordering $\ord$ of $(\x,\w)$ by
\begin{equation}
 (\x,\w)_{\ord}
 =\Bigl(
   \bigl(
     (z_h^{(j,i)},w_h^{(j,i)})_{h=1}^N,
     s_i^{(j)}
   \bigr)_{i=1}^M
 \Bigr)_{j=1}^K.
 \label{gs:batch:order}
\end{equation}
The final coordinate is $s_M^{(K)}$, at position $KM(2N+1)$.

Write
$\widetilde\Phi_3(\x):=\max_{\y\in\widehat{\cY}_D}\widetilde g_3(\x;\y)$
for its value function.

\begin{lemma}[Properties of $\widetilde g_3$]
\label{gs:lem:batched-instance}
There are numerical constants $D_0\ge1$ and $\varepsilon_0\in(0,1)$
such that the following holds for $D\ge D_0$,
$N=\lfloor\sqrt{D/2}\rfloor$, and integers
$1\le M\le\lfloor D^2\rfloor$ and $K\ge1$.
\begin{enumerate}[label=\textup{(\roman*)},leftmargin=*]
\item 
$\widetilde g_3\in\cF(c_{\ell},D,c_BKD^2)$,
where $c_{\ell},c_B\ge1$ are numerical constants independent of $K,M,N,D$.
\item For every fixed feasible controller vector $\u$, the function
$\widetilde g_3(\cdot;\u,\cdot)$ is a first-order zero-chain
in the order~\eqref{gs:batch:order}.
\item At every $2\varepsilon_0$-GS point, the last coordinate in the ordering
specified in \textup{(ii)} is nonzero.
\end{enumerate}
\end{lemma}
The proof is given in Section~\ref{gs:sec:large-gap}.

For $h>0$, set $D:=D_{\cY}/h$ in $\widetilde g_3$ and define
\begin{equation}
 g_3(\x;\y):=\frac{\ell h^2}{c_{\ell}}
 \widetilde g_3\left(\frac{\x}{h};\frac{\y}{h}\right).
 \label{gs:eq:scaled-instance-3}
\end{equation}
Dilate the primal and dual domains by $h$.
Section~\ref{gs:subsec:scaling} chooses $h,M,N,K$
to obtain the required lower bound.

\endgroup

\section{A Matching Upper-Bound Algorithm}
\label{sec:upper}
To establish the matching upper bounds, we develop an
inexact proximal-point method.
By introducing dual regularization and a primal proximal term,
we obtain strongly convex-strongly concave subproblems,
which can be solved approximately using the accelerated
FOAM method of \citet{kovalev2022first}.
Since the proximal center changes across outer iterations,
we need to control the approximation error across
successive subproblems.
Our approach is to retain the full FOAM state and track
its error together with the progress of the outer method.
This allows each inner solve to use a fixed contraction
factor, without requiring a prescribed
$\epsilon$-dependent accuracy.
The OS and GS variants share the same tracking
recursion but differ in their dual regularization
parameters and output rules.
Throughout, $f\in\mathcal F(\ell,D_{\mathcal Y},\Delta)$.

% We retain the accelerated \FOAM{} state across successive
% proximal subproblems and use a computable residual as the
% stopping criterion.
% The OS and GS variants share the same tracking recursion
% but differ in their dual regularization parameters and
% output rules.
% Throughout, $f\in\cF(\ell,D_{\cY},\Delta)$.

\subsection{Regularized proximal subproblems}
\label{subsec:regularized-proximal}
Following \citet{li2026smoothing}, we introduce quadratic
regularization in the dual variable~\citep{xu2023unified} and a proximal term
in the primal variable~\citep{zhang2020single,zheng2023universal}.
For $r=(2\ell,r_{\y})$ with $0<r_{\y}\le\ell/8$, define
\[
 f_r(\x;\y):=f(\x;\y)-\frac{r_{\y}}{2}\norm{\y}^2,
 \qquad
 F_r(\x,\z;\y):=f_r(\x;\y)+\ell\norm{\x-\z}^2.
\]
For each $\z\in\R^m$, the function $F_r(\cdot,\z;\cdot)$
is $\ell$-strongly convex in $\x$ and
$r_{\y}$-strongly concave in $\y$.
We denote its unique saddle point on $\cX\times\cY$ by
$(\x_{\z,r_{\y}}^\star,\y_{\z,r_{\y}}^\star)$.

To characterize the primal update, define the regularized
value function and its Moreau envelope by
\[
 \Phi_r(\x):=\max_{\y\in\cY}f_r(\x;\y)+\iota_{\cX}(\x),
 \qquad
 p_r(\z):=\min_{\x\in\cX}
 \{\Phi_r(\x)+\ell\norm{\x-\z}^2\}.
\]
Since $\Phi_r$ is $\ell$-weakly convex,
its Moreau envelope $p_r$ is continuously differentiable.
Moreover, \citet[Lemma~2.2]{davis2019stochastic}
give the identities
\begin{equation}
  \x_{\z,r_{\y}}^\star
  =\prox_{\frac{\Phi_r}{2\ell} }(\z),
  \quad \text{and} \quad 
  \nabla p_r(\z)=2\ell(\z-\x_{\z,r_{\y}}^\star).
  \label{eq:moreau1}
\end{equation}
Thus, solving the regularized proximal subproblem
exactly yields a gradient step on $p_r$ with
stepsize $1/(2\ell)$.

% We will also use a bound on the change in the value
% function caused by dual regularization.
% Since $\bz\in\cY$ and
% $\operatorname{diam}(\cY)\le D_{\cY}$,
% we have $\norm{\y}\le D_{\cY}$ for all $\y\in\cY$.
% Consequently, for every $\x\in\cX$,
% \begin{equation}
%  0\le\Phi(\x)-\Phi_r(\x)
%  \le\frac{r_{\y}}{2}D_{\cY}^2.
%  \label{eq:value-bias}
% \end{equation}
% This estimate will be used to control the initial
% regularized gap for both stationarity criteria.

\subsection{Approximate proximal points from FOAM}
\label{subsec:approx-proximal}
To approximate the proximal step in~\eqref{eq:moreau1},
we use the accelerated FOAM method of
\citet{kovalev2022first}, with strong convexity
and concavity parameters $\mu_x=\ell$ and $\mu_y=r_{\y}$.
We work with a conjugate formulation of the proximal
subproblem, which allows us to recover an approximate
primal proximal point from the FOAM iterates and
track their error.
For a fixed proximal center $\z$, define 
\begin{align}\label{eq:P_new}
    P_{\z,r_{\y}}(\omegavec,\y)
 :=\frac{\norm{\omegavec}^2}{2\ell}
   +\frac{r_{\y}}{2}\norm{\y}^2
   +\iota_{\cY}(\y)
   +\sup_{\x\in\cX}
     \left\{\ip{\omegavec}{\x}-f(\x;\y)
           -\ell\norm{\x-\z}^2
           +\frac{\ell}{2}\norm{\x}^2\right\}.
\end{align}
Let $(\omegavec_{\z,r_{\y}}^\star,
\y_{\z,r_{\y}}^\star)$ denote its minimizer,
and let $P_{\z,r_{\y}}^\star$ denote its minimum.
The conjugate and primal solutions satisfy
$\omegavec_{\z,r_{\y}}^\star
=-\ell\x_{\z,r_{\y}}^\star$.

FOAM maintains a slow pair $(\omegavec,\y)$ and a
fast pair $(\omegavec_f,\y_f)$, collected in the state
$S=(\omegavec,\y,\omegavec_f,\y_f)$.
The slow dual variable $\y$ need not belong to $\cY$,
whereas the fast pair has finite objective value
and hence $\y_f\in\cY$.  
To measure the error of the full FOAM state, set
$\alpha_{r_{\y}}:=\sqrt{8r_{\y}/\ell}\le1$
and define the Lyapunov function
\begin{equation}
 \mathcal L_{\z,r_{\y}}(S)
 :=\frac{2\alpha_{r_{\y}}}{\ell}
      \norm{\omegavec-\omegavec_{\z,r_{\y}}^\star}^2
   +2r_{\y}\norm{\y-\y_{\z,r_{\y}}^\star}^2
   +2\bigl(P_{\z,r_{\y}}(\omegavec_f,\y_f)
           -P_{\z,r_{\y}}^\star\bigr).
 \label{eq:tracking-energy}
\end{equation}
This is $\alpha_{r_{\y}}$ times the Lyapunov
function of \citet[Supplement, Lemma~7]{kovalev2022first}.
The first two terms measure the distance of the
slow pair to the minimizer, while the last term
measures the suboptimality of the fast pair.

For fixed $(\z,r_{\y})$, a call to \FOAM{} takes an
initial state $S$ with
$\mathcal L_{\z,r_{\y}}(S)<\infty$
and a target contraction factor $\rho\in(0,1)$.
It runs a prescribed number of iterations and returns
\[
S^+
=(\omegavec^+,\y^+,\omegavec_f^+,\y_f^+)
:=\FOAM_{\z,r_{\y}}(S,\rho).
\]
The corresponding primal approximation is
$\proj_{\cX}(-\omegavec_f^+/\ell)$. 
The following lemma shows that the Lyapunov function
controls the primal approximation error and contracts
under FOAM at an accelerated oracle cost.
\begin{lemma}[Primal approximation bound and \FOAM{} contraction]
\label{prop:foam-interface}
For every $\z\in\cX$, $0<r_{\y}\le\ell/8$, and every
state $S$ with $\mathcal L_{\z,r_{\y}}(S)<\infty$,
\begin{equation}
 \ell\norm*{
  \proj_{\cX}\left(-\frac{\omegavec_f}{\ell}\right)
  -\x_{\z,r_{\y}}^\star}^2
 \le\mathcal L_{\z,r_{\y}}(S).
 \label{eq:interface-primal-error}
\end{equation}
Moreover, for every $\rho\in(0,1)$,
\begin{equation}
 \mathcal L_{\z,r_{\y}}
  \left(\FOAM_{\z,r_{\y}}(S,\rho)\right)
 \le\rho\,\mathcal L_{\z,r_{\y}}(S),
 \label{eq:block-contraction}
\end{equation}
and the call requires
$\cO(1+\sqrt{\ell/r_{\y}}\log(1/\rho))$
first-order saddle-oracle calls.
\end{lemma}
For any fixed $\rho\in(0,1)$, the contraction requires
$\cO(\sqrt{\ell/r_{\y}})$ oracle calls.
Although $\mathcal L_{\z,r_{\y}}$ depends on unknown
optimal quantities, FOAM achieves this contraction
using a prescribed iteration count without evaluating
the Lyapunov function.
The implementation and proof are given in
Appendix~\ref{app:foam-details}.

\subsection{Tracking across successive subproblems}
\label{sec:tracking}
The contraction guarantee in
Lemma~\ref{prop:foam-interface} applies to a fixed
proximal subproblem.
Since the proximal center changes across outer iterations,
we need to control how this change affects the error of
the retained FOAM state.

Consider two proximal centers $\z$ and
$\z'=\z+\bm d$.
Changing the center introduces the linear term
$-2\ell\ip{\bm d}{\x}$ in $F_r$, up to an
additive constant independent of $(\x,\y)$.
Accordingly, we translate the conjugate components
of the slow and fast pairs:
\begin{equation}
 \widetilde S
 := (\omegavec-2\ell\bm d,\y,
     \omegavec_f-2\ell\bm d,\y_f).
 \label{eq:translated-state}
\end{equation}
The translated state satisfies the following Lyapunov bound.
\begin{proposition}
\label{prop:anchor-shift}
For every state $S$ with finite
$\mathcal L_{\z,r_{\y}}(S)$,
\begin{equation}
 \mathcal L_{\z',r_{\y}}(\widetilde S)
 \le 2\mathcal L_{\z,r_{\y}}(S)
      +24\ell\norm{\bm d}^2.
 \label{eq:anchor-tracking}
\end{equation}
Consequently, applying a $\rho$-contraction
at the new center gives
\begin{equation}
 \mathcal L_{\z',r_{\y}}
 \bigl(\FOAM_{\z',r_{\y}}(\widetilde S,\rho)\bigr)
 \le\rho\left(
 2\mathcal L_{\z,r_{\y}}(S)
 +24\ell\norm{\bm d}^2\right).
 \label{eq:interface-tracked-block}
\end{equation}
\end{proposition}
This estimate allows us to track the solver error
together with the progress of the outer iterations.
The proof is given in Section~\ref{subsec:upper-stability}.

\subsection{The \TrackedFOAM{} algorithm}
\label{subsec:main-upper}
We now combine the preceding estimates into an algorithm.
Two ingredients remain: an initial FOAM state with
controlled error and a computable stopping criterion.
We first address initialization.
% Since $\bm 0\in\mathcal Y$ and
% $\operatorname{diam}(\mathcal Y)\le D_{\mathcal Y}$,
% we have $\|\y\|\le D_{\mathcal Y}$ for every
% $\y\in\mathcal Y$.
% Consequently, for every $\x\in\mathcal X$,
% For both OS and GS, the outer iteration complexity
% depends on the initial gap of the regularized
% Moreau envelope and the initial FOAM error.
Since $\bz\in\cY$ and
$\operatorname{diam}(\cY)\le D_{\cY}$,
we have $\norm{\y}\le D_{\cY}$ for every $\y\in\cY$.
Consequently, for every $\x\in\cX$,
\begin{equation}
  0\le\Phi(\x)-\Phi_r(\x)
  \le\frac{r_{\y}}{2}D_{\cY}^2.
  \label{eq:value-bias}
\end{equation}
This gives
\[
p_r(\bz)-\inf_{\z\in\R^m}p_r(\z)
\le\Delta+\frac{r_{\y}}{2}D_{\cY}^2.
\]
We therefore seek an initial FOAM state with
Lyapunov error of order
$\Delta+r_{\y}D_{\cY}^2$, while keeping
$r_{\y}$ within a constant factor of its
prescribed target.
To achieve these guarantees without additional
logarithmic overhead, we decrease the dual
regularization parameter geometrically and
reuse the FOAM state across successive levels.

\begin{proposition}[Warm-start initialization]
\label{prop:warm-start}
For any target $0<\bar r_{\y}\le\ell/8$,
an initialization at $\z^0=\bz$ constructs
$r_{\y}$ and $S^0$ satisfying
\[
\frac{\bar r_{\y}}{4}<r_{\y}\le\bar r_{\y},
\quad \text{and} \quad
\mathcal L_{\z^0,r_{\y}}(S^0)
\le 15(\Delta+r_{\y}D_{\cY}^2),
\]
using $\cO(\sqrt{\ell/r_{\y}})$ feasible
first-order saddle-oracle calls.
\end{proposition}

To construct a computable stopping criterion,
we need to account for the error in the
approximate proximal step.
Lemma~\ref{prop:foam-interface} controls this error
through the FOAM Lyapunov function, which depends
on unknown optimal quantities and cannot be
evaluated directly.
We therefore maintain a computable upper bound
$B^t$ satisfying
\[
\mathcal L_{\z^t,r_{\y}}(S^t)\le B^t,
\]
with $B^0:=15(\Delta+r_{\y}D_{\cY}^2)$
provided by Proposition~\ref{prop:warm-start}.
Keeping $r_{\y}$ fixed, at each outer iteration
we compute
\[
\z^{t+1}
:=\proj_{\cX}\left(-\frac{\omegavec_f^t}{\ell}\right),
\qquad
\bm d^t:=\z^{t+1}-\z^t.
\]
On iterations that do not terminate, we apply
FOAM at $\z^{t+1}$ to the translated state
$\widetilde S^t$ defined in
Section~\ref{sec:tracking},
with contraction factor $\rho=1/400$.
Combining Proposition~\ref{prop:anchor-shift}
with Lemma~\ref{prop:foam-interface}
gives the computable update
\begin{equation}
  B^{t+1}:=\frac{1}{400}
  \left(2B^t+24\ell\norm*{\bm d^t}^2\right),
  \label{eq:outer-budget-update}
\end{equation}
which ensures
$\mathcal L_{\z^{t+1},r_{\y}}(S^{t+1})
\le B^{t+1}$.
Moreover, by~\eqref{eq:moreau1} and
Lemma~\ref{prop:foam-interface},
the computed displacement $\bm d^t$
and the bound $B^t$ satisfy
\[
\norm{\nabla p_r(\z^t)}^2
\le 8\ell\left(
\ell\norm{\bm d^t}^2+B^t
\right).
\]
We therefore define the computable residual $Q^t:=\ell\norm{\bm d^t}^2+B^t$
and terminate when
\begin{equation}
Q^t\le\frac{\epsilon^2}{64\ell}.
\label{eq:observable-stop}
\end{equation}
Although the inexact proximal update need not decrease
$p_r$ at every iteration, the tracked FOAM
error can be incorporated into the descent analysis.
As shown in Section~\ref{sec:upper-proof}, a joint potential combining
the envelope value and the tracked error decreases
whenever the stopping criterion is not satisfied.
This ensures progress with a fixed contraction factor
at each inner solve. 

Algorithm~\ref{alg:tracked-foam-overview}
summarizes the resulting method.
The OS and GS variants share the same outer iteration and stopping rule, but use different regularization parameters and output procedures. 
% The analysis first establishes a common descent guarantee and bounds the number of outer iterations. The OS and GS guarantees then follow from different stationarity certificates. We give the complete analysis in Section~\ref{sec:upper-proof}.

% Algorithm~\ref{alg:tracked-foam-overview}
% summarizes the resulting method.
% The fixed contraction factor suffices because the joint potential
% $W^t:=p_{\mathbf r}(\bm z^t)+2B^t$ satisfies
% \[
%   W^{t+1}\le W^t-\frac14 Q^t
% \]
% at every nonterminating iteration, as shown in Section~10.3.1.
% This estimate accounts jointly for the envelope value and the
% tracked FOAM error, without requiring each subproblem to be
% solved to a prescribed $\epsilon$-dependent accuracy.
% The next two subsections specify the regularization parameters
% and output certificates for OS and GS.

\begin{algorithm}[h]
\small
\caption{\TrackedFOAM{}}
\label{alg:tracked-foam-overview}
\KwData{$f\in\cF(\ell,D_{\cY},\Delta)$,
$\epsilon>0$, and mode \texttt{OS} (default)
or \texttt{GS}}

Set $\z^0=\bz$\;

\eIf{mode is \texttt{OS}}{
Set $\bar r_{\y}
=\bar r_{\y,\epsilon}^{\mathrm{OS}}$
from~\eqref{eq:dual_reg}\;
}{
Set $\bar r_{\y}
=\bar r_{\y,\epsilon}^{\mathrm{GS}}$
from~\eqref{eq:gs-dual-regularization}\;
}

Initialize $(r_{\y},S^0)$ by
Proposition~\ref{prop:warm-start};
set $B^0=15(\Delta+r_{\y}D_{\cY}^2)$
and $t=0$\;

\While{true}{
Set
$\z^{t+1}=\proj_{\cX}(-\omegavec_f^t/\ell)$,
$\bm d^t=\z^{t+1}-\z^t$,
and
$Q^t=\ell\norm{\bm d^t}^2+B^t$\;

\If{$Q^t\le\epsilon^2/(64\ell)$}{
\eIf{mode is \texttt{OS}}{
\KwRet{$\z^t$}\;
}{
Apply Algorithm~\ref{alg:app-feasible-rprox}
at $(\z^t,r_{\y})$ with input
$(\omegavec_f^t,\y_f^t)$ and $\tau=1/4$\;
\KwRet{its feasible pair $(\x_f,\y_f^+)$}\;
}
}

Set
$S^{t+1}=\FOAM_{\z^{t+1},r_{\y}}
((\omegavec^t-2\ell\bm d^t,\y^t,
\omegavec_f^t-2\ell\bm d^t,\y_f^t),1/400)$\;

Set
$B^{t+1}
=(2B^t+24\ell\norm{\bm d^t}^2)/400$
and $t=t+1$\;
}
\end{algorithm}

\subsection{Optimization stationarity}
\label{sec:os-output}

In the OS mode, Algorithm~\ref{alg:tracked-foam-overview}
returns the proximal center $\z^t$ once the
stopping criterion is satisfied.
This criterion controls $\norm{\nabla p_r(\z^t)}$,
whereas OS requires a bound on
$\norm{\nabla\Phi_{\frac{1}{2\ell}}(\z^t)}$.
Since $p_r$ is the Moreau envelope of the
regularized value function $\Phi_r$, we first
bound the difference between these two gradients.

\begin{lemma}[Regularization bias]
\label{lem:smoothing}
For every $\z\in\R^m$,
\[
 \norm{\nabla\Phi_{\frac{1}{2\ell}}(\z)-\nabla p_r(\z)}
 \le D_{\cY}\sqrt{2\ell r_{\y}}.
\]
\end{lemma}
\begin{proof}[Proof of Lemma~\ref{lem:smoothing}]
Let $\x^\star:=\prox_{\frac{\Phi}{2\ell}}(\z)$ and
$\x_r^\star:=\prox_{\frac{\Phi_r}{2\ell}}(\z)$.
Both $\Phi$ and $\Phi_r$ are $\ell$-weakly convex, so
$\x\mapsto\Phi(\x)+\ell\norm{\x-\z}^2$ and
$\x\mapsto\Phi_r(\x)+\ell\norm{\x-\z}^2$
are $\ell$-strongly convex.
Applying strong convexity to each function at the
minimizer of the other and adding them together gives
\begin{align*}
 \ell\norm*{\x^\star-\x_r^\star}^2
 \le\ & [\Phi(\x_r^\star)-\Phi_r(\x_r^\star)]
     -[\Phi(\x^\star)-\Phi_r(\x^\star)]
 \le\frac{r_{\y}}{2}D_{\cY}^2,
\end{align*}
where the last inequality follows from~\eqref{eq:value-bias}.
Consequently,
\[
 \norm*{\nabla\Phi_{\frac{1}{2\ell}}(\z)-\nabla p_r(\z)}
 =2\ell\norm*{\x^\star-\x_r^\star}
 \le D_{\cY}\sqrt{2\ell r_{\y}}.
\]
We finish the proof.
\end{proof}

By Lemma~\ref{lem:smoothing}, a choice
$r_{\y}\le\epsilon^2/(8\ell D_{\cY}^2)$
ensures that the bias is at most $\epsilon/2$.
To obtain this bound while using the largest
permitted regularization scale, we set the OS target
\begin{equation}
 \bar r_{\y,\epsilon}^{\mathrm{OS}}
 :=\min\left\{
  \frac{\ell}{8},
  \frac{\epsilon^2}{8\ell D_{\cY}^2}
 \right\}.
 \label{eq:dual_reg}
\end{equation}
The warm start then returns $\bar r_{\y,\epsilon}^{\mathrm{OS}}/4<r_{\y}\le
\bar r_{\y,\epsilon}^{\mathrm{OS}}$. 
The upper bound on $r_{\y}$ controls the regularization bias, while the lower bound controls the cost of each FOAM call. Section~\ref{sec:upper-proof} combines these bounds with the common descent analysis to establish OS stationarity and the claimed oracle complexity.

\subsection{Game stationarity}
\label{sec:gs-output}
For GS, the stopping criterion controls
$\norm{\nabla p_r(\z)}$ and the FOAM error,
but does not directly certify stationarity of
a feasible primal-dual pair.
We obtain such a pair by performing one additional
inexact proximal step initialized at the
FOAM fast pair $(\omegavec_f,\y_f)$.
The following lemma bounds the resulting primal
and dual GS residuals.
\begin{lemma}[Certificate for game stationarity]
\label{lem:gs-certificate}
Fix $\z\in\cX$, $0<r_{\y}\le\ell/8$, and a state
$S=(\omegavec,\y,\omegavec_f,\y_f)$ with
finite $\mathcal L_{\z,r_{\y}}(S)$.
With input $(\omegavec_g,\y_g)=(\omegavec_f,\y_f)$
and $\tau=1/4$,
Algorithm~\ref{alg:app-feasible-rprox}
uses $\cO(1)$ feasible first-order saddle-oracle calls
and returns a pair $(\x_f,\y_f^+)\in\cX\times\cY$.
Define its primal and dual GS stationarity residuals by $R_x :=\dist\bigl(\bz,\nabla_{\x}f(\x_f,\y_f^+)
       +\partial\iota_{\cX}(\x_f)\bigr), 
R_y :=\dist\bigl(\bz,-\nabla_{\y}f(\x_f,\y_f^+)
       +\partial\iota_{\cY}(\y_f^+)\bigr)$.
Then
\[
\begin{aligned}
R_x
 \le\ &
\norm{\nabla p_r(\z)}
+\left(2+\frac{2}{\sqrt{3}}\right)
\sqrt{\ell\mathcal L_{\z,r_{\y}}(S)} \quad \text{and}\\
R_y
 \le\ &
\sqrt{\frac{\ell\mathcal L_{\z,r_{\y}}(S)}{6}}
+r_{\y}D_{\cY}.
\end{aligned}
\]
\end{lemma}
At termination, the stopping criterion ensures
\[
\norm{\nabla p_r(\z^t)}^2\le8\ell Q^t  \quad \text{and} \quad 
\mathcal L_{\z^t,r_{\y}}(S^t)\le B^t\le Q^t.
\]
Thus, the first terms in the residual bounds
are controlled by the stopping criterion,
while the dual regularization error can be made
at most $\epsilon/4$ by requiring
$r_{\y}\le\epsilon/(4D_{\cY})$.
In addition, the outer iteration complexity
depends on $\Delta+r_{\y}D_{\cY}^2$.
To retain the bound
$\cO(1+\ell\Delta/\epsilon^2)$,
we impose
$r_{\y}D_{\cY}^2\le
(\Delta+\epsilon^2/\ell)/4$.
Together with the restriction $r_{\y}\le\ell/8$,
this leads to the GS target
\begin{equation}
\label{eq:gs-dual-regularization}
\bar r_{\y,\epsilon}^{\mathrm{GS}}
:=
\min\left\{
\frac{\ell}{8},
\frac{\epsilon}{4D_{\cY}},
\frac{\Delta+\epsilon^2/\ell}{4D_{\cY}^2}
\right\}.
\end{equation}
% The warm start returns
% $r_{\y}\le\bar r_{\y,\epsilon}^{\mathrm{GS}}$.
% The stopping criterion controls the gradient of the regularized Moreau envelope and the Lyapunov error of the FOAM state. The GS target must additionally control the dual regularization error and the initial potential gap. Accordingly, we choose
% \begin{equation}
% \label{eq:gs-dual-regularization}
% \bar r_{\y,\epsilon}^{\mathrm{GS}}
% :=
% \min\left\{
% \frac{\ell}{8},
% \frac{\epsilon}{4D_{\cY}},
% \frac{\Delta+\epsilon^2/\ell}{4D_{\cY}^2}
% \right\}.
% \end{equation}
The warm start returns $ \bar r_{\y,\epsilon}^{\mathrm{GS}}/ 4 
<r_{\y}\le\bar r_{\y,\epsilon}^{\mathrm{GS}}$.
The second restriction controls the dual regularization error in the GS certificate, while the third controls the initial potential gap. Section~\ref{sec:upper-proof} uses these bounds to establish GS stationarity and the corresponding oracle complexity.

\section{Proof of the OS Lower Bound}
\label{sec:verification}
\begingroup

We verify Lemma~\ref{lem:os-instance-properties} and choose the scaling
parameters for $f_3$, then combine its bound with
Lemma~\ref{lem:os-endpoint} to prove Theorem~\ref{thm:deterministic}.

\subsection{Properties of \texorpdfstring{$\bar f$}{fbar}}
\label{subsec:os-properties-proof}

\subsubsection{Smoothness, dual concavity, and initial gap}
\label{subsec:membership}
\label{subsec:initial-gap}

\begin{proof}[Proof of Lemma~\ref{lem:os-instance-properties}~\textup{(i)}]
For each $i\in[M]$, let $B_i$ denote the $i$-th summand in
\eqref{eq:hard-instance}.
The derivative bounds in Lemmas~\ref{lem:inner-interface},
\ref{lem:gate-transformations}, \ref{lem:outer-gate},
and~\ref{lem:identity-extensions} give
$\norm{\nabla^2B_i}\le\ell_B$ for a numerical constant
$\ell_B>0$ independent of $M,N,D$.
Since $B_i$ and $B_j$ depend on disjoint coordinates when
$|i-j|\ge2$, the sums over odd and even indices each have
a block diagonal Hessian of norm at most $\ell_B$.
Thus $\fbar$ is jointly $\ell_0$-smooth with $\ell_0:=2\ell_B$.

For $N\ge1$, the inner chains are the only terms involving $\y$.
Their disjoint dual coordinates and
Lemma~\ref{lem:inner-interface}~\textup{(ii)} make $\fbar$
$N^{-2}$-strongly concave in $\y$.
For $N=0$, the dual domain is a singleton.

At the primal origin, all outer terms vanish, and
Lemma~\ref{lem:inner-interface}~\textup{(iii)} gives
$\bar\Phi_3(\bz,\bz)=0$.

At any primal point, choose the feasible dual point $\y=\bz$.
Then $H(a_i,b_i;\bz)\ge0$.
Also, $C_{\rm st}\ge0$: it vanishes for $u\ge1$, while
$e_{\s}(u)\le e_{\s}(1)=1$ for $u<1$.
The bounds in Lemmas~\ref{lem:outer-gate}
and~\ref{lem:identity-extensions} give
$C_{\rm en}\ge-86$ and $C_{\rm ex}\ge-215/4$.
Finally, $p\ge0$ and $0\le p'\le1$ imply
$R\ge-(9/10)(c_R+1)$.
Consequently,
\[
 \bar\Phi_3(\s,\bm\nu)
 \ge\fbar(\s,\bm\nu;\bz)
 \ge-M\left[86+\frac{215}{4}+\frac{9}{10}(c_R+1)\right]
 =:-c_\Delta M.
\]
Together with the value at the origin, this gives an initial gap
of at most $c_\Delta M$ and proves
$\fbar\in\cF(\ell_0,D,c_\Delta M)$.

\end{proof}

\subsubsection{Zero-chain}
\label{subsec:zero-chain}

\begin{proof}[Proof of Lemma~\ref{lem:os-instance-properties}~\textup{(ii)}]
The definitions give $R'(0)=0$ and, for every $u,a,b\in\R$,
\begin{equation}
\begin{aligned}
  \nabla_{(u,v)}C_{\rm st}(u,0)=\ & (0,0),\qquad
  \nabla_{(u,a,v)}C_{\rm en}(u,0,0)=(0,-4q(u),0),\\
  \nabla_vC_{\rm en}(u,a,0)=\ & 0,\qquad\mbox{and}\qquad
  \nabla_{(u,b,v)}C_{\rm ex}(u,b,0)
  =(0,0,-q(u)e_{\bm{\nu}}(b)).
\end{aligned}
\label{eq:connector-gradient}
\end{equation}

Fix $k\in\{0,\ldots,L_3-1\}$ and a point whose support is contained
in the first $k$ coordinates of $(\x,\y)_{\ord_3}$.
Let $i$ be the stage containing the $(k+1)$-st coordinate.
Since $s_i$ lies after the first $k$ coordinates, $s_i=0$.
Every later stage therefore has all its variables equal to zero,
so its gradient vanishes by $R'(0)=0$,
\eqref{eq:connector-gradient}, and $\nabla H(0,0;\bz)=\bz$.
Earlier stages depend only on the first $k$
coordinates. In stage $i$, \eqref{eq:connector-gradient} shows
that an outer term can have a nonzero gradient beyond this
prefix only at $a_i$, or at $s_i$ if $b_i$ is already in the
prefix. In either case, the revealed coordinate is the
$(k+1)$-st. The inner zero-chain property in
Lemma~\ref{lem:inner-interface}~\textup{(iii)} gives the same
restriction for $H$. Hence the gradient support is contained
in the first $k+1$ coordinates, proving the claim.
\end{proof}

\subsubsection{Differentiability and gradient lower bound}
\label{subsec:terminal}

For $N\ge1$, strong concavity gives a unique dual
maximizer on the compact ball; for $N=0$, the dual domain is a singleton.
Danskin's theorem therefore gives differentiability of $\bar\Phi_3$.

Assuming that $s_M\le1/5$ and the gradient of the value function
is small, we establish the state structure and bound the connector
vector.
These bounds make the unconstrained maximizer feasible
when $N\le c_DD$, so we can compute the inner gradients
from~\eqref{eq:effective-link} and derive a contradiction.

Write
\[
  V_D(\bm{\nu})
  :=
  \max_{\y\in\B^{\bar{n}}_D}
  \sum_{i=1}^{M} H(a_i,b_i;\y^{(i)}).
\]
Then $\bar\Phi_3$ is the sum of $V_D$ and the outer terms,
and $V_D$ is independent of the state vector $\s$.

Set $\tau_0:=1/4$.
We call a state coordinate low if it lies in
$(-1/10,1/5]$ and high if it lies in $[1,\infty)$.

\begin{lemma}
\label{lem:bounded-connector}
For every admissible pair $(M,N)$, every $D>0$, and every
$(\s,\bm{\nu})\in\R^{3M}$, if
\[
  s_M\le\frac{1}{5},
  \qquad\mbox{and}\qquad
  \norm*{\nabla\bar\Phi_3(\s,\bm{\nu})}\le\tau_0,
\]
then there is a unique $j\in[M]$ satisfying
\begin{equation}
  s_i\ge1\quad(0\le i<j),
  \qquad\mbox{and}\qquad
  -\frac{1}{10}<s_i\le\frac{1}{5}\quad(j\le i\le M).
\label{eq:frontier}
\end{equation}
Equivalently, the state coordinates have the form
\[
  \underbrace{s_0,\ldots,s_{j-1}}_{\mathrm{high}}
  \ \big|\
  \underbrace{s_j,\ldots,s_M}_{\mathrm{low}}.
\]
Moreover, $\norm{\bm{\nu}}\le20$.
\end{lemma}

\begin{proof}[Proof of Lemma~\ref{lem:bounded-connector}]
Since $V_D$ is independent of $\s$, the definition of $c_R$ gives
$|\nabla_{s_i}\bar\Phi_3-R'(s_i)|\le c_R$.
By~\eqref{eq:phase-potential}, $R'\le0$, with $R'(t)=-24$ for
$t\le-1/10$ and $R'(t)=-(c_R+1)$ for $1/5\le t\le1$.
For $s_i\le1/5$, the constant branches of $q$ give
\[
  \nabla_{s_i}\bar\Phi_3
  =R'(s_i)-q(s_{i-1})e_{\s}'(s_i)e_{\bm{\nu}}(b_i)
       -24q(s_{i+1})e_{\s}'(s_i),
\]
where the last term is omitted when $i=M$.
Using $0\le q,e_{\s}'\le1$, $|e_{\bm{\nu}}|\le43/2<22$,
and the bound defining $c_R$, we obtain
\[
 \nabla_{s_i}\bar\Phi_3\le
 \begin{cases}
 -24+22=-2,&s_i\le-\frac{1}{10},\\
 -24+22=-2,&i<M,\ -\frac{1}{10}<s_i\le\frac{1}{5},\
                       s_{i+1}\ge1,\\
 -(c_R+1)+c_R=-1,&\frac{1}{5}<s_i<1.
 \end{cases}
\]
In the second case, $e_{\s}'(s_i)=q(s_{i+1})=1$ and $R'\le0$.
All three cases contradict the gradient bound $\tau_0<1$.
Thus every state is low or high, and a low state cannot be
followed by a high one. Since $s_0=1$ is high and $s_M$ is low,
the first low index is the unique $j$ satisfying~\eqref{eq:frontier}.

By~\eqref{eq:frontier}, $q(s_{i-1})(1-q(s_i))$ vanishes unless
$i=j$, and $e_{\s}(s_j)=s_j$. Since $R$ and $C_{\rm st}$ are
independent of $\bm\nu$, the connector derivatives satisfy
\[
 \nabla_{(a_i,b_i)}(\bar\Phi_3-V_D)
 =
 \begin{cases}
 \bigl(-4e_{\bm\nu}'(a_j),-s_j e_{\bm\nu}'(b_j)\bigr),&i=j,\\
 (0,0),&i\ne j.
 \end{cases}
\]
Using $0\le e_{\bm\nu}'\le1$ and $|s_j|\le1/5$ gives
\[
 \norm{\nabla_{\bm\nu}\bar\Phi_3-\nabla V_D}
 \le\sqrt{16+s_j^2}<7.
\]
Lemma~\ref{lem:inner-interface}~\textup{(iv)} and Cauchy-Schwarz
therefore yield
\[
  \frac{2}{5}\norm{\bm{\nu}}
  \le\norm{\nabla V_D(\bm{\nu})}
  \le\norm{\nabla_{\bm{\nu}}\bar\Phi_3}+7
  \le\tau_0+7.
\]
Thus $\norm{\bm{\nu}}\le(5/2)(7+\tau_0)<20$.

\end{proof}

\begin{proof}[Proof of Lemma~\ref{lem:os-instance-properties}~\textup{(iii)}]
Set $c_D:=(40c_{\y})^{-1}$.
Fix an admissible pair $(M,N)$, $D>0$, and a primal point satisfying
$N\le c_DD$ and $s_M\le1/5$.
Suppose, toward a contradiction, that
$\norm{\nabla\bar\Phi_3(\s,\bm{\nu})}\le\tau_0$.

Lemma~\ref{lem:bounded-connector} gives a frontier index $j$
satisfying \eqref{eq:frontier} and the bound $\norm{\bm{\nu}}\le20$.
Let
\[
  \y^\star(\bm{\nu})
  :=\bigl(\w^\star(a_1,b_1),\ldots,\w^\star(a_M,b_M)\bigr)
\]
be the unconstrained dual maximizer.

Summing the squared bounds in Lemma~\ref{lem:inner-interface}~\textup{(i)}
and using $N\le c_DD$ gives
\[
  \norm*{\y^\star(\bm{\nu})}
  \le c_{\y}N\norm*{\bm{\nu}}
  \le20c_{\y}N
  \le\frac{D}{2}.
\]
For $N\ge1$, this is also the unique maximizer in $V_D$.
Danskin's theorem and~\eqref{eq:effective-link} therefore give
\[
  \nabla V_D(\bm{\nu})
  =\nabla_{\bm{\nu}}\sum_{i=1}^M Q(a_i,b_i).
\]
For $N=0$, this identity holds directly because
$V_D(\bm{\nu})=\sum_{i=1}^M Q(a_i,b_i)$.

At stage $j$, $q(s_{j-1})=1$ and
$q(s_j)=q'(s_j)=0$; the bounds on $s_j,a_j,b_j$ place them
inside the identity intervals of $e_{\s}$ and $e_{\bm{\nu}}$.
Define $g_a:=\nabla_{a_j}\bar\Phi_3(\s,\bm{\nu})$ and
$g_b:=\nabla_{b_j}\bar\Phi_3(\s,\bm{\nu})$.
The preceding gradient identity and the definitions
of the entrance and exit couplings give
\[
  g_a=2a_j-b_j-4,
  \qquad\mbox{and}\qquad
  g_b=-a_j+2b_j-s_j.
\]
Since $s_j>-1/10$ and $|g_a|,|g_b|\le\tau_0$, we obtain
\[
  b_j
  =\frac{4+2s_j+g_a+2g_b}{3}
  >\frac{4-1/5-3\tau_0}{3}
  =\frac{61}{60}>1.
\]

If $j<M$, then $q(s_{j+1})=0$ by~\eqref{eq:frontier}.
The state derivative formula above therefore gives
\[
  \nabla_{s_j}\bar\Phi_3
  =R'(s_j)-b_j
  \le-b_j<-1,
\]
where we use $R'(s_j)\le0$ and $b_j>1$.
Hence $\norm{\nabla\bar\Phi_3(\s,\bm{\nu})}>1>\tau_0$, a contradiction.
Taking $g_0:=\tau_0$ proves the claim.
\end{proof}

\subsection{Scaling and proof of the OS lower bound}
\label{subsec:proof-zr}
\label{subsec:os-full-range}
\begin{proof}[Proof of Theorem~\ref{thm:zr-lower}]
Decrease $g_0$ if necessary so that $g_0\le\ell_0$, and set
\[
 \begin{gathered}
 d_0:=\frac{c_Dg_0}{4\ell_0},\qquad
 c_0:=\min\left\{\frac{1}{2},
       \frac{g_0}{\sqrt{128c_\Delta\ell_0}}\right\},\qquad
 c_1:=\frac{g_0^2}{64c_\Delta\ell_0}\min\{1,d_0\}.
 \end{gathered}
\]
Choose
\begin{equation}
 \begin{gathered}
 \lambda:=\frac{4\ell_0\epsilon}{g_0\ell},\qquad
 D:=\frac{D_{\cY}}{\lambda},\qquad N:=\lfloor c_DD\rfloor,\qquad
 M:=\left\lfloor\frac{\ell_0\Delta}
                       {2c_\Delta\ell\lambda^2}\right\rfloor.
 \end{gathered}
 \label{eq:lower-bound-parameters}
\end{equation}
This choice gives
$N=\lfloor d_0\ell D_{\cY}/\epsilon\rfloor$.
Thus $N=0$ when $\epsilon>d_0\ell D_{\cY}$ and $N\ge1$ otherwise,
as indicated in Table~\ref{tab:os-regimes}.
The accuracy assumption ensures that the quantity inside the floor
defining $M$ is at least $4$, hence
\[
 M\ge\frac{g_0^2\ell\Delta}{64c_\Delta\ell_0\epsilon^2}.
\]
Thus $M\ge4$ and $N\in\N_0$ with $N\le c_DD$, so $(M,N)$ is admissible.
Scaling~\eqref{eq:os-scaled-instance} gives joint $\ell$-smoothness,
dual concavity, and dual diameter at most $D_{\cY}$, and preserves
the zero-chain in the order $\ord_3$.
Rescaling the dual domain and differentiating also give
\begin{equation}
 \Phi_3(\x)
 =\frac{\ell\lambda^2}{\ell_0}
   \bar\Phi_3\!\left(\frac{\x}{\lambda}\right),
 \qquad
 \nabla\Phi_3(\x)
 =\frac{\ell\lambda}{\ell_0}
   \nabla\bar\Phi_3\!\left(\frac{\x}{\lambda}\right).
 \label{eq:gradient-scaling}
\end{equation}
By Lemma~\ref{lem:os-instance-properties}, the initial gap satisfies
\[
 \Phi_3(\bz)-\inf_{\x\in\R^{3M}}\Phi_3(\x)
 \le\frac{\ell\lambda^2}{\ell_0}c_\Delta M\le\Delta.
\]
Thus $f_3\in\cF(\ell,D_{\cY},\Delta)$.

Fix any $\x\in\R^{3M}$ with $x_{s_M}=0$.
If $\x$ were $\epsilon$-OS, set
$\uvec=\prox_{\frac{\Phi_3}{2\ell}}(\x)$. The Moreau gradient identity gives
\[
 \begin{aligned}
 \left|\left(\frac{\uvec}{\lambda}\right)_{s_M}\right|
 \le\ & \frac{\norm{\uvec-\x}}{\lambda}
 =\frac{\norm{\nabla[\Phi_3]_{1/(2\ell)}(\x)}}{2\ell\lambda}\le\frac{\epsilon}{2\ell\lambda}
 =\frac{g_0}{8\ell_0}\le\frac{1}{8}<\frac{1}{5}.
 \end{aligned}
\]
Differentiability of $\Phi_3$, proximal optimality, and the gradient
bound for $\bar\Phi_3$ give
\[
 \epsilon\ge\norm{\nabla[\Phi_3]_{1/(2\ell)}(\x)}
 =\norm{\nabla\Phi_3(\uvec)}
 =\frac{\ell\lambda}{\ell_0}
       \norm*{\nabla\bar\Phi_3\left(\frac{\uvec}{\lambda}\right)}
 \ge\frac{\ell\lambda g_0}{\ell_0}=4\epsilon,
\]
a contradiction. Hence every primal point with $x_{s_M}=0$ fails
the $\epsilon$-OS criterion.

Since $s_M$ is last in $\ord_3$, support on the first $L_3-1$
coordinates forces $x_{s_M}=0$.
Lemma~\ref{lem:deterministic-construction-os}, applied with $L=L_3$,
therefore gives a lower bound of $L_3$ queries.
Using the bound on $M$ and
$N+3\ge\max\{1,c_DD\}=\max\{1,d_0\ell D_{\cY}/\epsilon\}$,
we obtain
\[
 \cT_{\epsilon}^{\OS}\bigl(
   \cA_{\rm det},\cF(\ell,D_{\cY},\Delta)
 \bigr)
 \ge L_3=M(N+3)
 \ge c_1\frac{\ell\Delta}{\epsilon^2}
      \max\left\{1,\frac{\ell D_{\cY}}{\epsilon}\right\}.
\]

\end{proof}

\begin{proof}[Proof of Theorem~\ref{thm:deterministic}]
Let $c_0,c_1$ be the constants in
Theorem~\ref{thm:zr-lower}.
If $\epsilon\le c_0\sqrt{\ell\Delta}$, that theorem gives
the claimed lower bound after adjusting the numerical constant.
Otherwise, for
$c_0\sqrt{\ell\Delta}<\epsilon<\sqrt{8\ell\Delta/3}$,
Lemma~\ref{lem:os-endpoint} gives
\[
 \begin{aligned}
 \cT_\epsilon^{\OS}(\cA_{\rm det},\cF(\ell,D_{\cY},\Delta))
 \ge\ & \frac{1}{100}\left(1+\frac{\ell D_{\cY}}{\epsilon}\right)\ge\frac{c_0^2}{100}\frac{\ell\Delta}{\epsilon^2}
       \max\left\{1,\frac{\ell D_{\cY}}{\epsilon}\right\}.
 \end{aligned}
\]
Since $1+t\le2\max\{1,t\}$ for $t\ge0$,
$c_{\mathrm{OS}}:=\tfrac12\min\{c_1,c_0^2/100\}>0$ suffices.
\end{proof}
\endgroup

\section{Proof of the GS Lower Bound}
\label{gs:sec:verification}
\begingroup

We prove the properties of $\widetilde g_2$ and $\widetilde g_3$,
then choose the scaling parameters and apply
Lemma~\ref{lem:deterministic-construction-gs} to establish
Theorem~\ref{gs:thm:deterministic}.

\subsection{Properties of \texorpdfstring{$\widetilde g_2$}{g2}}
\label{gs:sec:fixed-controller}

For the proofs below, use the shorthand
\[
 q_{\rm en}(t):=e_2(3c_s-4t),\qquad
 q_{\rm ex}(t):=1-e_2(t).
\]
For $N\ge2$, we use the vectors $\w_0$ and
$\vvec:=\bD_N^{-1}\ee_1$ from
Lemma~\ref{gs:mem:matrix-bounds}. In particular,
\[
 \bD_N^\top\w_0=\frac{\ee_N}{\sqrt N},
 \qquad \norm{\w_0}<\sqrt3,
 \qquad \norm{\vvec}=N.
\]
In this subsection, write $P:=P(\s,\a,1)$, with $\a=\z_1$,
and set $h_i:=a_i-s_i$, $\eta_i:=q_{\rm en}(s_i)$, and
$\sigma_i:=q_{\rm ex}(s_i)$ for $i\in[M]$.
The fixed states give $\eta_0=1$ and $\eta_M=\sigma_M=0$.

\subsubsection{Smoothness, dual concavity, and initial gap}

\begin{proof}[Proof of Lemma~\ref{gs:lem:fixed-controller-instance}~\textup{(i)}]
Lemmas~\ref{gs:lem:affine-extension}, \ref{gs:lem:regularizers},
and~\ref{gs:mem:matrix-bounds} give a numerical Hessian bound for
each stage potential, state regularizer, and inner coupling. Stages of the same parity involve disjoint coordinates.
Splitting the sum into its odd and even stages therefore gives a
joint smoothness constant $c_L\ge1$ independent of $M,N,D$.
The function is independent of $\y$ when $N=1$ and affine in $\y$
when $N\ge2$, so it is concave in $\y$.

For the gap bound, set $c_{\Phi}:=32+(52+9c_S)c_s$.
Since $s_M=0$ and $R_1(0)=0$, only the first $M-1$ states
contribute to the regularizer. Dropping the nonnegative terms
involving $e_1$ in $P$ and using $(1-\sigma_i)s_i\le2c_s$ gives
\begin{equation}
 P
 \ge-8\sum_{i=1}^M h_i
      -8\sum_{i=1}^{M-1}(1-\sigma_i)s_i-36c_s(M-1)
 \ge-8\sum_{i=1}^M h_i-52c_s(M-1).
 \label{gs:fixed:potential-lower}
\end{equation}
For $N\ge2$, take $\overline{\y}=(8\vvec,\ldots,8\vvec)$.
Lemma~\ref{gs:mem:matrix-bounds} and $D\ge16\sqrt2N\sqrt M$ give
\[
 \norm{\overline{\y}}=8N\sqrt M\le\frac{D}{2},
 \qquad I_N(\z^{(i)},s_i;8\vvec)=8h_i.
\]
Evaluating at this feasible dual point and using $R_1\ge-9c_Sc_s$
therefore yields
\[
 \widetilde\Phi_2(\x)\ge\frac{3D}{2}-c_{\Phi}(M-1).
\]
For $N=1$, completing the square instead gives
\[
 \widetilde\Phi_2(\x)-\frac{3D}{2}
 =P+\frac{1}{2}\sum_{i=1}^M h_i^2+\sum_{i=1}^{M-1}R_1(s_i)
 \ge-32M-(c_{\Phi}-32)(M-1)\ge-c_{\Phi}M.
\]
At the primal origin, $\widetilde g_2$ equals $3D/2$ for every
feasible dual point. Consequently,
\begin{equation}
 \widetilde\Phi_2(\bz)-\inf_{\x\in\R^{MN+M-1}}\widetilde\Phi_2(\x)
 \le
 \begin{cases}
 c_{\Phi}M,&N=1,\\
 c_{\Phi}(M-1),&N\ge2.
 \end{cases}
 \label{gs:fixed:gap}
\end{equation}
In particular, the gap is zero when $N\ge2$ and $M=1$.
Together with smoothness and dual concavity, this proves
$\widetilde g_2\in\cF(c_L,D,c_{\Phi}M)$.
\end{proof}

\subsubsection{Zero-chain property}

\begin{proof}[Proof of Lemma~\ref{gs:lem:fixed-controller-instance}~\textup{(ii)}]
For $N\ge2$, use the order~\eqref{gs:fixed:order}.
At the origin, the only nonzero gradient coordinate is $a_1$,
with derivative $-8$. Direct differentiation gives, for every
$a\in\R$,
\[
 \begin{gathered}
 \nabla\widetilde C_{\rm en}(0,0)=\bz,\qquad
 \partial_t\widetilde C_{\rm en}(0,a)=0,\qquad
 \nabla\widetilde C_{\rm ex}(0,a)=\bz,\qquad R_1'(0)=0,
 \end{gathered}
\]
where each coupling gradient is taken in both arguments.

Consider a query with $\supp_{\ord}(\x,\y)\subseteq[r]$ for
$0\le r<2MN+M-1$, and let $i$ be the stage containing coordinate
$r+1$. The identities above make every later stage's gradient
zero: its head and preceding state both vanish.
Within stage $i$, the outer terms cannot reveal $s_i$ before the
inner chain reaches it, even if $a_i$ is already in the prefix.
They can reveal $a_i$ only when it is coordinate $r+1$.
Together with the inner zero-chain property in
Lemma~\ref{gs:mem:matrix-bounds}, this gives
\[
 \supp_{\ord}\bigl(\nabla\widetilde g_2(\x;\y)\bigr)
 \subseteq[r+1].
\]
For $N=1$, the same coupling identities and the direct quadratic
couplings give the zero-chain in the
order~\eqref{gs:fixed:primal-order}.
\end{proof}

\subsubsection{Nonzero terminal coordinate}

\begin{proof}[Proof of Lemma~\ref{gs:lem:fixed-controller-instance}~\textup{(iii)}]
Let $(\x,\y)$ be a $2$-GS point. Since the primal domain is
unconstrained, $\norm{\nabla_{\x}\widetilde g_2}\le2$.
For $i<M$, define
\[
 \delta_i:=
 \begin{cases}
 \nabla_{a_i}\widetilde g_2+\nabla_{s_i}\widetilde g_2,
     &N=1,\\
 \langle\w_0,\nabla_{\z^{(i)}}\widetilde g_2\rangle
       +\nabla_{s_i}\widetilde g_2,
     &N\ge2.
 \end{cases}
\]
In either case, $|\delta_i|\le4$ and the inner coupling cancels.
Direct differentiation gives
\[
 \delta_i=8\Bigl[
 -1+\sigma_i+\sigma_i'[s_i+e_1(h_i)]
 +(1-\eta_{i-1})e_1'(a_i)-\eta_i'e_1(a_{i+1})\Bigr]+R_1'(s_i).
\]
If an internal index has $\eta_i<1$, choose the first such index.
Then $\eta_{i-1}=1$ and $\sigma_i=\sigma_i'=0$.
Since $\eta_i'\ge0$, $e_1\ge0$, and $R_1'\le0$,
\[
 \delta_i=-8[1+\eta_i'e_1(a_{i+1})]+R_1'(s_i)\le-8,
\]
contradicting $|\delta_i|\le4$.
Thus $\eta_i=1$ for $0\le i<M$.
For $N\ge2$, the last stage consequently satisfies
\[
 \left|\frac{y_N^{(M)}}{\sqrt N}-8\right|
 =\left|\langle\w_0,\nabla_{\z^{(M)}}\widetilde g_2\rangle\right|
 \le2\norm{\w_0}<2\sqrt3<8,
\]
so $y_N^{(M)}\ne0$.
For $N=1$, the last primal coordinate is nonzero because
\[
 |a_M-8|=|\nabla_{a_M}\widetilde g_2|\le2.
\]
\end{proof}

\subsection{Properties of \texorpdfstring{$\widetilde g_3$}{g3}}
\label{gs:sec:large-gap}

Throughout this subsection, use the parameter ranges in
Lemma~\ref{gs:lem:batched-instance}. Choose $R_0$ first and then a
sufficiently large $D_0\ge\max\{16c_s,4S_0\}$.

Write $t_0:=D$ and $t_j:=s_M^{(j)}$ for $j\in[K]$, and set
\begin{equation}
 b_0:=1,\qquad b_j:=q\left(\frac{t_j}{D}\right)\quad(1\le j<K),\qquad b_K:=0.
 \label{gs:eq:batch-boundary-gates}
\end{equation}
The convention $b_K=0$ leaves $t_K$ unrestricted.
We also use the shorthand
\[
 \begin{gathered}
 \alpha_j:=\frac{2u_j}{D}\in[0,1],\qquad
 a_i^{(j)}:=z_1^{(j,i)},\qquad
 h_i^{(j)}:=a_i^{(j)}-s_i^{(j)},\qquad
 P^{(j)}:=P\bigl(\s^{(j)},(a_i^{(j)})_{i=1}^M,\eta_0^{(j)}\bigr).
 \end{gathered}
\]
The stage gates $\eta_i^{(j)},\sigma_i^{(j)}$ are defined as in
Section~\ref{gs:sec:fixed-controller}, with $s_i^{(j)}$ in place of $s_i$.

\subsubsection{Smoothness, dual concavity, and initial gap}
\begin{proof}[Proof of Lemma~\ref{gs:lem:batched-instance}~\textup{(i)}]
We first verify joint smoothness. The scalar-function bounds and
the overlap of only neighboring stages give, with derivatives in
$(\x^{(j)},t_{j-1})$,
\[
 \norm{\nabla P^{(j)}}\le C\left(\sqrt M+\frac{1}{D}\right),
 \qquad \norm{\nabla^2P^{(j)}}\le C
\]
for a numerical constant $C>0$.
The input tail enters through
$-8\eta_0^{(j)}(e_1(a_1^{(j)})-e_0)$, whose coefficient is bounded.
The mixed controller-primal block is $(2/D)\nabla P^{(j)}$.
Since $0\le\alpha_j\le1$, $M\le D^2$, and $D\ge1$,
\[
 \norm{\nabla^2(\alpha_jP^{(j)})}
 \le\norm{\nabla^2P^{(j)}}+\frac{2}{D}\norm{\nabla P^{(j)}}
 \le5C.
\]

The bounds $\norm{\nabla w_j}\le C/D$,
$\norm{\nabla^2w_j}\le C/D^2$, and $\norm{\y^{(j)}}\le D/2$
also give a numerical Hessian bound for
$3w_ju_j-(1-w_j)\norm{\y^{(j)}}^2/2$.
The inner-chain and regularizer terms have uniformly bounded
Hessians by Lemmas~\ref{gs:mem:matrix-bounds}
and~\ref{gs:lem:regularizers}.
For each ordering term, the factor $D^2$ cancels the two
derivatives of the rescaled arguments.

Group each batch with its ordering and ramp terms when $j<K$.
Each group depends only on $(\x^{(j)},\y^{(j)},t_{j-1})$,
omitting the fixed $t_0$, and has Hessian norm at most a numerical
constant $c_{\ell,0}\ge1$.
Groups of the same parity involve disjoint coordinates, so the
odd and even sums each have a block diagonal Hessian of norm
at most $c_{\ell,0}$. Consequently,
\[
 \norm{\nabla^2\widetilde g_3}\le c_{\ell,0}+c_{\ell,0}=2c_{\ell,0}=:c_{\ell},
\]
independently of $K,M,N,D$.

For dual concavity, we examine the dual Hessian.
All dual terms except the quadratic penalties are affine in $\y$.
Since $0\le w_j\le1$,
\[
 \nabla_{\y\y}^2\widetilde g_3
 =-\operatorname{diag}_{j\in[K]}
       \bigl((1-w_j)\bI_{MN+1}\bigr)
 \preceq0.
\]

Finally, we bound the initial gap.
At the primal origin, $P^{(j)}=0$ for every $j$, $w_1=1$, and
$w_j=0$ for $j>1$, so $\widetilde\Phi_3(\bz)=3D/2$.
Taking the feasible dual point $\y=\bz$, using $\widetilde C_{\rm st}\ge0$,
$R_1\ge-S_0$, and the $K-1$ ramp bounds in
\eqref{gs:batch:ramp-bounds}, gives
\[
 \widetilde\Phi_3(\x)\ge-S_0KM-R_0(K-1)D^2.
\]
Since $M\le D^2$ and $D\ge1$, choosing $c_B\ge3/2+S_0+R_0$ gives
\begin{equation}
 \widetilde\Phi_3(\bz)-\inf_{\x\in\R^{KM(N+1)}}\widetilde\Phi_3(\x)
 \le\frac{3D}{2}+S_0KM+R_0(K-1)D^2
 \le c_BKD^2.
 \label{gs:eq:unified-gap}
\end{equation}

\end{proof}

\subsubsection{Zero-chain}
\begin{proof}[Proof of Lemma~\ref{gs:lem:batched-instance}~\textup{(ii)}]
Fix a feasible controller vector $\u$ and a query supported on a proper
prefix, with
$\supp_{\ord}(\x,\w)\subseteq[r]$ in the order~\eqref{gs:batch:order}.
Let $j_*$ be the batch containing coordinate $r+1$.

The inner chains reveal at most the next coordinate by
Lemma~\ref{gs:mem:matrix-bounds}. For the stage couplings, direct
differentiation gives, for every $a\in\R$,
\[
 \nabla\widetilde C_{\rm ex}(0,a)=\bz,\qquad
 \nabla\widetilde C_{\rm en}(0,0)=\bz,\qquad
 R_1'(0)=e_2'(0)=0,
\]
where each coupling gradient is taken in both arguments.
Thus an exit coupling cannot reveal a state before its inner chain
reaches it, and the entrance to each subsequent stage remains inactive
until its preceding state is reached.

For the first head of a batch, differentiating
\eqref{gs:batch:potential} gives
\[
 (\x^{(j)},\w^{(j)})=\bz
 \quad\Longrightarrow\quad
 \nabla_{a_1^{(j)}}\widetilde g_3
 =-8\alpha_j\eta_0^{(j)}.
\]
For $j>j_*$, the input tail $t_{j-1}$ is zero, so
$\eta_0^{(j)}=q(0)=0$ and this derivative vanishes.
Thus a first head outside the prefix can be revealed only if it is
coordinate $r+1$.

For $j>1$, the derivative through the input activation is
\[
 \nabla_{t_{j-1}}P^{(j)}
 =-\frac{8}{D}q'\left(\frac{t_{j-1}}{D}\right)
       \bigl(e_1(a_1^{(j)})-e_0\bigr).
\]
This derivative vanishes at an unreached input tail. The remaining
tail terms also have zero derivatives at unreached tails, since
\[
 q'(0)=R_2'(0)=0,\qquad
 \nabla\widetilde C_{\rm st}(t,0)=\bz
 \quad(t\in\R).
\]
Finally, the quadratic penalties satisfy
\[
 \nabla_{\w^{(j)}}\left[
 -\frac{1-w_j}{2}\norm{\y^{(j)}}^2
 \right]=-(1-w_j)\w^{(j)},
\]
so they introduce no new dual coordinates. Their derivatives with
respect to unreached tails also vanish by $q'(0)=0$.
Together these observations give
\[
 \supp_{\ord}\bigl(\nabla_{(\x,\w)}\widetilde g_3\bigr)
 \subseteq[r+1],
\]
which proves the zero-chain property for every fixed
feasible controller vector.
\end{proof}

\subsubsection{Nonzero terminal state at GS points}
We first determine the order of the batch tails and derive local
estimates from a small primal gradient. The dual residual then bounds
the full dual vector, allowing us to prove
Lemma~\ref{gs:lem:batched-instance}~\textup{(iii)}.

For a feasible pair $(\x,\y)$, write
$\bm g_x:=\nabla_{\x}\widetilde g_3$, with components
$\bm g_z^{(j,i)}$ and $g_s^{(j,i)}$. Using $\w_0$ from
Section~\ref{gs:sec:fixed-controller}, set
\[
 \begin{gathered}
 r_j:=\norm{\nabla_{\x^{(j)}}\widetilde g_3},\qquad
 \bm\delta^{(j)}:=(\delta_i^{(j)})_{i=1}^M,\qquad
 \delta_i^{(j)}:=\langle\w_0,\bm g_z^{(j,i)}\rangle+g_s^{(j,i)}.
 \end{gathered}
\]
The identity $\bD_N^\top\w_0=\ee_N/\sqrt N$ cancels the inner
terms. Disjoint stage coordinates and $\norm{\w_0}^2<3$ give
\begin{equation}
 \norm{\bm\delta^{(j)}}\le2r_j,\qquad
 \sum_{j=1}^K r_j^2=\norm{\bm g_x}^2.
 \label{batch:eq:sleeping-clock-norm}
\end{equation}

\begin{lemma}[Order of the batch tails]
\label{lem:gs-batch-tail-order}
Let $(\x,\y)$ be feasible, with $t_K=0$ and
$\norm{\nabla_{\x}\widetilde g_3}\le1$.
There is a unique $j_*\in[K]$ such that
\[
 t_j>D\quad(1\le j<j_*),\qquad
 -c_s<t_j<\frac{c_s}{4}\quad(j_*\le j\le K).
\]
Consequently, $w_{j_*}=1$ and $w_j=0$ for $j\ne j_*$, and
\begin{equation}
 \begin{gathered}
 \eta_M^{(j)}=\sigma_M^{(j)}=b_j,\qquad
 w_j=\eta_0^{(j)}-b_j=b_{j-1}-b_j\in\{0,1\},\qquad
 \sum_{j=1}^K w_j=b_0-b_K=1.
 \end{gathered}
 \label{gs:batch:binary}
\end{equation}
All outer-gate and ordering derivatives vanish.
\end{lemma}

\begin{proof}[Proof of Lemma~\ref{lem:gs-batch-tail-order}]
For $j<K$, write $t=t_j$. Its own batch potential and state
regularizer contribute
\begin{equation}
 8\alpha_j\left[-1+q_{\rm ex}(t)
 +(1-\eta_{M-1}^{(j)})e_1'(a_M^{(j)})
 +q_{\rm ex}'(t)\{t+e_1(h_M^{(j)})-e_0\}\right]+R_1'(t)
 \label{gs:batch:tail-clock}
\end{equation}
to $\delta_M^{(j)}$. This is at most $-24$ for $t\le-c_s$,
and at most a numerical constant $B_0$ for $t\in[c_s/4,D]$,
since $q_{\rm ex}'$ is supported in $[c_s,2c_s]$ and $e_1$ is bounded.

The only additional positive contributions come from the next
potential, the next controller bias and damping, and the incoming
ordering term. They vanish outside $[D/5,D]$ and, on this interval,
are bounded respectively by the four terms in
\[
 \frac{8e_0\norm{q'}_\infty}{D}
 +\frac{3}{2}\norm{q'}_\infty
 +\frac{D}{8}\norm{q'}_\infty
 +D\norm{\nabla_s\widetilde C_{\rm st}}_\infty
 \le B_1t,
\]
where $B_1$ is numerical because $D\ge1$.
The current bias and damping and the outgoing ordering term have
nonpositive derivatives. Choose
$R_0>\max\{B_0+2,5B_1\}$. By~\eqref{gs:batch:ramp-derivative},
\[
 \delta_M^{(j)}\le
 \begin{cases}
  -24, & t\le-c_s,\\
  B_0-R_0<-2, & \frac{c_s}{4}\le t<\frac{D}{5},\\
  B_0-R_0+\left(B_1-\frac{R_0}{5}\right)D<-2,
       & \frac{D}{5}\le t\le D.
 \end{cases}
\]
For the first case, all $q'(t/D)$ and $R_2'(t)$ terms vanish.
For the last case, we used $t\le D$ and
$R_2'(t)=-R_0(1+D/5)$.
These bounds contradict $|\delta_M^{(j)}|\le2r_j\le2$.
Thus, including $t_K=0$, every tail satisfies
\begin{equation}
 t_j\in\left(-c_s,\frac{c_s}{4}\right)\quad\hbox{or}\quad t_j>D.
 \label{gs:batch:phases}
\end{equation}

If a low tail $t_j$ precedes a high tail $t_{j+1}$, then $j+1<K$
because $t_K=0$, so their ordering term is present. Since $D\ge16c_s$,
\[
 q\left(\frac{t_j}{D}\right)=q'\left(\frac{t_j}{D}\right)=0,
 \qquad
 \nabla_r\widetilde C_{\rm st}\left(\frac{t_j}{D},\frac{t_{j+1}}{D}\right)
 =-1.
\]
Its contribution to $\delta_M^{(j)}$ is $-D$; the core contribution
and $R_2'(t_j)$ are nonpositive, and all others vanish. Hence
$\delta_M^{(j)}\le-D<-2$, again a contradiction.
The tails therefore form a high prefix followed by a low suffix.
Using $b_0=1$, $b_K=0$, and $t_K=0$, we obtain
\eqref{gs:batch:binary}.
All outer-gate and ordering derivatives vanish on these ranges.
\end{proof}

For the local estimates, set $d_e(v):=1-e_1'(v)$ and define,
for each batch $j$,
\[
 \nu_j:= \#\{i\in[M]\mid\eta_i^{(j)}<1,\ \eta_{i-1}^{(j)}=1\},\qquad
 \beta_i^{(j)}
 := \sigma_i^{(j)}d_e(h_i^{(j)})
      +(1-\sigma_i^{(j)})d_e(a_i^{(j)}).
\]
Using $\vvec$ from Lemma~\ref{gs:mem:matrix-bounds}, write
\begin{equation}
 \bm p^{(j,i)}:=8\alpha_j\beta_i^{(j)}\vvec,\qquad
 \bm\xi^{(j,i)}:=\w^{(j,i)}-\bm p^{(j,i)}.
 \label{batch:eq:sleeping-response}
\end{equation}
We use $\bm\beta^{(j)},\bm p^{(j)},\bm\xi^{(j)}$ for the corresponding
stacked vectors.

\begin{lemma}[Local batch estimates]
\label{lem:gs-batch-local-estimates}
Under the hypotheses of Lemma~\ref{lem:gs-batch-tail-order},
every batch $j$ satisfies $s_i^{(j)}>-c_s$ for all $i\in[M]$.
Its internal residuals, $i<M$, satisfy~\eqref{gs:eq:clock-residual}
with the corresponding batch variables.
Moreover, $0\le\beta_i^{(j)}\le1$, and
\[
 w_j\le\nu_j,\qquad
 \alpha_j^2\nu_j\le\frac{\norm{\bm\delta^{(j)}}^2}{64}
                    \le\frac{r_j^2}{16}.
\]
The potential and dual response satisfy
\begin{equation}
 P^{(j)}\le36c_sw_j+72c_s\nu_j
      -\frac{32c_s}{15}\norm{\bm\beta^{(j)}}^2
      +8\langle(\bm h^{(j)})_-,\bm\beta^{(j)}\rangle,
 \label{batch:eq:sleeping-energy-upper}
\end{equation}
\begin{equation}
 \norm{\bm\xi^{(j)}}\le6Nr_j,\qquad
 \norm{\bm p^{(j)}}=8\alpha_jN\norm{\bm\beta^{(j)}}
                  \le8\alpha_jN\sqrt M.
 \label{batch:eq:sleeping-error-bound}
\end{equation}
\end{lemma}

\begin{proof}[Proof of Lemma~\ref{lem:gs-batch-local-estimates}]
Lemma~\ref{lem:gs-batch-tail-order} makes all outer-gate and ordering
derivatives vanish. Hence
\begin{align}
 \bm g_z^{(j,i)}
 =\ & \bD_N\w^{(j,i)}
       +\alpha_j\nabla_{a_i^{(j)}}P^{(j)}\,\ee_1,
 \label{batch:eq:sleeping-head-gradient}\\
 g_s^{(j,i)}
 =\ & -\frac{w_N^{(j,i)}}{\sqrt N}
       +\alpha_j\nabla_{s_i^{(j)}}P^{(j)}
       +R_1'(s_i^{(j)})+k_i^{(j)},
 \label{batch:eq:sleeping-state-gradient}
\end{align}
where the derivatives of $P^{(j)}$ hold $\eta_0^{(j)}$ fixed and
\[
 k_i^{(j)}=
 \begin{cases}
 R_2'(t_j)\le0,&i=M,\ j<K,\\
 0,&\text{otherwise}.
 \end{cases}
\]
Fix a batch $j$ and suppress its superscript, writing $P=P^{(j)}$
and $\alpha=\alpha_j$, $\nu=\nu_j$. Set $E_{M+1}:=0$ and
$E_{i+1}:=e_1(a_{i+1})$ for $i<M$.
Equations~\eqref{batch:eq:sleeping-head-gradient}
and~\eqref{batch:eq:sleeping-state-gradient} give
\begin{equation}
 \delta_i
 \le 8\alpha\Bigl[
  -1+\sigma_i+\sigma_i'\bigl(s_i+e_1(h_i)\bigr)
  +(1-\eta_{i-1})e_1'(a_i)-\eta_i'E_{i+1}\Bigr]+R_1'(s_i).
 \label{batch:eq:sleeping-clock-upper}
\end{equation}
For $i<M$, this is~\eqref{gs:eq:clock-residual}.
For $i=M$, the omitted terms are
$-8\alpha e_0\sigma_M'+k_M^{(j)}\le0$.
If $s_i\le-c_s$, the gates and their derivatives vanish, so
\[
 \delta_i
 \le8\alpha\bigl[-1+(1-\eta_{i-1})e_1'(a_i)\bigr]-24
 \le-24.
\]
This contradicts $|\delta_i|\le2r_j\le2$; hence $s_i>-c_s$.

Define
\[
 \mathcal U:=\{i\in[M]\mid\eta_i<1\},\qquad
 \mathcal H:=\{i\in[M]\mid\eta_i=1\}.
\]
For $i\in\mathcal U$, $\sigma_i=\sigma_i'=0$, and
\eqref{batch:eq:sleeping-clock-upper} gives
\[
 \delta_i\le
 -8\alpha\bigl[1-(1-\eta_{i-1})e_1'(a_i)+\eta_i'E_{i+1}\bigr]
 -8c_S\sqrt{\eta_i}.
\]
The bracket is nonnegative, and is at least one if $\eta_{i-1}=1$.
Squaring the resulting bounds and summing gives
\begin{equation}
 \sum_{i\in\mathcal U}\eta_i
 \le\frac{\norm{\bm\delta}^2}{64c_S^2},
 \qquad
 \alpha^2\nu\le\frac{\norm{\bm\delta}^2}{64}
 \le\frac{r_j^2}{16}.
 \label{batch:eq:sleeping-run-bound}
\end{equation}
The binary endpoint values satisfy $\eta_0-\eta_M=w_j$.
Counting transitions between $\eta_i<1$ and $\eta_i=1$ therefore yields
\[
 \#\{i\in\mathcal H\mid\eta_{i-1}<1\}=\nu-w_j,
 \qquad 0\le w_j\le\nu.
\]

For $i\in\mathcal H$, $\eta_i'=0$ and
$\sigma_i'(s_i+e_1(h_i))\le44\sqrt{1-\sigma_i}$,
since $\sigma_i'\ne0$ implies $s_i\le2c_s$.
Using $c_S=192$ and $\alpha\le1$ in
\eqref{batch:eq:sleeping-clock-upper} gives
\[
 \delta_i\le8\alpha(1-\eta_{i-1})-16\sqrt{1-\sigma_i},
 \qquad
 1-\sigma_i\le\frac{\alpha^2}{2}(1-\eta_{i-1})^2
                  +\frac{\delta_i^2}{128}.
\]
Only the $\nu-w_j$ upward transitions contribute to
$1-\eta_{i-1}$ on $\mathcal H$, so
\[
 \sum_{i\in\mathcal H}(1-\sigma_i)
 \le\frac{\alpha^2\nu}{2}+\frac{\norm{\bm\delta}^2}{128}
 \le\frac{\norm{\bm\delta}^2}{64}.
\]

Set $\gamma_i:=\sigma_i-\eta_{i-1}$. On $\mathcal U$,
$\gamma_i=-\eta_{i-1}$; on $\mathcal H$,
$\gamma_i=(1-\eta_{i-1})-(1-\sigma_i)$.
Thus
\[
 \begin{aligned}
 \sum_{i\in\mathcal U}\gamma_i^2
 \le\nu+\sum_{i\in\mathcal U}\eta_i,\qquad\sum_{i\in\mathcal H}\gamma_i^2
 \le\nu-w_j+\sum_{i\in\mathcal H}(1-\sigma_i),\qquad
 \sum_i(\gamma_i)_+\le \nu-w_j\le\nu.
 \end{aligned}
\]
Here the first inequality uses distinct predecessors, and the second uses
$(x-y)^2\le x^2+y^2$ for $x,y\ge0$.
Together with~\eqref{batch:eq:sleeping-run-bound}, these estimates give
\begin{equation}
 \norm{\bm\gamma}^2\le2\nu+\frac{\norm{\bm\delta}^2}{32},
 \qquad
 \sum_i(\gamma_i)_+\le\nu,
 \qquad
 8\alpha\norm{\bm\gamma}\le2\norm{\bm\delta}.
 \label{batch:eq:sleeping-mismatch-bound}
\end{equation}
The last inequality uses $\alpha\le1$:
\[
 64\alpha^2\norm{\bm\gamma}^2
 \le128\alpha^2\nu+2\alpha^2\norm{\bm\delta}^2
 \le4\norm{\bm\delta}^2.
\]

Set $v_e(v):=e_1(v)-e_0-v$.
For either $v=h_i$ or $v=a_i$, we claim that
\[
 v_e(v)\le-\frac{4c_s}{15}d_e(v)^2+(h_i)_-d_e(v).
\]
For $v\ge0$, this follows from~\eqref{gs:sc:tail-bounds};
when $d_e(v)=0$, both defects vanish.
For $v<0$ with $d_e(v)>0$, monotonicity of $d_e$,
$s_i>-c_s$, and $a_i=h_i+s_i$ give
\[
 v_e(v)\le (-v-3c_s)d_e(v)
 \le\bigl((h_i)_-2c_s\bigr)d_e(v)
 \le(h_i)_-d_e(v)-\frac{4c_s}{15}d_e(v)^2,
\]
where the last step uses $0\le d_e\le1$.
Convexity of the square then gives
\[
 \sigma_i v_e(h_i)+(1-\sigma_i)v_e(a_i)
 \le-\frac{4c_s}{15}\beta_i^2+(h_i)_-\beta_i.
\]
Using $\eta_0-\sigma_M=w_j$, expand~\eqref{gs:batch:potential} as
\[
 \frac{P}{8}=w_je_0+
 \sum_i\bigl[
  \sigma_i v_e(h_i)+(1-\sigma_i)v_e(a_i)+\gamma_i e_1(a_i)\bigr].
\]
Since $0\le e_1\le9c_s$ and $\sum_i(\gamma_i)_+\le\nu$,
restoring the batch superscripts proves~\eqref{batch:eq:sleeping-energy-upper}.
Since $w_j\le\nu_j$,
$36c_sw_j+72c_s\nu_j\le108c_s\nu_j$.

Differentiation gives
$-(\nabla_{a_i}P)/8=\beta_i-\gamma_i e_1'(a_i)$.
Equation~\eqref{batch:eq:sleeping-head-gradient} therefore implies
\[
 \bm\xi^{(j,i)}=\bD_N^{-1}\bm g_z^{(j,i)}
           -8\alpha_j\gamma_i^{(j)}e_1'(a_i^{(j)})\vvec.
\]
Using $\norm{\vvec}=N$, $\norm{\bD_N^{-1}}\le\sqrt2N$,
$0\le e_1'\le1$, and~\eqref{batch:eq:sleeping-mismatch-bound}, we obtain
\[
 \begin{aligned}
 \norm{\bm\xi^{(j)}}
 \le\ & \sqrt2Nr_j+8\alpha_jN\norm{\bm\gamma^{(j)}}
 \le(\sqrt2+4)Nr_j\le6Nr_j,\\
 \norm{\bm p^{(j)}}=\ & 8\alpha_jN\norm{\bm\beta^{(j)}}
 \le8\alpha_jN\sqrt M.
 \end{aligned}
\]
This proves~\eqref{batch:eq:sleeping-error-bound}.
\end{proof}

For the next two lemmas, let $(\x,\y)$ be a $\tau$-GS point
with $t_K=0$ and $0<\tau\le1/4$.
The normal cone of $\widehat{\cY}_D$ gives a common multiplier
$\lambda\ge0$, multipliers $\zeta_j\ge0$ with $\zeta_ju_j=0$,
and errors satisfying
\begin{equation}
 \nabla_{u_j}\widetilde g_3=\lambda u_j-\zeta_j+e_{u,j},\qquad
 \nabla_{\w^{(j)}}\widetilde g_3
 = \lambda\w^{(j)}+\bm e^{(j)},
 \qquad\sum_{j=1}^K\rho_j^2
 = \norm{(\bm e_u,\bm e)}^2\le\tau^2,\label{batch:eq:dual-normal-cone}
\end{equation}
where $
 \bm e^{(j)}:=(\bm e^{(j,i)})_{i=1}^M,\mbox{ and }
 \rho_j^2:=|e_{u,j}|^2+\norm{\bm e^{(j)}}^2.$ Here $\lambda=0$ unless $\norm\y=D/2$, and
$\sum_jr_j^2\le\tau^2$.
Set
\[
 J:=[K],\qquad \mu_j:=\lambda+1-w_j\ge0,\qquad
 T_j:=\norm{\bm p^{(j)}}=8\alpha_jN\norm{\bm\beta^{(j)}}.
\]
Subscripts on stacked vectors indicate the batches being concatenated.
Substituting the dual gradients into~\eqref{batch:eq:dual-normal-cone}
gives
\begin{align}
 3w_j+\frac{2P^{(j)}}{D}
 =\ & \mu_ju_j+e_{u,j}\quad\text{if }u_j>0,
 \label{batch:eq:sleeping-anchor-normal}\\
 \bD_N^\top\z^{(j,i)}-\frac{s_i^{(j)}\ee_N}{\sqrt N}
 =\ & \mu_j\w^{(j,i)}+\bm e^{(j,i)}.
 \label{batch:eq:sleeping-physical-normal}
\end{align}

\begin{lemma}[Estimates from the dual residual]
\label{lem:gs-batch-dual-energy}
Let $(\x,\y)$ be a $\tau$-GS point with $t_K=0$ and
$0<\tau\le1/4$. With the notation above, every batch $j$ satisfies
\begin{equation}
 8\alpha_j\langle(\bm h^{(j)})_-,\bm\beta^{(j)}\rangle
 \le\frac{\mu_j}{4}\norm{\bm\xi^{(j)}}^2+T_j\rho_j
 \label{batch:eq:dual-residual-charge}
\end{equation}
and
\begin{equation}
 \begin{aligned}
 \frac{c_s}{30N^2}T_j^2+\alpha_j\mu_ju_j^2
 \le\ & \left(\frac{27c_s}{4}+\frac{3D}{32}\right)r_j^2+\frac{\alpha_j\mu_j}{4}\norm{\bm\xi^{(j)}}^2
       +\alpha_j(T_j+u_j)\rho_j.
 \end{aligned}
 \label{batch:eq:sleeping-key-energy}
\end{equation}
\end{lemma}

\begin{proof}[Proof of Lemma~\ref{lem:gs-batch-dual-energy}]
Taking the inner product of
\eqref{batch:eq:sleeping-physical-normal} with $\vvec$ gives
\[
 h_i^{(j)}=\langle\vvec,\mu_j\w^{(j,i)}+\bm e^{(j,i)}\rangle.
\]
On the indices where $h_i^{(j)}<0$, substitute
\eqref{batch:eq:sleeping-response} and complete the square to obtain
\[
 8\alpha_j\langle(\bm h^{(j)})_-,\bm\beta^{(j)}\rangle
 = -\mu_j\sum_{h_i^{(j)}<0}
   \langle\bm p^{(j,i)},\bm p^{(j,i)}+\bm\xi^{(j,i)}\rangle-\sum_{h_i^{(j)}<0}
   \langle\bm p^{(j,i)},\bm e^{(j,i)}\rangle
 \le \frac{\mu_j}{4}\norm{\bm\xi^{(j)}}^2+T_j\rho_j.
\]
Here $-\norm{\bm p}^2-\langle\bm p,\bm\xi\rangle
\le\norm{\bm\xi}^2/4$, and the residual term is bounded by
Cauchy-Schwarz. This proves~\eqref{batch:eq:dual-residual-charge}.

For $u_j>0$, the controller equation gives
\[
 \alpha_jP^{(j)}=\mu_ju_j^2-3w_ju_j+u_je_{u,j}.
\]
Moreover, $w_j\le\nu_j$ and~\eqref{batch:eq:sleeping-run-bound} imply
\[
 108c_s\alpha_j^2\nu_j\le \frac{27c_s}{4}r_j^2,\qquad
 3w_j\alpha_ju_j
 = \frac{3D}{2}w_j\alpha_j^2
 \le\frac{3D}{2}\nu_j\alpha_j^2
 \le\frac{3D}{32}r_j^2.
\]
Multiplying~\eqref{batch:eq:sleeping-energy-upper} by $\alpha_j^2$
and using~\eqref{batch:eq:dual-residual-charge} gives
\[
 \begin{aligned}
 \frac{c_s}{30N^2}T_j^2+\alpha_j\mu_ju_j^2
 \le\ & 108c_s\alpha_j^2\nu_j+3w_j\alpha_ju_j+\frac{\alpha_j\mu_j}{4}\norm{\bm\xi^{(j)}}^2
       +\alpha_j(T_j+u_j)\rho_j.
 \end{aligned}
\]
The preceding two bounds yield~\eqref{batch:eq:sleeping-key-energy}.
When $u_j=0$, both $\alpha_j$ and $T_j$ vanish, so the same
inequality holds.
\end{proof}

\begin{lemma}[Bound on the dual norm]
\label{lem:gs-batch-dual-norm}
Let $(\x,\y)$ be a $\tau$-GS point with $t_K=0$ and
$0<\tau\le1/4$. Then
\begin{equation}
 \norm\y\le70D\sqrt{\tau}.
 \label{batch:eq:sleeping-final-bound}
\end{equation}
\end{lemma}

\begin{proof}[Proof of Lemma~\ref{lem:gs-batch-dual-norm}]
Use the quantities from Lemmas~\ref{lem:gs-batch-local-estimates}
and~\ref{lem:gs-batch-dual-energy}.
For $\nu_j\ge1$, \eqref{batch:eq:sleeping-run-bound} gives
\[
 \norm{\u_{\nu\ge1}}^2
 \le\frac{D^2}{64}\sum_jr_j^2
 \le\frac{D^2\tau^2}{64}.
\]
For $\nu_j=0$, we have $w_j=0$ and
$\mu_j=1+\lambda=:\mu\ge1$. Multiplying
\eqref{batch:eq:sleeping-energy-upper} by $\alpha_j$ and using
\eqref{batch:eq:dual-residual-charge} together with the controller
equation~\eqref{batch:eq:sleeping-anchor-normal} gives
\[
 \mu u_j^2
 =\alpha_jP^{(j)}-u_je_{u,j}
 \le\frac{\mu}{4}\norm{\bm\xi^{(j)}}^2+(T_j+u_j)\rho_j.
\]
This inequality is trivial when $u_j=0$.
Set $A_0:=\norm{\u_{\nu=0}}$. Feasibility and
\eqref{batch:eq:sleeping-error-bound} give
\[
 \norm{\bm\xi_J}\le6N\tau,\qquad
 \norm{\bm p_J}\le\norm{\w_J}+\norm{\bm\xi_J}
 \le\frac{D}{2}+6N\tau.
\]
Summing the controller inequality and dividing by $\mu\ge1$ yields
\[
 A_0^2
 \le9N^2\tau^2+\left(\frac{D}{2}+6N\tau\right)\tau+A_0\tau.
\]
Using $A_0\tau\le(A_0^2+\tau^2)/2$ and $N\ge2$, we obtain
\[
 A_0^2\le(18N^2+12N+1)\tau^2+D\tau
 \le25N^2\tau^2+D\tau.
\]
Since $N^2\le D/2$, $\tau\le1$, and $D\ge1$,
\begin{equation}
 \norm{\u_J}^2
 \le 25N^2\tau^2+D\tau+\frac{D^2\tau^2}{64}
 \le25D^2\tau,\qquad
 \norm{\u_J}\le 5D\sqrt{\tau}.
 \label{batch:eq:sleeping-anchor-sum}
\end{equation}

For the inner dual variables, set
$J_L:=\{j\mid u_j\ge\norm{\bm\xi^{(j)}}\}$ and $J_S:=J\setminus J_L$.
For $j\in J_L$, \eqref{batch:eq:sleeping-key-energy} implies
\[
 \frac{c_s}{30N^2}T_j^2+\frac{3}{4}\alpha_j\mu_ju_j^2
 \le\left(\frac{27c_s}{4}+\frac{3D}{32}\right)r_j^2
       +\alpha_j(T_j+u_j)\rho_j.
\]
Drop the nonnegative controller term and sum. Since $\alpha_j\le1$,
\[
 \begin{aligned}
 \frac{c_s}{30N^2}\norm{\bm p_{J_L}}^2
 \le\ & \left(\frac{27c_s}{4}+\frac{3D}{32}\right)\tau^2
       +\bigl(\norm{\bm p_J}+\norm{\u_J}\bigr)\tau\\
 \le\ & \left(\frac{27c_s}{4}+\frac{3D}{32}\right)\tau^2
       +(D+6N\tau)\tau.
 \end{aligned}
\]
Here we used the feasibility bound $\norm{\u_J}\le D/2$.
With $c_s=64$, $N^2\le D/2$, $N\le D$, and $\tau\le1$, this gives
\[
 \norm{\bm p_{J_L}}^2\le121D^2\tau,
 \qquad \norm{\bm p_{J_L}}\le11D\sqrt{\tau}.
\]
For $j\in J_S$, use $u_j<\norm{\bm\xi^{(j)}}\le6Nr_j$ and $M\le D^2$:
\[
 T_j\le\frac{16N\sqrt M\,u_j}{D}
 \le96N^2r_j\le48Dr_j,
 \qquad
 \norm{\bm p_{J_S}}\le48D\tau.
\]
The response decomposition now gives
\[
 \begin{aligned}
 \norm{\w_J}
 \le\ & \norm{\bm p_{J_L}}+\norm{\bm p_{J_S}}+\norm{\bm\xi_J}
 \le11D\sqrt{\tau}+54D\tau\le65D\sqrt{\tau},\\
 \norm\y
 \le\ & \norm{\u_J}+\norm{\w_J}\le70D\sqrt{\tau},
 \end{aligned}
\]
proving~\eqref{batch:eq:sleeping-final-bound} independently of $K$.
We divided by $\mu_j$ only when $\nu_j=0$, where
$\mu_j=1+\lambda\ge1$. All other estimates use only $\mu_j\ge0$
and remain valid when $\lambda=0$.
\end{proof}

\begin{proof}[Proof of Lemma~\ref{gs:lem:batched-instance}~\textup{(iii)}]
Choose $\varepsilon_0=7000^{-2}$. Suppose that $(\x,\y)$ is a
$2\varepsilon_0$-GS point with $t_K=0$, and set
$\tau:=2\varepsilon_0<1/4$. Lemma~\ref{lem:gs-batch-dual-norm} gives
\[
 \norm\y\le70D\sqrt{2\varepsilon_0}
 =\frac{\sqrt2D}{100}<\frac{D}{2},
 \qquad \lambda=0.
\]
By Lemma~\ref{lem:gs-batch-tail-order}, there is a unique batch $j_*$ with
$\eta_0^{(j_*)}=1$, $b_{j_*}=0$, and $w_{j_*}=1$.
Since $\mu_{j_*}=0$, \eqref{batch:eq:sleeping-physical-normal} gives
\[
 h_i^{(j_*)}=\langle\vvec,\bm e^{(j_*,i)}\rangle,
 \qquad\norm{\bm h^{(j_*)}}\le N\tau.
\]
Lemma~\ref{lem:gs-batch-local-estimates} gives $s_i^{(j_*)}>-c_s$
and~\eqref{gs:eq:clock-residual} for the internal residuals.
The tail-order lemma gives $\eta_0^{(j_*)}=1$ and
$\eta_M^{(j_*)}=\sigma_M^{(j_*)}=0$.
We may therefore apply Lemma~\ref{gs:bnd:energy}, setting its unused
last residual to zero; the remaining vector has norm at most $2\tau$.
It follows that
\[
 P^{(j_*)}\ge-4c_s-\frac{120N^2\tau^2}{c_s}
                  -\frac{c_s}{800}(2\tau)^2.
\]
Using $N^2\le D/2$, $D\ge16c_s$, $c_s=64$, and $\tau<1/4$,
we obtain
\[
 \begin{aligned}
 \nabla_{u_{j_*}}\widetilde g_3
 = 3+\frac{2P^{(j_*)}}{D}
 \ge3-\frac{8c_s}{D}
       -\frac{240N^2\tau^2}{c_sD}
       -\frac{c_s\tau^2}{100D}>2.
 \end{aligned}
\]
But~\eqref{batch:eq:dual-normal-cone} gives
\[
 \nabla_{u_{j_*}}\widetilde g_3
 =-\zeta_{j_*}+e_{u,j_*}\le\tau,
\]
including when $u_{j_*}=0$, since $\zeta_{j_*}\ge0$.
This contradiction proves that every $2\varepsilon_0$-GS point has $t_K\ne0$.
\end{proof}

\subsection{Scaling and proof of the GS lower bound}
\label{gs:subsec:scaling}
We now choose the parameters of $g_2$ and $g_3$ in
\eqref{gs:eq:scaled-instance-2} and~\eqref{gs:eq:scaled-instance-3}
to obtain the lower bounds in Table~\ref{tab:gs-regimes}.
For $g_2$, take $c_L,c_{\Phi}$ from
Lemma~\ref{gs:lem:fixed-controller-instance} and define
\[
 h:=\frac{c_L\epsilon}{\ell},\qquad
 D:=\frac{D_{\cY}}{h},\qquad
 a_f:=\frac{\ell h^2}{c_L}.
\]

\begin{lemma}[Combined lower bound from $g_1$ and $g_2$]
\label{gs:lem:fixed-controller-lower}
There is a numerical constant $c_{\rm f}>0$ such that, for all
$\ell,D_{\cY},\Delta,\epsilon>0$, the instances $g_1$ and $g_2$ give
\begin{equation}
 \cT^{\GS}_{\epsilon}
   (\cA_{\rm det},\cF(\ell,D_{\cY},\Delta))
 \ge c_{\rm f}\left[
  1+\frac{\ell\Delta}{\epsilon^2}
  +\frac{\ell D_{\cY}}{\epsilon}
       \sqrt{1+\frac{\ell\Delta}{\epsilon^2}}\right].
 \label{gs:fixed:uniform-lower}
\end{equation}
\end{lemma}

\begin{proof}[Proof of Lemma~\ref{gs:lem:fixed-controller-lower}]
Let $\mathcal K$ denote the complexity on the left side of
\eqref{gs:fixed:uniform-lower}.
For $M,N,D$ satisfying Lemma~\ref{gs:lem:fixed-controller-instance},
set
\begin{equation}
 L:=
 \begin{cases}
 M,&N=1,\\
 \left\lfloor\frac{2MN+M-1}{4}\right\rfloor,&N\ge2.
 \end{cases}
 \label{gs:fixed:queries}
\end{equation}
The scaled function $g_2$ is jointly $\ell$-smooth and concave in $\y$.
For $N=1$, append the unused dual coordinates to the primal
order~\eqref{gs:fixed:primal-order};
for $N\ge2$, use the joint order~\eqref{gs:fixed:order}.
In either case, $g_2$ is a zero-chain in the resulting order.
Since GS residuals scale by $a_f/h=\epsilon$, every $2\epsilon$-GS
point of $g_2$ has a nonzero terminal coordinate as specified in
Lemma~\ref{gs:lem:fixed-controller-instance}~\textup{(iii)}.
Support on the first $L-1$ coordinates forces that coordinate to vanish.
Whenever the scaled initial gap is at most $\Delta$,
Lemma~\ref{lem:deterministic-construction-gs}, applied to $g_2$ with
$k=0$ and this $L$, therefore gives $\mathcal K\ge L$.

The instance $g_1$ gives $\mathcal K\ge1$.
When $\ell\Delta/(c_Lc_{\Phi}\epsilon^2)\ge2$, use $g_2$ with $N=1$ and
\[
 M=\left\lfloor\frac{\ell\Delta}{c_Lc_{\Phi}\epsilon^2}\right\rfloor.
\]
Its scaled gap is at most $a_fc_{\Phi}M\le\Delta$, and the query bound above gives
$\mathcal K\ge M\ge\ell\Delta/(2c_Lc_{\Phi}\epsilon^2)$.
For the remaining parameter choices, this inequality follows from
$\mathcal K\ge1$. Combining the two bounds yields a numerical
$c_{\rm base}>0$ such that
\begin{equation}
 \mathcal K\ge c_{\rm base}
 \left(1+\frac{\ell\Delta}{\epsilon^2}\right).
 \label{gs:eq:baseline-lower}
\end{equation}

If
\[
 \frac{\ell D_{\cY}}{\epsilon}
 <32\sqrt2c_L\sqrt{1+\frac{\ell\Delta}{\epsilon^2}},
\]
then the last term in~\eqref{gs:fixed:uniform-lower} is
at most $32\sqrt2c_L(1+\ell\Delta/\epsilon^2)$, so
\eqref{gs:eq:baseline-lower} proves the claim after decreasing
$c_{\rm f}$.

Otherwise use $g_2$ with bilinear inner chains and
\begin{equation}
 M=1+\left\lfloor
   \frac{\ell\Delta}{c_Lc_{\Phi}\epsilon^2}
   \right\rfloor,\qquad
 N=\left\lfloor\frac{\ell D_{\cY}}
                         {16c_L\epsilon\sqrt{2M}}\right\rfloor.
 \label{gs:fixed:scaling}
\end{equation}
Since $c_Lc_{\Phi}\ge1$, we have
$M\le1+\ell\Delta/\epsilon^2$, so $D/(16\sqrt{2M})\ge2$.
Consequently $N\ge2$ and
\[
 N\ge\frac{D}{32\sqrt{2M}},\qquad
 D\ge16\sqrt2N\sqrt M,\qquad
 M\ge\frac{1+\ell\Delta/\epsilon^2}{1+c_Lc_{\Phi}}.
\]
By~\eqref{gs:fixed:gap}, the scaled initial gap is at most
\[
 a_fc_{\Phi}(M-1)
 =\frac{c_Lc_{\Phi}\epsilon^2}{\ell}(M-1)\le\Delta.
\]
As $MN\ge2$, the same application of
Lemma~\ref{lem:deterministic-construction-gs} gives
\[
 \begin{aligned}
 \mathcal K
 \ge\ & \left\lfloor\frac{2MN+M-1}{4}\right\rfloor
 \ge\frac{MN}{4}
 \ge\frac{\ell D_{\cY}\sqrt M}{128\sqrt2c_L\epsilon}
 \ge\frac{1}{128\sqrt2c_L\sqrt{1+c_Lc_{\Phi}}}
 \frac{\ell D_{\cY}}{\epsilon}
 \sqrt{1+\frac{\ell\Delta}{\epsilon^2}}.
 \end{aligned}
\]
Combining this bound with~\eqref{gs:eq:baseline-lower}, and the
preceding case of small diameter, proves~\eqref{gs:fixed:uniform-lower}
with a numerical $c_{\rm f}>0$.
\end{proof}

\begin{proof}[Proof of Theorem~\ref{gs:thm:deterministic}]
Write
$\mathcal K:=\cT^{\GS}_\epsilon(\cA_{\rm det},
\cF(\ell,D_{\cY},\Delta))$.
Lemma~\ref{gs:lem:fixed-controller-lower} supplies
\eqref{gs:fixed:uniform-lower}.
We show that $g_3$ gives the remaining contribution
$\ell^{3/2}D_{\cY}^{1/2}\Delta/\epsilon^{5/2}$
where it is not already controlled by that bound.

Fix the constants $D_0,c_{\ell},c_B,\varepsilon_0$ from
Lemma~\ref{gs:lem:batched-instance}, with
$D_0\ge\max\{4S_0,16c_s\}$.
Choose the constants used below as follows:
\begin{equation}
\begin{gathered}
 c_2:=\max\{2c_Lc_{\Phi},32\sqrt2c_L\},\qquad
 c_3:=\frac{\varepsilon_0}{c_{\ell}D_0},
 \qquad
 c_4:=\frac{3}{\vartheta\varepsilon_0},\\
 c_5:=\max\left\{1,\frac{S_0}{c_{\ell}},\frac{2c_B}{c_{\ell}}\right\},
 \qquad
 \vartheta:=\min\left\{1,\frac{S_0}{c_{\ell}c_5}\right\}.
\end{gathered}
 \label{gs:eq:regime-constants}
\end{equation}
They satisfy $0<c_3<1<c_4$.
If $\epsilon>c_3\ell D_{\cY}$, then
\[
 \frac{\ell^{3/2}D_{\cY}^{1/2}\Delta}{\epsilon^{5/2}}
 \le c_3^{-1/2}\frac{\ell\Delta}{\epsilon^2}.
\]
If $\Delta<c_4\epsilon D_{\cY}$, then
\[
 \frac{\ell^{3/2}D_{\cY}^{1/2}\Delta}{\epsilon^{5/2}}
 \le\sqrt{c_4}\,\frac{\ell D_{\cY}}{\epsilon}
                       \sqrt{1+\frac{\ell\Delta}{\epsilon^2}}.
\]
Indeed, the square of the ratio of the left side to the last two
factors is at most $\Delta/(\epsilon D_{\cY})$.
Thus $g_1$ and $g_2$ also suffice for the product lower bound
in the first three rows of Table~\ref{tab:gs-regimes}.

For the fourth row, assume
\begin{equation}
 0<\epsilon\le\min\left\{c_3\ell D_{\cY},
                         \frac{\Delta}{c_4D_{\cY}}\right\},
 \label{gs:eq:product-regime}
\end{equation}
and set
\begin{equation}
 h:=\frac{c_{\ell}\epsilon}{\ell\varepsilon_0},\qquad
 D:=\frac{D_{\cY}}{h},\qquad a_f:=\frac{\ell h^2}{c_{\ell}},
 \qquad N:=\left\lfloor\sqrt{\frac{D}{2}}\right\rfloor.
 \label{gs:batch:scaling}
\end{equation}
Then $D\ge D_0$. For $1\le M\le\lfloor D^2\rfloor$ and $K\ge1$,
Lemma~\ref{gs:lem:batched-instance} shows that $g_3$ is jointly
$\ell$-smooth and concave in $\y$, with dual diameter $D_{\cY}$.
It is a zero-chain in the order~\eqref{gs:batch:order}
for every fixed feasible controller vector.
Since GS residuals scale by $a_f/h=\epsilon/\varepsilon_0$,
every $2\epsilon$-GS point of $g_3$ has $s_M^{(K)}\ne0$.
Set $L:=\lfloor KM(2N+1)/4\rfloor$.
Support on the first $L-1$ hidden coordinates forces $s_M^{(K)}=0$.
Whenever the scaled initial gap is at most $\Delta$,
Lemma~\ref{lem:deterministic-construction-gs}, applied to $g_3$ with
$k=K$ and this $L$, therefore gives $\mathcal K\ge L$.

When $\Delta\le c_5\ell D_{\cY}^2$, use $g_3$ with $K=1$ and set
\[
 \Delta_0:=\vartheta\Delta,\qquad
 G:=\frac{\Delta_0}{a_f}
   =\frac{\ell\varepsilon_0^2\Delta_0}{c_{\ell}\epsilon^2},
 \qquad M:=\left\lfloor\frac{G}{2S_0}\right\rfloor.
\]
The definition of $\vartheta$ gives
$\Delta_0\le(S_0/c_{\ell})\ell D_{\cY}^2$, while
\eqref{gs:eq:product-regime} gives $G\ge3D\ge4S_0$.
Consequently,
\[
 M\ge\frac{G}{4S_0}\ge1,\qquad
 \frac{G}{2S_0}\le\frac{D^2}{2},\qquad
 M\le\lfloor D^2\rfloor.
\]
The gap bound in~\eqref{gs:eq:unified-gap} yields
\[
 a_f\left(\frac{3D}{2}+S_0M\right)
 \le a_fG=\Delta_0\le\Delta.
\]
Thus the GS extension of $g_3$ belongs to the prescribed class. Its query
bound and $N\ge\frac12\sqrt{D/2}$ give, for a numerical $c_M>0$,
\[
 \begin{aligned}
 \mathcal K
 \ge\ & \left\lfloor\frac{M(2N+1)}{4}\right\rfloor
 \ge\frac{MN}{4}
 \ge c_M\frac{\ell^{3/2}D_{\cY}^{1/2}\Delta_0}{\epsilon^{5/2}}
 = c_M\vartheta
       \frac{\ell^{3/2}D_{\cY}^{1/2}\Delta}{\epsilon^{5/2}}.
 \end{aligned}
\]

When $\Delta>c_5\ell D_{\cY}^2$, use $g_3$ with
\[
 M:=\lfloor D^2\rfloor,\qquad
 K:=\left\lfloor\frac{c_{\ell}\Delta}{c_B\ell D_{\cY}^2}\right\rfloor.
\]
The definition of $c_5$ implies
$c_{\ell}\Delta/(c_B\ell D_{\cY}^2)>2$, so $K\ge2$ and
$K\ge c_{\ell}\Delta/(2c_B\ell D_{\cY}^2)$.
The GS extension of $g_3$ has gap at most
\[
 a_fc_BKD^2=\frac{c_B\ell D_{\cY}^2}{c_{\ell}}K\le\Delta.
\]
Its query bound gives a numerical $c_K>0$ such that
\[
 \mathcal K\ge\left\lfloor\frac{KM(2N+1)}{4}\right\rfloor
 \ge\frac{KMN}{4}
 \ge c_K\frac{\ell^{3/2}D_{\cY}^{1/2}\Delta}{\epsilon^{5/2}}.
\]
Together the cases give a numerical constant $c_{\rm prod}>0$ such that
\begin{equation}
 \mathcal K\ge c_{\rm prod}
       \frac{\ell^{3/2}D_{\cY}^{1/2}\Delta}{\epsilon^{5/2}}.
 \label{gs:eq:product-lower}
\end{equation}
For $a,b\ge0$, the identity
\[
 (1+a)(1+b)(1+a+b)
 =\bigl(1+a+b\sqrt{1+a}\bigr)^2
  +b\bigl(1+a-\sqrt{1+a}\bigr)^2
\]
and $0\le1+a-\sqrt{1+a}\le a$ imply
\[
 \sqrt{(1+a)(1+b)(1+a+b)}
 \le1+a+b\sqrt{1+a}+a\sqrt b.
\]
Apply this inequality with $a=\ell\Delta/\epsilon^2$ and
$b=\ell D_{\cY}/\epsilon$.
Combining~\eqref{gs:fixed:uniform-lower}
and~\eqref{gs:eq:product-lower} proves~\eqref{gs:eq:uniform-lower}
for a numerical constant $c_{\rm GS}>0$.

Finally, the definition of $c_2$ makes the bracket in
\eqref{gs:fixed:uniform-lower} of order $1$, $a$, and
$b\sqrt{1+a}$ in the first three rows of Table~\ref{tab:gs-regimes},
respectively, and ensures that the corresponding choices
$g_1$, $g_2$ with $N=1$, and $g_2$ with $N\ge2$ are valid.
The earlier comparisons control $a\sqrt b$ outside the fourth row.
In the fourth row, $b\ge c_3^{-1}>1$ and $a\ge c_4b\ge b$, so
the square-root expression is within numerical factors of $a\sqrt b$.
The choices of $g_3$ with $K=1$ or $K\ge2$ above therefore
attain the full lower bound in that row.
\end{proof}

\endgroup

\section{Proof of the Upper Bound}
\label{sec:upper-proof}
We establish stability and initialization, then analyze the common
stopping rule before applying the OS and GS output certificates.

\subsection{Proof of Proposition~\ref{prop:anchor-shift}}
\label{subsec:upper-stability}
\begin{proof}[Proof of Proposition~\ref{prop:anchor-shift}]
Write $\bm u=(\omegavec,\y)$, $\bm u_f=(\omegavec_f,\y_f)$, and set
\[
 \mathbf H:=\operatorname{diag}(\ell^{-1}\bI_m,r_{\y}\bI_n),\qquad
 \mathbf B:=\operatorname{diag}(\alpha_{r_{\y}}\ell^{-1}\bI_m,r_{\y}\bI_n)
 \preceq \mathbf H,
\]
where the last inequality follows from $r_{\y}\le \ell/8$.
For a symmetric positive definite matrix $\mathbf G$,
we use the weighted norm
$\norm{\z}_{\mathbf G}:=(\z^\top\mathbf G\z)^{1/2}$.

We first express the change of proximal center as a linear
perturbation. For $\z'=\z+\bm d$, expanding the proximal objective
gives
\begin{equation}
 P_{\z',r_{\y}}(\omegavec-2\ell\bm d,\y)
 =P_{\z,r_{\y}}(\omegavec,\y)-2\ip{\bm d}{\omegavec}+C,
 \label{eq:upper-conjugate-translation}
\end{equation}
where $C$ is independent of $\bm u$.
Thus, in the original coordinates, the new objective is
$\widetilde P(\bm u)=P_{\z,r_{\y}}(\bm u)-\ip{\bm h}{\bm u}$,
with $\bm h=(2\bm d,0)$.
Let $\bm u^\star,\widetilde{\bm u}^\star$ be the respective minimizers.
Strong convexity in the $\mathbf H$-metric gives
\[
 \norm{\widetilde{\bm u}^\star-\bm u^\star}_{\mathbf H}^2
 \le\norm{\bm h}_{\mathbf H^{-1}}^2=4\ell\norm{\bm d}^2.
\]
Moreover, completing the square in the lower bound from strong convexity,
and then using Young's inequality, gives
\begin{align*}
 \widetilde P(\bm u_f)-\min_{\bm u\in\R^m\times\cY}\widetilde P(\bm u)
 \le\ & P_{\z,r_{\y}}(\bm u_f)-P_{\z,r_{\y}}^\star
     -\ip{\bm h}{\bm u_f-\bm u^\star}
     +\frac{1}{2}\norm{\bm h}_{\mathbf H^{-1}}^2\\
 \le\ & 2\bigl(P_{\z,r_{\y}}(\bm u_f)-P_{\z,r_{\y}}^\star\bigr)
     +\norm{\bm h}_{\mathbf H^{-1}}^2.
\end{align*}
Since
$2\norm{\bm u-\widetilde{\bm u}^\star}_{\mathbf B}^2
 \le4\norm{\bm u-\bm u^\star}_{\mathbf B}^2+4\norm{\bm h}_{\mathbf H^{-1}}^2$,
the energy satisfies
$\mathcal L_{\z',r_{\y}}(\widetilde S)
 \le2\mathcal L_{\z,r_{\y}}(S)+6\norm{\bm h}_{\mathbf H^{-1}}^2$,
which proves~\eqref{eq:anchor-tracking}.  The subsequent
contraction bound follows by applying
Lemma~\ref{prop:foam-interface} at the new center. 
\end{proof}

\subsection{Proof of Proposition~\ref{prop:warm-start}}
\label{subsec:warm-start}

We construct an initial state at a sequence of geometrically
decreasing regularization levels. The proof uses two ingredients:
a constant-cost initialization at the largest regularization level,
and a bound on the Lyapunov energy when the regularization parameter
is reduced.

\begin{lemma}[Initial \FOAM{} state]
\label{lem:startup-state}
At $\z^0=\bz$ and $r_{\y}^0=\ell/8$, a state $\widehat S^0$ satisfying
\[
 \mathcal L_{\z^0,r_{\y}^0}(\widehat S^0)
 \le5\left(
 \Phi(\z^0)-\inf_{\x\in\cX}\Phi(\x)
 +\frac{3}{2}r_{\y}^0D_{\cY}^2
 \right)
\]
can be constructed using $\cO(1)$ feasible oracle calls.
\end{lemma}

The proof is given in Appendix~\ref{app:base-initialization}.
In particular, using the definition of $\Delta$, this gives
\[
 \mathcal L_{\z^0,r_{\y}^0}(\widehat S^0)
 \le 8\left(\Delta+r_{\y}^0D_{\cY}^2\right).
\]

\begin{lemma}[Transfer between regularization levels]
\label{lem:curvature-transfer}
For $\z\in\cX$, $0<r_{\y}\le\ell/8$, and any state $S$ with finite energy,
\begin{equation}
 \mathcal L_{\z,r_{\y}/4}(S)
 \le \mathcal L_{\z,r_{\y}}(S)
       +\frac{27}{4}r_{\y}D_{\cY}^2.
 \label{eq:upper-regularization-comparison}
\end{equation}
\end{lemma}

\begin{proof}[Proof of Lemma~\ref{lem:curvature-transfer}]
Write $r_{\y}'=r_{\y}/4$ and
$\mathbf H'=\operatorname{diag}(\ell^{-1}\bI_m,r_{\y}'\bI_n)$.
Let $\bm u^\star$ and $\bm u'^\star$ minimize
$P_{\z,r_{\y}}$ and $P_{\z,r_{\y}'}$, respectively.
The two objectives satisfy
\begin{equation}
 P_{\z,r_{\y}'}(\omegavec,\y)
 =P_{\z,r_{\y}}(\omegavec,\y)
       -\frac{3r_{\y}}{8}\norm{\y}^2.
 \label{eq:upper-regularization-identity}
\end{equation}
Strong convexity therefore yields
\[
 \norm{\bm u^\star-\bm u'^\star}_{\mathbf H'}^2
 \le\norm*{\left(0,-\frac{3r_{\y}\y_{\z,r_{\y}}^\star}{4}\right)}_{(\mathbf H')^{-1}}^2
 \le\frac{9}{4}r_{\y}D_{\cY}^2.
\]
Since $\alpha_{r_{\y}'}=\alpha_{r_{\y}}/2$, the squared triangle
inequality bounds the increase in the slow-pair contribution to
the Lyapunov energy by
\[
 \frac{2\alpha_{r_{\y}}}{\ell}
   \norm{\omegavec_{\z,r_{\y}}^\star-\omegavec_{\z,r_{\y}'}^\star}^2
 +r_{\y}\norm{\y_{\z,r_{\y}}^\star-\y_{\z,r_{\y}'}^\star}^2
 \le\left(\frac{9}{2}+1\right)r_{\y}D_{\cY}^2
 \le6r_{\y}D_{\cY}^2.
\]
Here we use the diameter bound on $\cY$ to control the distance
between the dual optimizers. By
\eqref{eq:upper-regularization-identity}, the objective gap of the
fast pair increases by at most
$3r_{\y}D_{\cY}^2/8$. Combining these bounds for the two
contributions in $\mathcal L$ proves
\eqref{eq:upper-regularization-comparison}.
\end{proof}

\begin{proof}[Proof of Proposition~\ref{prop:warm-start}]
Let
\begin{equation}
 \begin{gathered}
 J:=\left\lceil\log_4\frac{\ell}{8\bar r_{\y}}\right\rceil,
 \qquad
 r_{\y}^j:=\frac{\ell}{8\cdot4^j},\\
 \widehat S^j:=
 \FOAM_{\z^0,r_{\y}^j}\left(\widehat S^{j-1},\frac{1}{8}\right),
 \qquad 1\le j\le J.
 \end{gathered}
 \label{eq:upper-initialization-levels}
\end{equation}
Thus, at each stage, we reduce the regularization parameter by a
factor of four and then contract the Lyapunov energy by a factor
of eight.

We prove by induction that
\begin{equation}
 \mathcal L_{\z^0,r_{\y}^j}(\widehat S^j)
 \le15\left(\Delta+r_{\y}^jD_{\cY}^2\right),
 \qquad 0\le j\le J.
 \label{eq:upper-initialization-invariant}
\end{equation}
For $j=0$, this follows from Lemma~\ref{lem:startup-state}.
Suppose the claim holds at stage $j-1$. Applying
Lemma~\ref{lem:curvature-transfer} before the contraction gives
\[
 \begin{aligned}
 \mathcal L_{\z^0,r_{\y}^j}(\widehat S^j)
 \le\ & \frac{1}{8}\left[
 15\left(\Delta+r_{\y}^{j-1}D_{\cY}^2\right)
 +\frac{27}{4}r_{\y}^{j-1}D_{\cY}^2
 \right] = \frac{15}{8}\Delta
   +\frac{87}{8}r_{\y}^jD_{\cY}^2 \leq  15\left(\Delta+r_{\y}^jD_{\cY}^2\right),
 \end{aligned}
\]
where we used $r_{\y}^{j-1}=4r_{\y}^j$. This proves
\eqref{eq:upper-initialization-invariant}.
We return $r_{\y}=r_{\y}^J$ and $S^0=\widehat S^J$.
By the definition of $J$, we have $\frac{\bar r_{\y}}4<r_{\y}\le\bar r_{\y}$,
and \eqref{eq:upper-initialization-invariant} gives
\[
 \mathcal L_{\z^0,r_{\y}}(S^0)
 \le15\left(\Delta+r_{\y}D_{\cY}^2\right).
\]
Finally, by
Lemma~\ref{prop:foam-interface}, each stage $j\ge1$ costs
$\cO(\sqrt{\ell/r_{\y}^j})$ oracle calls. Since the regularization levels
decrease geometrically, the total cost, including the constant-cost
initialization, is
\[
 \cO\left(
 1+\sum_{j=1}^J\sqrt{\frac{\ell}{r_{\y}^j}}
 \right)
 =\cO\left(\sqrt{\frac{\ell}{r_{\y}}}\right).
\]
This proves Proposition~\ref{prop:warm-start}.
\end{proof} 

\subsection{Proof of Theorem~\ref{thm:upper} and Theorem~\ref{thm:upper-gs}}
\label{subsec:upper-proof}

We now combine Propositions~\ref{prop:anchor-shift}
and~\ref{prop:warm-start} to prove Theorem~\ref{thm:upper}
and Theorem~\ref{thm:upper-gs}.
Let $r_{\y},S^0,B^0$ be the initialization used by
Algorithm~\ref{alg:tracked-foam-overview}, and set $r=(2\ell,r_{\y})$.
We keep $r_{\y}$ fixed throughout the outer iterations.
% Recall that $W^t=p_r(\z^t)+2B^t$.
% Since $B^t\ge0$, we have $W^t\ge\inf p_r$.
% We first establish a common tracking guarantee, then prove
% stationarity and bound the oracle complexity for each criterion.
 
\subsubsection{Lyapunov descent and outer complexity}
\label{subsubsec:tracking-stopping}
We first establish a common guarantee for both stationarity criteria. The proof has three steps: maintaining an upper bound on the FOAM error, showing that a potential function decreases whenever the stopping test fails, and using the resulting decrease to bound the number of outer iterations and the total oracle complexity.

\begin{lemma}[Descent guarantee]
\label{lem:tracking-stop}
Consider Algorithm~\ref{alg:tracked-foam-overview}, initialized
with any target $\bar r_{\y}\in(0,\ell/8]$.
At every round reached by the algorithm,
\begin{equation}
 \mathcal L_{\z^t,r_{\y}}(S^t)\le B^t,
 \qquad
 \norm{\nabla p_r(\z^t)}^2\le8\ell Q^t.
 \label{eq:gradient-vs-Q}
\end{equation}
The algorithm terminates at the first round satisfying
\eqref{eq:observable-stop}, using
\begin{equation}
 \cO\left[
   \left(1+\frac{\ell(\Delta+r_{\y}D_{\cY}^2)}{\epsilon^2}\right)
          \sqrt{\frac{\ell}{r_{\y}}}
 \right]
 \label{eq:generic-tracking-cost}
\end{equation}
oracle calls, including initialization and an optional final
$\cO(1)$-cost extraction.
\end{lemma}

\begin{proof}[Proof of Lemma~\ref{lem:tracking-stop}]
Proposition~\ref{prop:warm-start} and the definition of $B^0$ give
$\mathcal L_{\z^0,r_{\y}}(S^0)\le B^0$.
Combining~\eqref{eq:interface-tracked-block} with
\eqref{eq:outer-budget-update} yields, by induction,
\[
  \mathcal L_{\z^t,r_{\y}}(S^t)\le B^t
\]
at every round reached by the algorithm.
Since $\z^{t+1}=\proj_{\cX}(-\omegavec_f^t/\ell)$,
the primal approximation bound in
Lemma~\ref{prop:foam-interface} implies
\begin{equation}
  \ell\norm*{\z^{t+1}-\x_{\z^t,r_{\y}}^\star}^2\le B^t.
  \label{eq:inexact-proximal-step}
\end{equation}
For every round in which the stopping criterion is not satisfied,
the definition of $p_r$ gives
\begin{align}
  p_r(\z^{t+1})-p_r(\z^t)
  \le\ &
    \Phi_r(\x_{\z^t,r_{\y}}^\star)
    +\ell\norm*{\z^{t+1}-\x_{\z^t,r_{\y}}^\star}^2
    -p_r(\z^t) \notag \\
  =\ &
    -\ell\norm*{\z^t-\x_{\z^t,r_{\y}}^\star}^2
    +\ell\norm*{\z^{t+1}-\x_{\z^t,r_{\y}}^\star}^2 \notag \\
  \le\ &
    -\ell\norm*{\z^t-\x_{\z^t,r_{\y}}^\star}^2+B^t, \label{eq:envelope-descent}
\end{align}
where the last inequality follows from
\eqref{eq:inexact-proximal-step}.
Moreover, since $\bm d^t=\z^{t+1}-\z^t$, we have
\begin{equation}
\begin{aligned}
  \ell\norm*{\bm d^t}^2
  \le
    2\ell\norm*{\z^t-\x_{\z^t,r_{\y}}^\star}^2
    +2\ell\norm*{\z^{t+1}-\x_{\z^t,r_{\y}}^\star}^2 \le
    2\ell\norm*{\z^t-\x_{\z^t,r_{\y}}^\star}^2+2B^t.
\end{aligned}
\label{eq:outer-step-bound}
\end{equation}
Substituting~\eqref{eq:outer-step-bound} into
\eqref{eq:outer-budget-update} gives
\[
  B^{t+1}
  \le\frac{48}{400}
       \ell\norm*{\z^t-\x_{\z^t,r_{\y}}^\star}^2
       +\frac{50}{400}B^t.
\]

Define $ W^t:=p_r(\z^t)+2B^t, 
 R^t:=\ell\norm{\z^t-\x_{\z^t,r_{\y}}^\star}^2+B^t$.
Combining the preceding inequality with~\eqref{eq:envelope-descent}
and the definition of $W^t$ yields
\begin{align}
 W^{t+1}-W^t
  \le \ &
    -\left(1-\frac{96}{400}\right)
      \ell\norm*{\z^t-\x_{\z^t,r_{\y}}^\star}^2
    -\left(1-\frac{100}{400}\right)B^t \le
    -\frac{3}{4}R^t, \notag\\ 
 Q^t  \le\ & 2\ell\norm*{\z^t-\x_{\z^t,r_{\y}}^\star}^2+3B^t
 \le3R^t.
 \label{eq:W-descent}
\end{align}
Furthermore, using
$\nabla p_r(\z^t)=2\ell(\z^t-\x_{\z^t,r_{\y}}^\star)$,
$\bm d^t=\z^{t+1}-\z^t$, and
\eqref{eq:inexact-proximal-step}, we obtain
\begin{align*}
  \norm*{\nabla p_r(\z^t)}^2
   \le
    8\ell^2\norm*{\bm d^t}^2
    +8\ell^2
      \norm*{\z^{t+1}-\x_{\z^t,r_{\y}}^\star}^2 \le
    8\ell\bigl(\ell\norm*{\bm d^t}^2+B^t\bigr)
   =8\ell Q^t.
\end{align*}
Since $p_r(\z)\le\Phi_r(\z)$ and
$\inf_{\z\in\mathbb R^m}p_r(\z)
=\inf_{\x\in\mathcal X}\Phi_r(\x)$,
the initialization bound gives
\begin{equation}
W^0-\inf_{\z\in\mathbb R^m}p_r(\z)
\le 31(\Delta+r_{\y}D_{\mathcal Y}^2).
\label{eq:initial-W-bound}
\end{equation}
Every round in which the stopping criterion is not satisfied has
$Q^t>\epsilon^2/(64\ell)$ and hence, by~\eqref{eq:W-descent},
decreases $W^t$ by more than $\epsilon^2/(256\ell)$.
Since $W^t\ge\inf_{\z\in\mathbb R^m}p_r(\z)$, the number of such
rounds is at most $\left\lceil
 256\ell\bigl(W^0-\inf_{\z\in\mathbb R^m}p_r(\z)\bigr)
     /\epsilon^2
 \right\rceil$.
The final stopping check requires only the projection and evaluation of
the stopping criterion. In GS mode, the final extraction uses
$\cO(1)$ additional oracle calls. Each round in which the stopping criterion
is not satisfied costs $\cO(\sqrt{\ell/r_{\y}})$ oracle calls,
as does initialization. This proves~\eqref{eq:generic-tracking-cost}.
\end{proof}

\subsubsection{Proof of Theorem~\ref{thm:upper}}
\label{subsec:os-complexity}

We now specialize the common tracking guarantee to
optimization stationarity (OS).
The only additional ingredient is a bound on the bias
introduced by dual regularization.

\begin{proof}[Proof of Theorem~\ref{thm:upper}] Choose the OS target~\eqref{eq:dual_reg}: $\bar r_{\y,\epsilon}^{\mathrm{OS}}$. By Proposition~\ref{prop:warm-start}, the returned regularization parameter satisfies $\bar r_{\y,\epsilon}^{\mathrm{OS}}/4<r_{\y}\le\bar r_{\y,\epsilon}^{\mathrm{OS}}$.
Consequently,
\[
r_{\y}D_{\mathcal Y}^2\le\frac{\epsilon^2}{8\ell},
\qquad
\sqrt{\frac{\ell}{r_{\y}}}
=\mathcal O\left(
\max\left\{1,\frac{\ell D_{\mathcal Y}}{\epsilon}\right\}
\right).
\]
At termination,~\eqref{eq:observable-stop} and
Lemma~\ref{lem:tracking-stop} give
\[
 \norm{\nabla p_r(\z^t)}
 \le \sqrt{8\ell Q^t}
 \le \frac{\epsilon}{\sqrt{8}}.
\]
The smoothing-bias bound in Lemma~\ref{lem:smoothing} then yields
\[
\begin{aligned}
 \norm{\nabla\Phi_{\frac{1}{2\ell}}(\z^t)}
 \le
 \norm{\nabla p_r(\z^t)}
 +D_{\cY}\sqrt{2\ell r_{\y}} \le
 \left(\frac{1}{\sqrt{8}}+\frac{1}{2}\right)\epsilon
 <\epsilon.
\end{aligned}
\]
Thus the returned point is $\epsilon$-OS.

Finally, substituting the bounds on $r_{\y}D_{\cY}^2$
and $\sqrt{\ell/r_{\y}}$ into
\eqref{eq:generic-tracking-cost} gives
\[
 \cO\left[
 \left(1+\frac{\ell\Delta}{\epsilon^2}\right)
 \max\left\{1,\frac{\ell D_{\cY}}{\epsilon}\right\}
 \right]
\]
first-order oracle calls, which is the claimed complexity.
\end{proof}

\subsubsection{Proof of Theorem~\ref{thm:upper-gs}}
\label{subsec:gs-complexity}
For game stationarity (GS), we use the same descent guarantee but choose the regularization parameter to control both the dual regularization error and the contribution of regularization to the initial potential gap. We then apply Lemma~\ref{lem:gs-certificate} to bound the two stationarity residuals.

\begin{proof}[Proof of Theorem~\ref{thm:upper-gs}]
Choose the GS target~\eqref{eq:gs-dual-regularization}: $\bar r_{\y,\epsilon}^{\mathrm{GS}}
 :=
 \min\left\{
   \ell/8,
   \epsilon/(4D_{\cY}),
   (\Delta+\epsilon^2/\ell)/(4D_{\cY}^2)
 \right\}$.
The warm-start guarantee implies
$r_{\y}\le\bar r_{\y,\epsilon}^{\mathrm{GS}}$, and hence
\begin{equation}
 r_{\y}D_{\cY}\le\frac{\epsilon}{4},
 \qquad
 r_{\y}D_{\cY}^2
 \le\frac{\Delta+\epsilon^2/\ell}{4}.
 \label{eq:gs-regularization-bounds}
\end{equation}

At termination, the tracking invariant and the stopping
criterion give
\[
 \mathcal L_{\z^t,r_{\y}}(S^t)
 \le B^t\le Q^t
 \le\frac{\epsilon^2}{64\ell}.
\]
Applying Lemma~\ref{lem:gs-certificate} therefore bounds
the primal and dual residuals of the output pair by
\[
\begin{aligned}
 R_x
 \le\ &
 \left(\sqrt{8}+2+\frac{2}{\sqrt{3}}\right)
 \sqrt{\ell Q^t} \le
 \frac{1}{8}\left(\sqrt{8}+2+\frac{2}{\sqrt{3}}\right)\epsilon
 <\frac{3}{4}\epsilon,
 \\
 R_y
 \le\ &
 \sqrt{\frac{\ell Q^t}{6}}+r_{\y}D_{\cY} \le
 \left(\frac{1}{\sqrt{384}}+\frac{1}{4}\right)\epsilon
 <\epsilon.
\end{aligned}
\]
Thus both residuals are smaller than $\epsilon$, and the
output is an $\epsilon$-GS pair.

We next bound the cost of the warm start.
Since Proposition~\ref{prop:warm-start} guarantees
$r_{\y}>\bar r_{\y,\epsilon}^{\mathrm{GS}}/4$, the definition
of the GS target implies
\begin{equation}
 \sqrt{\frac{\ell}{r_{\y}}}
 =
 \cO\left(
 \max\left\{
  1,\,
  \sqrt{\frac{\ell D_{\cY}}{\epsilon}},\,
  \frac{\ell D_{\cY}/\epsilon}
       {\sqrt{1+\ell\Delta/\epsilon^2}}
 \right\}
 \right).
 \label{eq:gs-condition-factor}
\end{equation}
Moreover,~\eqref{eq:gs-regularization-bounds} gives
\[
 1+\frac{\ell(\Delta+r_{\y}D_{\cY}^2)}{\epsilon^2}
 =
 \cO\left(1+\frac{\ell\Delta}{\epsilon^2}\right).
\]
Combining these estimates with
\eqref{eq:generic-tracking-cost} yields
\begin{equation}
 \cO\left[
 \left(1+\frac{\ell\Delta}{\epsilon^2}\right)
 \max\left\{1,\sqrt{\frac{\ell D_{\cY}}{\epsilon}}\right\}
 +\frac{\ell D_{\cY}}{\epsilon}
  \sqrt{1+\frac{\ell\Delta}{\epsilon^2}}
 \right].
 \label{eq:gs-uniform-rate}
\end{equation}
By Cauchy-Schwarz, the bracketed expression is at most
$\sqrt2$ times the square-root expression in Theorem~\ref{thm:upper-gs}.
This proves the theorem.
% To recover the stated form of the theorem, set
% \[
%  A:=1+\frac{\ell\Delta}{\epsilon^2},
%  \qquad
%  b:=\frac{\ell D_{\cY}}{\epsilon}.
% \]
% The expression in~\eqref{eq:gs-uniform-rate} is
% \[
%  A\max\{1,\sqrt b\}+b\sqrt A.
% \]
% By the squared-sum inequality, it is at most
% \[
%  \sqrt{2\left(
%  A^2\max\{1,b\}+Ab^2
%  \right)}
%  \le
%  \sqrt{2A(1+b)(A+b)}.
% \]
% Consequently, the oracle complexity is
% \[
%  \cO\left(
%  \sqrt{
%  \left(1+\frac{\ell\Delta}{\epsilon^2}\right)
%  \left(1+\frac{\ell D_{\cY}}{\epsilon}\right)
%  \left(1+\frac{\ell\Delta}{\epsilon^2}
%           +\frac{\ell D_{\cY}}{\epsilon}\right)
%  }
%  \right),
% \]
% which proves the theorem.
\end{proof}

\section{Conclusion}
\label{sec:conclusion}
% We established the optimal deterministic first-order oracle complexity
% of finding an $\epsilon$-OS point in smooth NC-C minimax optimization
% with a compact convex dual domain.
% For $\epsilon\lesssim\min\{\ell D_{\cY},\sqrt{\ell\Delta}\}$,
% we proved a lower bound of
% $\Omega(\ell^2D_{\cY}\Delta/\epsilon^3)$
% for arbitrary deterministic first-order methods.
% Our algorithm, \TrackedFOAM{}, attains the matching upper bound,
% thereby removing the logarithmic gap in previous complexity guarantees.

We established matching deterministic first-order lower and upper
bounds for smooth NC-C minimax optimization under both optimization
stationarity and game stationarity.
For OS, the optimal dependence on the target accuracy is
$\epsilon^{-3}$, while for GS it is $\epsilon^{-5/2}$,
with the remaining problem parameters fixed and $\epsilon\downarrow0$.
For GS, we further establish matching bounds for
every $\ell,D_{\cY},\Delta,\epsilon>0$, including the additional
cost that persists at small initial gaps in the value function.
Thus, the different dependence on $\epsilon$ under the two
stationarity criteria reflects an intrinsic separation in their
worst-case deterministic first-order oracle complexity.

% A natural next step is to extend our lower bound to randomized
% first-order algorithms and determine whether randomization can
% improve the worst-case oracle complexity.
% Our $C^\infty$ hard instance may also serve as a starting point
% for establishing higher-order oracle lower bounds, extending
% the study of such bounds for smooth nonconvex minimization
% \citep{carmon2019lower} to the minimax setting.
% This direction requires controlling the higher-order derivatives
% and the information they reveal about the hard instance.

\section*{AI Disclosure}
The authors developed the constructions and algorithm through
iterative collaboration with generative AI tools.
Throughout this process, the authors formulated and refined
the mathematical tasks, guided the exploration through
successive prompts, critically evaluated the proposed arguments,
and progressively revised and simplified the constructions
and proofs.

\texttt{GPT-5.5 Pro} assisted with the initial development of
the OS hard instance for sufficiently small $\epsilon$.
\texttt{GPT-5.6 Sol Ultra} assisted with simplifying this
construction and developing
Algorithm~\ref{alg:tracked-foam-overview}.
\texttt{GPT-6 Astra Ultra} assisted with developing the OS
hard instances for the remaining parameter regimes and
the GS hard instances, and with improving the presentation
of the proofs.
The authors independently verified all mathematical arguments
and take full responsibility for the content of the paper.
An accompanying Lean formalization, developed with Codex,
is available at \url{https://github.com/SiyuPan04/ncc-lean}.

\section*{Acknowledgements}
Jiajin Li was supported by the Natural Sciences and Engineering
Research Council of Canada through Discovery Grant RGPIN-2025-05817.
% Bibliography embedded from ref.bib using abbrvnat for the single-file preview.
\bibliography{ref}
\bibliographystyle{abbrvnat}
 
\appendix

\section{Useful Technical Lemmas}
\label{app:components}
\subsection{Proof of Lemmas~\ref{lem:deterministic-construction-os}
and~\ref{lem:deterministic-construction-gs}}
\label{app:deterministic-extension-proof}

\begin{proof}[Proof of Lemmas~\ref{lem:deterministic-construction-os}
and~\ref{lem:deterministic-construction-gs}]
For either construction, let $\widehat\Phi$ denote its value function.
Setting $\bm\eta=\bz$ preserves feasibility and cannot decrease
the objective. The auxiliary primal variables separate in the
Moreau minimization, so
\[
 \widehat\Phi(\x,\bm\xi)=\Phi(\x),\qquad
 \nabla\widehat\Phi_{1/(2\ell)}(\x,\bm\xi)
 =\bigl(\nabla\Phi_{1/(2\ell)}(\x),\bz_L\bigr).
\]
Thus the initial gap and OS residual are unchanged.
Both constructions preserve joint $\ell$-smoothness and dual concavity,
since the added GS quadratic acts on separate coordinates and has
smoothness $\ell$. The dual diameter bound is also unchanged,
so both constructions belong to $\cF(\ell,D_{\cY},\Delta)$.

For GS, write $\y=(\bm c,\w)$ and $R=D_{\cY}/2$, and consider
an $\epsilon$-GS point $(\x,\bm\xi;\y,\bm\eta)$ of
$\widehat f_{\GS}$. Then $\norm{\nabla_{\x}f(\x;\y)}\le\epsilon$.
The dual normal cone gives $\lambda\ge0$,
$\bm\zeta=(\bm\zeta_c,\bz_n)$ with $\bm\zeta_c\ge\bz$ and
$\ip{\bm\zeta_c}{\bm c}=0$, and an error $\bm e$ such that
\[
 \nabla_{\y}f(\x;\y)=\lambda\y-\bm\zeta+\bm e,
 \qquad
 \norm{\bm e}^2+(\ell+\lambda)^2\norm{\bm\eta}^2
 \le\epsilon^2.
\]
Here $\lambda=0$ unless $\norm{(\y,\bm\eta)}=R$.
If $\lambda=0$, $\bm\eta=\bz$, or $\y=\bz$, then
$\lambda\y-\bm\zeta\in N_{\cY}(\y)$ and $(\x,\y)$ is
$\epsilon$-GS for $f$. Otherwise, set $\bar\y:=R\y/\norm\y$.
Positive scaling preserves the dual support and gives
\[
 \lambda\y-\bm\zeta
 =\frac{\lambda\norm\y}{R}\bar\y-\bm\zeta
 \in N_{\cY}(\bar\y),
 \qquad
 \norm{\bar\y-\y}=R-\norm\y
 \le\norm{\bm\eta}\le\frac{\epsilon}{\ell}.
\]
Joint smoothness bounds both GS residuals at $(\x,\bar\y)$ by
$\epsilon+\ell\norm{\bar\y-\y}\le2\epsilon$.
In either case, we obtain a $2\epsilon$-GS point of $f$ with the
same primal variable and original dual support.

We next check that rotations preserve the function class and
stationarity residuals. Let $g\in\cF(\ell,D_{\cY},\Delta)$ be defined
on $\mathcal P\times\mathcal Q$, with value function $\Psi$, and
let $\bU,\bV$ be square orthogonal matrices satisfying
$\bU\mathcal P=\mathcal P$ and $\bV\mathcal Q=\mathcal Q$.
The rotated value function satisfies
$\Psi_{\bU,\bV}(\x)=\Psi(\bU^\top\x)$, including the indicator terms.
Orthogonal changes of variables preserve joint smoothness, dual
concavity, and the initial gap, so
$g_{\bU,\bV}\in\cF(\ell,D_{\cY},\Delta)$.
Substituting $\z=\bU\uvec$ in the Moreau minimization gives
\[
 \bigl[\Psi_{\bU,\bV}\bigr]_{1/(2\ell)}(\x)
 =\Psi_{1/(2\ell)}(\bU^\top\x),\qquad
 \nabla\bigl[\Psi_{\bU,\bV}\bigr]_{1/(2\ell)}(\x)
 =\bU\nabla\Psi_{1/(2\ell)}(\bU^\top\x).
\]
The OS residual is therefore unchanged.
For GS, the normal cones and gradients satisfy
\[
 \begin{aligned}
 N_{\mathcal P}(\x)
 &=\bU N_{\mathcal P}(\bU^\top\x),
 &\qquad
 N_{\mathcal Q}(\y)
 &=\bV N_{\mathcal Q}(\bV^\top\y),\\
 \nabla_{\x}g_{\bU,\bV}(\x;\y)
 &=\bU\nabla_{\x}g(\bU^\top\x;\bV^\top\y),
 &\qquad
 \nabla_{\y}g_{\bU,\bV}(\x;\y)
 &=\bV\nabla_{\y}g(\bU^\top\x;\bV^\top\y).
 \end{aligned}
\]
Since orthogonal matrices preserve distances, both GS residuals
are unchanged as well. These identities apply to the rotations
$(\bU,\bV)$ for $\widehat f_{\OS}$ and
$(\bU,\operatorname{diag}(\bI_k,\bV))$ for $\widehat f_{\GS}$:
they preserve the Euclidean primal domains and dual balls, and
fix the nonnegative controller coordinates in the GS case.

Finally, fix $\sA\in\cA_{\rm det}$ and construct the rotations
successively along its queries. Before query $t<L$, suppose the
first $t$ original chain directions have been fixed, together
with $t$ auxiliary directions on each side, and every orthogonal
completion reproduces the previous replies. This holds at $t=0$.
Determinism fixes the next query. Align the directions for
$\xi_{t+1}$ and $\eta_{t+1}$ with its primal and hidden dual
components outside the assigned spans. If a component is zero,
choose any available orthogonal unit vector.
The original chain coordinates of the query are then supported
on $[t]$.

By the zero-chain property, the reply requires at most the next
original chain direction, which we choose orthogonal to the
assigned directions. The auxiliary gradients are zero for
$\widehat f_{\OS}$ and $(\bz,-\ell\bm\eta)$ for
$\widehat f_{\GS}$, so no further directions are needed.
The $L$ auxiliary coordinates on each side suffice through
query $L-1$. These choices fix the full reply, including the value
and controller gradients, independently of the remaining basis
directions, and hence preserve the induction hypothesis.

Completing the bases after $L$ queries gives a fixed rotated
instance with this transcript. Each of the first $L$ queries has
original chain support in $[L-1]$.
For OS, the unchanged residual and the hypothesis on $f$ exclude
$\epsilon$-stationarity. For GS, an $\epsilon$-GS query would yield
a $2\epsilon$-GS point of $f$ with the same original chain support,
again contradicting the hypothesis. Thus the oracle complexity
is at least $L$ in both cases.
\end{proof}

\begingroup

\subsection{Proof of Lemma~\ref{lem:os-endpoint}}
\label{app:os-endpoint}

\renewcommand{\u}{\bm u}
\renewcommand{\v}{\bm v}
We first verify the properties of $f_2$ in
Lemma~\ref{lem:os-endpoint-properties}, then use $f_1$ and $f_2$
to prove Lemma~\ref{lem:os-endpoint}.
Fix $\ell,D_{\cY},\Delta>0$ and
$0<\epsilon<\sqrt{8\ell\Delta/3}$, with $R$ as
in~\eqref{eq:os-endpoint-diameter}. Then
\begin{equation}
 \frac{\epsilon}{2\ell}<R\le\frac{\epsilon}{\ell},
 \qquad
 \frac{3}{2}\ell R^2\le\Delta.
 \label{eq:os-endpoint-scale}
\end{equation}

\begin{proof}[Proof of Lemma~\ref{lem:os-endpoint-properties}]
Use the parameters in~\eqref{eq:os-endpoint-parameters},
write $L:=L_2$, and let $\Psi,Q$ be defined
by~\eqref{eq:os-endpoint-scalars}.
The function $\Psi$ is infinitely differentiable at $-\delta$
because the exponential and all its derivatives vanish there.
For $t>-\delta$, setting $z=2\delta/(t+\delta)$ gives
\[
 \Psi''(t)=e^{-z}\left(1+z+\frac{z^2}{2}\right)\in[0,1].
\]
This expression is nondecreasing in $t$, since its derivative
with respect to $z$ is $-z^2e^{-z}/2$.
Consequently, $0\le Q''\le1$ and $0\le Q'\le R$.
Also, $Q(t)=0$ for $t\le-\delta$ and
\begin{equation}
 Q(0)=\frac{\delta^2}{2e^2},\qquad
 Q'(0)=\frac{2\delta}{e^2},\qquad
 Q(t)\ge Rt-\frac{R^2}{2}.
 \label{eq:os-endpoint-scalar-bounds}
\end{equation}
To obtain the last bound, expand at infinity:
$\Psi(t)=t^2/2+\delta^2/2+o(1)$ and hence
$Q(t)-Rt\to-R^2/2$. Since $Q'\le R$, the function $Q(t)-Rt$
is nonincreasing and lies above this limit.

To verify smoothness and dual concavity, let $\bA_L$ be defined
by~\eqref{eq:A} and set
\[
 \mathbf B:=\frac{1}{4}\bA_L+\frac{\ee_L\ee_L^\top}{L}.
\]
The matrix $\mathbf B$ is positive semidefinite, and
$\bI_L-\mathbf B$ is symmetric and diagonally dominant with
nonnegative diagonal entries, so $0\preceq \mathbf B\preceq\bI_L$.
For $L=1$, $\mathbf B=[1]$.
With $b=u_L/\sqrt L$ and $\theta=Q''(b-a-R)\in[0,1]$, the
Hessian of $f_2/\ell$ is
\[
 \mathbf K_\theta=
 \begin{pmatrix}
 -\theta&-\frac{(1-\theta)\ee_L^\top}{\sqrt L}\\
 -\frac{(1-\theta)\ee_L}{\sqrt L}&-\frac{\bA_L}{4}-\frac{\theta\ee_L\ee_L^\top}{L}
 \end{pmatrix}
 =(1-\theta)\mathbf K_0+\theta \mathbf K_1.
\]
We have $\mathbf K_1=-\operatorname{diag}(1,\mathbf B)$, so $\norm{\mathbf K_1}\le1$.
For every $(h,\w)\in\R\times\R^L$,
\[
 \left|(h,\w)^\top \mathbf K_0(h,\w)\right|
 \le\frac{2|hw_L|}{\sqrt L}+\frac{1}{4}\w^\top\bA_L\w
 \le h^2+\w^\top \mathbf B\w
 \le h^2+\norm{\w}^2.
\]
Thus $\norm{\mathbf K_0}\le1$ and $\norm{\mathbf K_\theta}\le1$,
which proves joint $\ell$-smoothness. The negative semidefinite
dual Hessian block gives dual concavity.

To bound the initial gap, for $0\le b\le2R$ define
\[
 u_i(b):=\sqrt L\,b+\frac{4R(L-i)}{\sqrt L},\qquad i\in[L].
\]
These vectors satisfy
$\norm{\u(b)}\le6RL\le3D_{\cY}/8<D_{\cY}/2$ and make all
squares in~\eqref{eq:os-endpoint-core} vanish.
Set $G(a,b):=(R-a)b-Q(b-a-R)$. Then
$f_2(a;\u(b))=\ell G(a,b)$, whereas every $\u$ satisfies
$f_2(a;\u)\le\ell G(a,u_L/\sqrt L)$.

For $a\ge0$, choose $b=0$ to obtain $G(a,0)=0$.
For $a<0$, choose $b=R$. Since $Q'\le R$,
\[
 G(a,R)=R(R-a)-Q(-a)\ge R^2-Q(0)>0.
\]
Hence $\Phi_2\ge0$. Since $Q\ge0$, the definition of $f_2$
also gives $f_2(R;\u)\le0$, and therefore $\Phi_2(R)=0$.
Finally,~\eqref{eq:os-endpoint-scalar-bounds} gives, for every $b$,
\[
 G(0,b)=Rb-Q(b-R)\le\frac{3}{2}R^2.
\]
The initial gap is therefore at most $3\ell R^2/2\le\Delta$,
which, together with smoothness and dual concavity, proves
$f_2\in\cF(\ell,D_{\cY},\Delta)$.

We next compute the OS residual at the origin. Set
\[
 a_\star:=R-Q'(0)=R-\frac{2\delta}{e^2},\qquad
 b_\star:=a_\star+R\in(0,2R).
\]
At $(a_\star,\u(b_\star))$, all shifted squares vanish and
$b_\star-a_\star-R=0$. Moreover,
\[
 G_b(a_\star,b_\star)=R-a_\star-Q'(0)=0,\qquad
 G_a(a_\star,b_\star)+2a_\star
 =-b_\star+Q'(0)+2a_\star=0.
\]
This pair is therefore a saddle point of
$f_2(a;\u)+\ell a^2$, which is $\ell$-strongly convex
in $a$ and concave in $\u$.
It follows that
\[
 \prox_{\frac{\Phi_2}{2\ell}}(0)=a_\star,\qquad
 \norm{\nabla[\Phi_2]_{\frac{1}{2\ell}}(0)}=2\ell a_\star.
\]
Since $2/e^2<1$ and $\delta>0$, this norm exceeds
$2\ell(R-\delta)=\epsilon$.

Finally, to check the zero-chain property, expand the shifted
squares to obtain, up to an additive constant,
\[
 \frac{f_2(a;\u)}{\ell}
 =\frac{Ru_1-au_L}{\sqrt L}
  -Q\!\left(\frac{u_L}{\sqrt L}-a-R\right)
  -\frac{1}{8}\sum_{i=1}^{L-1}(u_i-u_{i+1})^2.
\]
When $a=u_L=0$, the argument of $Q$ is $-R\le-\delta$, so
the derivatives of that term vanish. The remaining terms form
a zero-chain in the order $(u_1,\ldots,u_L,a)$.
\end{proof}

\begin{proof}[Proof of Lemma~\ref{lem:os-endpoint}]
If $\ell D_{\cY}/\epsilon<16$, use $f_1$
from~\eqref{eq:os-scalar-endpoint-instance}.
This scalar quadratic is $\ell$-smooth and independent of the dual
variable. Its initial gap is $3\ell R^2/2\le\Delta$, so
$f_1\in\cF(\ell,D_{\cY},\Delta)$.
The proximal point of $\Phi_1$ at zero is $R$, so the Moreau-gradient norm is
$2\ell R>\epsilon$. Every algorithm therefore fails the OS
criterion at its initial query, giving
\[
 \cT_\epsilon^{\OS}(\cA_{\rm det},\cF(\ell,D_{\cY},\Delta))
 \ge1\ge\frac{1}{100}\left(1+\frac{\ell D_{\cY}}{\epsilon}\right).
\]

For $\ell D_{\cY}/\epsilon\ge16$,
Lemma~\ref{lem:os-endpoint-properties} shows that $f_2$ belongs to
$\cF(\ell,D_{\cY},\Delta)$ and is a zero-chain in the order
$(u_1,\ldots,u_{L_2},a)$.
Support on the first $L_2-1$ coordinates forces $a=0$.
Thus~\eqref{eq:os-endpoint-obstruction} and
Lemma~\ref{lem:deterministic-construction-os}, applied with $L=L_2$,
give
\[
 \cT_\epsilon^{\OS}(\cA_{\rm det},\cF(\ell,D_{\cY},\Delta))\ge L_2.
\]
Set $\rho:=D_{\cY}/(16R)$. By
\eqref{eq:os-endpoint-scale}, $\rho\ge{\ell D_{\cY}}/{(16\epsilon)}\ge1,$ and thus
\begin{equation}
 L_2= \lfloor\rho\rfloor\ge\frac{\rho}{2}
 \ge\frac{1}{32}\frac{\ell D_{\cY}}{\epsilon}
 \ge\frac{1}{100}\left(1+\frac{\ell D_{\cY}}{\epsilon}\right).
 \label{eq:os-endpoint-length}
\end{equation}
This proves the claimed bound in both cases.
\end{proof}
\endgroup

\subsection{Properties of OS inner chain}
\label{app:inner-interface-proof}
\begingroup

The following lemma adapts \citet[Lemmas~4 and~7]{li2021lower}
to our setting and records the properties of the inner chain used in
Section~\ref{sec:verification}.

\begin{lemma}[Properties of the inner chain]
\label{lem:inner-interface}
There are numerical constants $c_H,c_{\y}>0$, independent of $M,N,D$,
such that the following properties hold for every $M\in\N$, $N\in\N_0$, and $D>0$.
\begin{enumerate}[label=(\roman*)]
  \item The identity~\eqref{eq:effective-link} holds, and the maximizer satisfies
  $\norm{\w^\star(a,b)}\le c_{\y}N\sqrt{a^2+b^2}$.
  For $N\ge1$, one also has $1\le c_N\le8/5$.
  \item The joint Hessian of $H$ has norm at most $c_H$.
  For $N\ge1$, the function $H$ is $N^{-2}$-strongly concave in the dual variable;
  \item $H$ is a zero-chain with coordinate ordering $\ord$ given by
  $(a,b,\w)_{\ord}=(a,(w_j)_{j=1}^N,b)$, while
  $H(a,b;\bz)\ge0$ and
  $H(0,0;\w)\le0$;
  \item The function
  \[
    V_D(\bm{\nu}):=\max_{\y\in\B^{MN}_D}
      \sum_{i=1}^{M}H(a_i,b_i;\y^{(i)})
  \]
  is differentiable and satisfies
  \[
    \ip*{\bm{\nu}}{\nabla V_D(\bm{\nu})}
    \ge\frac{2}{5}\norm*{\bm{\nu}}^2.
  \]
\end{enumerate}
\end{lemma}

\begin{proof}[Proof of Lemma~\ref{lem:inner-interface}]
When $N=0$, one has $H(a,b;\bz_0)=Q(a,b)=a^2-ab+b^2$,
and the unique dual maximizer is $\bz_0$.
The Hessian of $Q$ has norm three, and its only mixed term is $-ab$,
so \textup{(i)}-\textup{(iii)} hold with coordinate order $(a,b)$.
Moreover, $V_D(\bm\nu)=\sum_{i=1}^M Q(a_i,b_i)$, and hence
\[
 \ip{\bm\nu}{\nabla V_D(\bm\nu)}
 =2\sum_{i=1}^M Q(a_i,b_i)\ge\norm{\bm\nu}^2,
\]
which proves \textup{(iv)} in this case. For the remainder, let $N\ge1$.

We first prove \textup{(i)}. For $N=1$, $\bM_1=\bB_1=[1]$ and
$k_1=c_1=1$. For $N\ge2$, write $\xi_j:=(\bB_N)_{j1}$.
The interior and last rows of
$\bM_N(\xi_1,\ldots,\xi_N)^\top=\ee_1$ give
\begin{equation}
\label{eq:xi}
  \xi_{j-1}-(2+N^{-2})\xi_j+\xi_{j+1}=0
  \quad(2\le j\le N-1),
  \quad\mbox{and}\quad
  \xi_{N-1}=(1+N^{-2})\xi_N.
\end{equation}
Choose $\lambda_N>0$ so that
$\cosh\lambda_N=1+1/(2N^2)$.
The formula
\[
  \xi_j=\xi_N
  \frac{\cosh((N-j+1/2)\lambda_N)}
       {\cosh(\lambda_N/2)},
  \qquad j\in[N],
\]
satisfies both relations in \eqref{eq:xi}.
These relations determine all entries from $\xi_N$, so the formula holds.

Since $\cosh t\ge1+t^2/2$ for $t\ge0$, the defining equation
gives $\lambda_N\le1/N\le1$.
Thus the ratios
\[
 r_j:=\frac{\cosh((N-j+1/2)\lambda_N)}
              {\cosh(\lambda_N/2)}
 \quad(j\in[N])
\]
satisfy $1\le r_j\le\cosh(1)<8/5$.
Summing the equations $\bM_N\bm\xi=\ee_1$ gives
$\sum_{j=1}^N\xi_j=N^2$, because every row sum of $\bA_N$ is zero.
Since $\xi_j=\xi_Nr_j$, it follows that
\[
 k_N=\xi_N=\frac{N^2}{\sum_{j=1}^N r_j}>0,
 \qquad
 \frac{N}{\cosh(1)}\le\xi_j\le N\cosh(1).
\]
Reversing the coordinate order leaves $\bM_N$ and its inverse unchanged,
so the same entrywise bounds hold for the last column. Consequently,
for every $N\ge1$, including the explicit case $N=1$,
\begin{equation}
  \frac{N}{\cosh(1)}\le k_N\le N,
  \qquad
  k_N^{-\frac{1}{2}}\le\sqrt{\frac{\cosh(1)}{N}},
  \qquad\mbox{and}\qquad
  \norm*{\bB_N\ee_1}=\norm*{\bB_N\ee_N}
  \le\cosh(1)N^{\frac32}.
\label{eq:inherited-column-bounds}
\end{equation}
Also, $c_N=\xi_1/\xi_N=r_1$ for $N\ge2$, so
$1\le c_N\le\cosh(1)<8/5$ for every $N\ge1$.
Thus $2-c_N\in[2/5,1]$.

The dual optimality condition gives
$\w^\star(a,b)=k_N^{-1/2}\bB_N(a\ee_1-b\ee_N)$.
The triangle inequality and \eqref{eq:inherited-column-bounds} yield
\[
  \norm*{\w^\star(a,b)}
  \le\cosh(1)^{3/2}N(|a|+|b|)
  \le20\sqrt{20}\,N\sqrt{a^2+b^2}.
\]
Thus we may take $c_{\y}:=20\sqrt{20}$.
Symmetry also gives
$(\bB_N)_{11}=(\bB_N)_{NN}=c_Nk_N$ and
$(\bB_N)_{1N}=(\bB_N)_{N1}=k_N$.
Substituting the maximizer therefore gives
\[
\begin{aligned}
  \max_{\w\in\R^N}H(a,b;\w)
  =\ & \frac{1}{2k_N}
    (a\ee_1-b\ee_N)^\top\bB_N(a\ee_1-b\ee_N)
    +\frac{2-c_N}{2}(a^2+b^2)\\
  =\ & \frac{c_N}{2}(a^2+b^2)-ab
    +\frac{2-c_N}{2}(a^2+b^2)\\
  =\ & a^2-ab+b^2,
\end{aligned}
\]
which proves \eqref{eq:effective-link} and \textup{(i)}.

For \textup{(ii)}, when $N=1$, the formula
\[
 H(a,b;w)=-\frac{1}{2}w^2+(a-b)w+\frac{1}{2}(a^2+b^2)
\]
has Hessian eigenvalues $1,\sqrt3,-\sqrt3$ and is $1$-strongly
concave in $w$. For $N\ge2$, the eigenvalue bounds
$N^{-2}\bI_N\preceq\bM_N\preceq(4+N^{-2})\bI_N$
show that $H(a,b;\cdot)$ is $N^{-2}$-strongly concave.
Let ${\bf E}_N:=(\ee_1,-\ee_N)$, so $\norm{{\bf E}_N}=1$.
In the variables $(a,b,\w)$, the Hessian is
\[
  \nabla^2H
  =\begin{pmatrix}
      (2-c_N)\bI_2 & k_N^{-\frac{1}{2}}{\bf E}_N^\top\\
      k_N^{-\frac{1}{2}}{\bf E}_N & -\bM_N
    \end{pmatrix}.
\]
Splitting this matrix into its block diagonal and off-diagonal parts gives
\[
  \norm*{\nabla^2H}
  \le\max\{2-c_N,\norm*{\bM_N}\}
      +k_N^{-\frac{1}{2}}\norm*{{\bf E}_N}
  <5+1=6,
\]
where \eqref{eq:inherited-column-bounds} and $N\ge2$ give
$k_N^{-1/2}\le\sqrt{\cosh(1)/2}<1$.
Taking $c_H:=6$ proves \textup{(ii)} in every case.

For \textup{(iii)}, $\bM_N$ is tridiagonal, and the endpoint
couplings involve only $aw_1$ and $bw_N$.
Thus each component of the saddle vector field depends linearly only
on its own coordinate and its immediate neighbors in
$(a,b,\w)_{\ord}=(a,w_1,\ldots,w_N,b)$.
If the input is supported on a prefix of this order, every component
of the vector field beyond the next coordinate is therefore zero.
This proves the zero-chain property.
The remaining claims follow from
$H(a,b;\bz)=(2-c_N)(a^2+b^2)/2\ge0$ and
$H(0,0;\w)=-\w^\top\bM_N\w/2\le0$.

For \textup{(iv)}, define
$V_D^\circ(\bm{\nu}):=V_D(\bm{\nu})-(2-c_N)\norm{\bm{\nu}}^2/2$.
By the definition of $H$, this is the maximum of functions affine
in $\bm{\nu}$, so it is convex.
The set $\B^{MN}_D$ is compact and convex, and the objective is strongly
concave in $\y$, so the maximizer exists and is unique.
Danskin's theorem therefore gives differentiability of $V_D^\circ$
and $V_D$.

Feasibility of $\y=\bz$ gives $V_D^\circ(\bm{\nu})\ge0$.
At $\bm{\nu}=\bz$, the maximized expression is a nonpositive quadratic
that vanishes at $\y=\bz$, so $V_D^\circ(\bz)=0$.
Convexity now yields
\[
  \ip*{\bm{\nu}}{\nabla V_D^\circ(\bm{\nu})}
  \ge V_D^\circ(\bm{\nu})-V_D^\circ(\bz)
  \ge0.
\]
Adding back the quadratic term gives
\[
  \ip*{\bm{\nu}}{\nabla V_D(\bm{\nu})}
  =\ip*{\bm{\nu}}{\nabla V_D^\circ(\bm{\nu})}
   +(2-c_N)\norm*{\bm{\nu}}^2
  \ge\frac{2}{5}\norm*{\bm{\nu}}^2,
\]
which proves \textup{(iv)}. We finish the proof.
\end{proof}
\endgroup
\subsection{Properties of \texorpdfstring{$p$}{p}}
\label{app:gate-transformations-proof}
\begin{lemma}
\label{lem:gate-transformations}
$p\in C^\infty(\R)$. Moreover, for every $t\in\R$,
\[0\le p'(t)\le1,\qquad\mbox{and}\qquad0\le p''(t)\le2.\]
\end{lemma}
\begin{proof}[Proof of Lemma~\ref{lem:gate-transformations}]
Direct computation gives the explicit form of $p'$
\begin{equation}
p'(t)=
\begin{cases}
0, & t\le0,\\[2pt]
\displaystyle
\frac{\exp(-1/t)}
     {\exp(-1/t)+\exp(-1/(1-t))},
& 0<t<1,\\[8pt]
1, & t\ge1.
\end{cases}
\label{eq:base-step}
\end{equation}
The middle branch of \eqref{eq:base-step} extends smoothly to
$0$ and $1$, with values $0$ and $1$, respectively, and all derivatives
of positive order vanishing at both endpoints.
Moreover, the identity $p'(t)+p'(1-t)=1$ on $(0,1)$ gives
$\int_0^1p'(v)\,{\rm d}v=1/2$.
Thus the integral branch in \eqref{eq:base-ramp} matches the two
affine tails in both value and all derivatives, proving
$p\in C^\infty(\R)$.

Equation~\eqref{eq:base-step} also gives $0\le p'(t)\le1$
and $p''(t)=0$ outside $(0,1)$.
For $0<t<1$, set $v:=2t-1$ and $z:=2v/(1-v^2)$.
Differentiating \eqref{eq:base-step} yields
\begin{equation}
  0\le p''(t)
  =\frac{2(1+v^2)}{(1-v^2)^2\cosh^2 z}
  \le\frac{2}{1+v^2}
  \le2,
\label{eq:base-ramp-second-derivative}
\end{equation}
where we used
$\cosh^2z\ge1+z^2=(1+v^2)^2/(1-v^2)^2$.
We finish the proof.
\end{proof}
% \subsection{Properties of $q$}
\begingroup

\subsection{Properties of \texorpdfstring{$q$}{q}}
\label{app:outer-gate-proof}
\begin{lemma}[Properties of $q$]
\label{lem:outer-gate}
The function $q$ belongs to $C^\infty(\R)$, equals zero on
$(-\infty,1/5]$ and one on $[1,\infty)$, and satisfies
\[
 0\le q\le1,\qquad 0\le q'\le\frac{5}{2},\qquad \norm{q''}_\infty<\infty.
\]
\end{lemma}
\begin{proof}[Proof of Lemma~\ref{lem:outer-gate}]
By definition, $q(t)=p'((5t-1)/4)$.
Lemma~\ref{lem:gate-transformations} gives the plateau values and
$0\le q'(t)=(5/4)p''((5t-1)/4)\le5/2$.
The continuous function $q''$ vanishes outside $[1/5,1]$ and
is therefore bounded.
\end{proof}
\endgroup

\subsection{Properties of \texorpdfstring{$e_{\s}$}{e\_s} and \texorpdfstring{$e_{\bm{\nu}}$}{e\_nu}}
% \subsection{Properties of $e_{\texorpdfstring{\s}{s}}$ and $e_{\texorpdfstring{\bm\nu}{nu}}$}
\label{app:identity-extensions-proof}
\begin{lemma}[Properties of $e_{\s}$ and $e_{\bm{\nu}}$]
\label{lem:identity-extensions}
The functions $e_{\s},e_{\bm{\nu}}$ belong to $C^\infty(\R)$, with
$e_{\s}(t)=t$ for $|t|\le2$ and $e_{\bm{\nu}}(t)=t$ for $|t|\le21$.
They are constant on each tail outside $[-3,3]$ and $[-22,22]$,
respectively.  For every $t\in\R$,
\[
  |e_{\s}(t)|\le\frac{5}{2},\qquad\mbox{and}\qquad |e_{\bm{\nu}}(t)|\le\frac{43}{2}.
\]
Moreover, for $e\in\{e_{\s},e_{\bm{\nu}}\}$ and every $t\in\R$,
\[
0\le e'(t)\le1,\qquad\mbox{and}\qquad|e''(t)|\le2.
\]
\end{lemma}
\begin{proof}[Proof of Lemma~\ref{lem:identity-extensions}]
Lemma~\ref{lem:gate-transformations} gives $e_{\s},e_{\bm{\nu}}\in C^\infty(\R)$.
For example, $e_{\s}'(t)=p'(t+3)-p'(t-2)$ and
$e_{\s}''(t)=p''(t+3)-p''(t-2)$.
The same lemma gives $0\le e_{\s}'\le1$ and $|e_{\s}''|\le2$, since
$p'$ is nondecreasing with values in $[0,1]$ and $0\le p''\le2$.
When $|t|\le2$, the first term defining $e_{\s}$ is linear and the
second vanishes, so $e_{\s}(t)=(t+3-1/2)-5/2=t$.
For $t\le-3$, both terms vanish and $e_{\s}(t)=-5/2$; for $t\ge3$,
both are linear and $e_{\s}(t)=5/2$.
Monotonicity then gives $|e_{\s}|\le5/2$.

The same argument, with shifts $22$ and $21$, proves all the
stated properties of $e_{\bm{\nu}}$.
We finish the proof.
\end{proof}
\subsection{Properties of GS inner chain}
\label{gs:app:identities}

The following lemma records the matrix bounds and derivative identities
used to analyze the GS inner chain.
For $N\ge2$, let $\bD_N$ be defined in
\eqref{gs:eq:weighted-matrix} and set
\[
 \w_0:=\bigl(1,\underbrace{N^{-1/2},\ldots,N^{-1/2}}_{N-2},1\bigr)^\top,
 \qquad
 \vvec:=\bD_N^{-1}\ee_1.
\]
The middle block is empty when $N=2$.
Let $\ind\in\R^N$ denote the all-ones vector.

\begin{lemma}[Properties of the GS inner chain]
\label{gs:mem:matrix-bounds}
The function $I_1$ is jointly $2$-smooth and a zero-chain
in the order $(z,t)$ for each fixed $y$.
The following properties hold for every $N\ge2$.
\begin{enumerate}[label=(\roman*)]
 \item The matrix $\bD_N$ is invertible and satisfies
 \[
  \norm{\bD_N}\le2,
  \qquad\mbox{and}\qquad
  \norm{\bD_N^{-1}}\le\sqrt2\,N.
 \]
 \item The vectors $\w_0$ and $\vvec$ satisfy
 \[
  \bD_N^\top\w_0=\frac{\ee_N}{\sqrt N},
  \qquad
  \norm{\w_0}^2=3-\frac{2}{N}<3,
  \qquad
  \vvec=\sqrt N\,\ind,
  \qquad\mbox{and}\qquad
  \norm{\vvec}=N.
 \]
 \item The function $I_N(\z,t;\y)$ is bilinear in $(\z,t)$ and $\y$,
 jointly $3$-smooth, and a first-order zero-chain with coordinate
 ordering $\ord$ given by
 \[
  (\z,t,\y)_{\ord}=(z_1,y_1,\ldots,z_N,y_N,t).
 \]
 \item For every $\z,\y\in\R^N$ and $t\in\R$, the derivatives of
 $I_N$ at $(\z,t;\y)$ satisfy
 \[
  \ip{\w_0}{\nabla_{\z}I_N}+\nabla_t I_N=0,
  \qquad\mbox{and}\qquad
  |z_1-t|\le N\norm{\nabla_{\y}I_N}.
 \]
\end{enumerate}
\end{lemma}

\begin{proof}[Proof of Lemma~\ref{gs:mem:matrix-bounds}]
For $N=1$, the quadratic $I_1(z,t;y)=(z-t)^2/2$ has Hessian
norm two and, for each fixed $y$, is a zero-chain in the order $(z,t)$.
Henceforth, fix $N\ge2$.

For \textup{(i)}, let $\mathbf T_N$ have diagonal entries $1$,
subdiagonal entries $-1$, and zero elsewhere. Set
\[
 \bW_N:=\operatorname{diag}\bigl(
  \sqrt N,\underbrace{1,\ldots,1}_{N-2},\sqrt N
 \bigr).
\]
Then $\bD_N=\bW_N^{-1}\mathbf T_N$, with both factors invertible
and $\mathbf T_N^{-1}$ the lower triangular matrix of ones.
Since $\mathbf T_N$ is the identity minus a shift of norm at most one
and $\norm{\bW_N^{-1}}\le1$, we have $\norm{\bD_N}\le2$.
The Frobenius norm bounds the inverse:
\[
 \norm{\bD_N^{-1}}^2
 \le\norm{\mathbf T_N^{-1}\bW_N}_{\rm F}^2
 =N^2+N+\sum_{j=2}^{N-1}(N-j+1)
 =\frac{3N^2+N-2}{2}
 \le2N^2.
\]
The sum is empty when $N=2$.

For \textup{(ii)}, use $\w_0=\bW_N\ind/\sqrt N$ and
$\mathbf T_N^\top\ind=\ee_N$ to obtain
\[
  \bD_N^\top\w_0= \frac{\mathbf T_N^\top\ind}{\sqrt N}
                    =\frac{\ee_N}{\sqrt N},\qquad
  \norm{\w_0}^2=2+\frac{N-2}{N}=3-\frac{2}{N},\qquad
  \vvec=\mathbf T_N^{-1}\bW_N\ee_1=\sqrt N\,\ind.
\]

For \textup{(iii)}, bilinearity follows from the definition of $I_N$.
The two terms defining $I_N$ have joint Hessian norms
$\norm{\bD_N}$ and $1/\sqrt N$, respectively.
Hence $\norm{\nabla^2I_N}\le2+1/\sqrt N<3$.
Differentiation gives
\begin{equation}
 \nabla_{\z}I_N=\bD_N\y,\qquad
 \nabla_{\y}I_N=\bD_N^\top\z-\frac{t}{\sqrt N}\ee_N,
 \qquad\mbox{and}\qquad
 \nabla_t I_N=-\frac{y_N}{\sqrt N}.
 \label{gs:eq:inner-derivatives}
\end{equation}
Since $\bD_N$ is lower bidiagonal, each partial derivative
depends only on adjacent coordinates in the stated order.
Hence $I_N$ is a first-order zero-chain.

Finally, \eqref{gs:eq:inner-derivatives} and \textup{(ii)},
with $\bD_N\vvec=\ee_1$ and $v_N=\sqrt N$, give
\[
 \begin{aligned}
  \ip{\w_0}{\nabla_{\z}I_N}
  =\ & \ip{\bD_N^\top\w_0}{\y}
       =\frac{y_N}{\sqrt N}=-\nabla_t I_N,\\
  \ip{\vvec}{\nabla_{\y}I_N}
  =\ & \ip{\bD_N\vvec}{\z}-\frac{t}{\sqrt N}v_N
       =z_1-t.
 \end{aligned}
\]
Cauchy-Schwarz and $\norm{\vvec}=N$ prove the remaining bound
in \textup{(iv)}.
\end{proof}

\subsection{Properties of \texorpdfstring{$e_1$ and $e_2$}{e1 and e2}}
\label{gs:app:scalar-functions}

We prove the bounds for $e_1,e_2$ defined in
Section~\ref{gs:subsec:inner-block}, using the properties of $p$
in Lemma~\ref{lem:gate-transformations}.

\begin{lemma}[Properties of $e_1$ and $e_2$]
\label{gs:lem:affine-extension}
The functions $e_1,e_2$ belong to $C^\infty(\R)$.
The function $e_1$ equals $t+9c_s/2$ for $|t|\le3c_s$,
and equals $0$ for $t\le-20c_s$ and $9c_s$ for $t\ge20c_s$.
Its derivative is even and nonincreasing on $[0,\infty)$, with
\[
 \begin{gathered}
 0\le e_1(t)\le9c_s,\qquad
 0\le e_1'(t)\le1,\qquad
 \norm{e_1''}_{\infty}=\frac{15}{8c_s},\\
 e_1'(t)\ge\frac{1}{16}\qquad(|t|\le19c_s).
 \end{gathered}
\]
The function $e_2$ equals one for $t\le c_s$ and zero for
$t\ge2c_s$, and satisfies
\[
 0\le e_2(t)\le1,\qquad
 0\le-e_2'(t)\le\frac{4}{c_s}\sqrt{e_2(t)},\qquad
 \norm{e_2''}_{\infty}<\infty.
\]
\end{lemma}

\begin{proof}[Proof of Lemma~\ref{gs:lem:affine-extension}]
Lemma~\ref{lem:gate-transformations} gives $e_1,e_2\in C^\infty(\R)$.
The constant and affine branches of $p$ give the stated formulas
for $e_1$.
Differentiation and the symmetry $p'(r)+p'(1-r)=1$ give
\begin{equation}
 e_1'(t)=1-\frac{15}{16}p'\!\left(\frac{|t|}{c_s}-3\right)
           -\frac{1}{16}p'\!\left(\frac{|t|}{c_s}-19\right).
 \label{gs:eq:affine-extension-derivative}
\end{equation}
Since $p'$ is nondecreasing with values in $[0,1]$, this formula
gives $0\le e_1'\le1$, evenness, and monotonicity on $[0,\infty)$.
The last term vanishes for $|t|\le19c_s$, giving $e_1'(t)\ge1/16$.
Monotonicity of $e_1$ and its constant tails give $0\le e_1\le9c_s$.

On $(0,\infty)$, the two contributions to $e_1''$ can be nonzero
only on the disjoint intervals $(3c_s,4c_s)$ and $(19c_s,20c_s)$.
Thus $0\le p''\le2$ and evenness of $e_1'$ give
$|e_1''|\le15/(8c_s)$, with equality at $t=7c_s/2$ since
$p''(1/2)=2$ by~\eqref{eq:base-ramp-second-derivative}.

The plateau values and range $0\le e_2\le1$ follow from those
of $p'$. With $r=t/c_s-1$, differentiation gives
\[
 0\le-e_2'(t)=\frac{2}{c_s}[1-p'(r)]p''(r)
 \le\frac{4}{c_s}\sqrt{e_2(t)},
\]
where $1-p'(r)=\sqrt{e_2(t)}\ge0$ and $0\le p''(r)\le2$.
The continuous function $e_2''$ vanishes outside $[c_s,2c_s]$
and is therefore bounded.
\end{proof}

\begingroup

\subsection{Common identities and a GS potential bound}
\label{gs:app:derivative-identities}
For $\eta_i=q_{\rm en}(s_i)$ and $\sigma_i=q_{\rm ex}(s_i)$,
Lemma~\ref{gs:lem:affine-extension} and the chain rule give
\begin{equation}
 0\le\sigma_i\le\eta_i\le1,\qquad
 0\le\eta_i'\le\frac{16\sqrt{\eta_i}}{c_s},\qquad
 0\le\sigma_i'\le\frac{4\sqrt{1-\sigma_i}}{c_s}.
 \label{gs:eq:scalar-estimates}
\end{equation}
If $\eta_i<1$, then $\sigma_i=\sigma_i'=0$; if $\eta_i=1$,
then $\eta_i'=0$.

\begin{lemma}[Properties of $R_1$ and $R_2$]
\label{gs:lem:regularizers}
The functions $R_1$ and $R_2$ defined in
\eqref{gs:eq:state-regularizer} and~\eqref{gs:batch:tail-ramp}
belong to $C^\infty(\R)$, with $R_j(0)=R_j'(0)=0$ for $j=1,2$.
\begin{enumerate}[label=\textup{(\roman*)},leftmargin=*]
\item The state regularizer satisfies
\[
 R_1\ge-S_0,\qquad
 \norm{R_1''}_\infty\le\frac{64c_S}{c_s}.
\]
Its derivative obeys
\begin{equation}
 R_1'(s_i)=-8c_S\sqrt{\eta_i-\sigma_i}
 -24p'\!\left(-\frac{4s_i+c_s}{3c_s}\right)\le0.
 \label{gs:eq:state-derivative}
\end{equation}
In particular, $R_1'(t)=-24$ for $t\le-c_s$ and $R_1'(t)=0$ for
$t\ge2c_s$.
\item For $D\ge5$, the tail regularizer satisfies
\begin{equation}
 -R_0D^2\le R_2\le0,\qquad R_2'\le0,\qquad
 \norm{R_2''}_\infty\le4R_0.
 \label{gs:batch:ramp-bounds}
\end{equation}
It has $R_2(t)=R_2'(t)=0$ for $t\le0$, and
\begin{equation}
 \begin{aligned}
 R_2'(t)\le\ & -R_0 &&(1\le t\le D),\\
 R_2'(t)=\ & -R_0\left(1+\frac{D}{5}\right) &&\left(\frac{D}{5}\le t\le D\right).
 \end{aligned}
 \label{gs:batch:ramp-derivative}
\end{equation}
\end{enumerate}
\end{lemma}
\begin{proof}[Proof of Lemma~\ref{gs:lem:regularizers}]
Smoothness follows from Lemma~\ref{lem:gate-transformations}.
The constant and affine branches of $p$ give the values at the
origin and the stated formulas for $R_1'$ outside $(-c_s,2c_s)$.
The two transition intervals $(c_s/4,c_s/2)$ and $(c_s,2c_s)$
are disjoint, so
\[
 p'\left(\frac{4s_i}{c_s}-1\right)-p'\left(\frac{s_i}{c_s}-1\right)=\sqrt{\eta_i-\sigma_i}.
\]
Thus differentiation gives~\eqref{gs:eq:state-derivative};
$R_1'\le0$ and $R_1(t)=-S_0$ for $t\ge2c_s$ give $R_1\ge-S_0$.
The three terms in $R_1''$ have disjoint supports, so $0\le p''\le2$ gives
\[
 \norm{R_1''}_\infty
 \le\frac{2}{c_s}\max\{32c_S,8c_S,32\}
 =\frac{64c_S}{c_s}.
\]

For $D\ge5$, monotonicity and $1$-Lipschitzness of $p$ give
\[
 -R_0D^2\le-R_0\left(D+\frac{D^2}{5}\right)\le R_2(t)\le0.
\]
The differences of $p'$ in $R_2'$ are nonnegative, and those of
$p''$ in $R_2''$ have magnitude at most two. Hence $R_2'\le0$
and $\norm{R_2''}_\infty\le4R_0$. The branches of $p$ give
$R_2(t)=R_2'(t)=0$ for $t\le0$ and the derivative bounds in
\eqref{gs:batch:ramp-derivative}.
\end{proof}

Let $d_e(r)=1-e_1'(r)$ and $v_e(r)=e_1(r)-e_0-r$.
Both vanish on $[-3c_s,3c_s]$; $d_e$ is even,
nondecreasing on $[0,\infty)$, and $15/(8c_s)$-Lipschitz.
Integrating $d_e$ gives
\begin{equation}
 \frac{4c_s}{15}d_e(r)^2
 \le-\operatorname{sgn}(r)v_e(r)
 =\int_0^{|r|}d_e(v)\,dv
 \le|r|d_e(r).
 \label{gs:sc:tail-bounds}
\end{equation}
The lower bound integrates a triangle of height $d_e(r)$ and base
$8c_sd_e(r)/15$ supplied by Lipschitzness; the upper bound uses monotonicity.

\begin{lemma}[A lower bound for the local potential]
\label{gs:bnd:energy}
Let $M\ge1$ and $\s=(s_i)_i,\a=(a_i)_i\in\R^M$, with
$s_i>-c_s$ and $h_i=a_i-s_i$. Set $\eta_0=1$, $\eta_i=q_{\rm en}(s_i)$,
$\sigma_i=q_{\rm ex}(s_i)$, and assume $\eta_M=\sigma_M=0$.
Let $P:=P(\s,\a,1)$ be the potential in~\eqref{gs:batch:potential}.
Suppose $\alpha\in[0,1]$, $\delta_M=0$, and, for $i<M$,
\begin{equation}
 \delta_i\le8\alpha\Bigl[
 -1+\sigma_i+\sigma_i'[s_i+e_1(h_i)]
 +(1-\eta_{i-1})e_1'(a_i)-\eta_i'e_1(a_{i+1})\Bigr]+R_1'(s_i).
 \label{gs:eq:clock-residual}
\end{equation}
Then, writing $\bm h=(h_i)_i$ and $\bm\delta=(\delta_i)_i$,
\begin{equation}
 P\ge-4c_s-\frac{120}{c_s}\norm{\bm h}^2
                 -\frac{c_s}{800}\norm{\bm\delta}^2.
 \label{gs:bnd:residual-energy}
\end{equation}
In particular, the final state may be any value in
$(-c_s,c_s/4)$; it need not be fixed at zero.
\end{lemma}
\begin{proof}[Proof of Lemma~\ref{gs:bnd:energy}]
Write $\Delta\eta_i:=\eta_i-\eta_{i-1}$ and
$\theta_i:=\eta_i-\sigma_i$, and extend the states by
$s_0=s_{M+1}:=c_s/2$ and $\eta_{M+1}:=1$.
The function $q_{\rm en}$ is nondecreasing and $16/c_s$-Lipschitz,
and $e_1'=1$ on its transition interval $(c_s/4,c_s/2)$.
Each integral below is therefore at least the area of a triangle
of height $|\Delta\eta_i|$ and base $c_s|\Delta\eta_i|/16$:
\begin{equation}
 \sum_{i=1}^M\Delta\eta_i e_1(s_i)+e_0+\frac{c_s}{2}
 =\sum_{i=1}^{M+1}\int_{s_{i-1}}^{s_i}
      [q_{\rm en}(v)-\eta_{i-1}]e_1'(v)\,dv
 \ge\frac{c_s}{32}\sum_{i=1}^M(\Delta\eta_i)^2.
 \label{gs:bnd:cyclic-credit}
\end{equation}
The equality uses $s_{M+1}=s_0$ to cancel the integrals of
$q_{\rm en}e_1'$; the inequality drops the nonnegative terminal contribution.

Since $\eta_i<1$ implies $-c_s<s_i<c_s/2$, we have
$e_1(s_i)=s_i+e_0$ and $e_1'(s_i)=1$ whenever $\eta_i<1$.
Expanding $P$ at $\a=\s$ gives
\[
 \begin{aligned}
 \frac18 P(\s,\s,1)
 =\ &e_0+\sum_i\Delta\eta_i e_1(s_i)-\sum_i\theta_i(s_i+e_0),\\
 \left.\frac{1}{8}\nabla_{a_i}P\right|_{\a=\s}
 =\ &e_1'(s_i)\Delta\eta_i-\theta_i.
 \end{aligned}
\]
Here $\theta_i(s_i+e_0)\le(13c_s/2)\theta_i$, since
$\theta_i>0$ implies $s_i<2c_s$.
The Hessian in $\a$ is diagonal, with
$\norm{\nabla_{\a}^2(P/8)}\le15/(4c_s)$.
Taylor's theorem,~\eqref{gs:bnd:cyclic-credit}, and
\[
 |\Delta \eta_i h_i|
 \le\frac{c_s}{48}(\Delta \eta_i)^2+\frac{12}{c_s}h_i^2,
 \qquad
 \theta_i|h_i|\le\frac{c_s}{2}\theta_i+\frac{h_i^2}{2c_s}
\]
therefore give
\begin{equation}
 \frac{P}{8}\ge-\frac{c_s}{2}
 +\frac{c_s}{96}\sum_i(\Delta\eta_i)^2
 -7c_s\sum_i\theta_i
 -\frac{115}{8c_s}\norm{\bm h}^2.
 \label{gs:bnd:retained-energy}
\end{equation}

For $i<M$,~\eqref{gs:eq:clock-residual} and $0\le\alpha\le1$ yield
\[
 \delta_i\le
 \begin{cases}
 -8c_S\sqrt{\theta_i},&\eta_i<1,\\
 8|\Delta\eta_i|-8(c_S-44)\sqrt{\theta_i},&\eta_i=1.
 \end{cases}
\]
The first case uses $\sigma_i=\sigma_i'=0$; the second uses
$\eta_i'=0$ and $\sigma_i'[s_i+e_1(h_i)]\le44\sqrt{\theta_i}$.
Since $c_S=192$ and $\theta_M=0$, these bounds imply
\begin{equation}
 147\sqrt{\theta_i}\le|\Delta\eta_i|+\frac{|\delta_i|}{8},
 \qquad
 \sum_i\theta_i\le\frac{2}{147^2}
 \left[\sum_i(\Delta\eta_i)^2+\frac{\norm{\bm\delta}^2}{64}\right].
 \label{gs:bnd:transition}
\end{equation}
Substitution into~\eqref{gs:bnd:retained-energy} leaves a
nonnegative coefficient for $\sum_i(\Delta\eta_i)^2$, since
$1/96-14/147^2>0$. Dropping it gives
\[
 P\ge-4c_s-\frac{115}{c_s}\norm{\bm h}^2
              -\frac{7c_s}{4\cdot147^2}\norm{\bm\delta}^2,
\]
which implies~\eqref{gs:bnd:residual-energy} because
$7/(4\cdot147^2)<1/800$.
\end{proof}
\endgroup

\section{Technical Details of \FOAM{}}
\label{app:foam-details}

This appendix provides the technical details of the FOAM
subroutine, including its feasible implementation,
contraction guarantee, GS stationarity certificate,
and initialization.
Throughout, $r=(2\ell,r_{\y})$, and
$P_{\z,r_{\y}}$ and $\mathcal L_{\z,r_{\y}}$ denote the
conjugate objective and Lyapunov function defined
in Section~\ref{sec:upper}.

\subsection{Feasible implementation and contraction of \FOAM{}}
\label{app:foam-implementation}
We use the accelerated FOAM updates of
\citet{kovalev2022first}, specialized to the
conjugate objective $P_{\z,r_{\y}}$ (see equation~\eqref{eq:P_new}).
The updates themselves are unchanged.
Unlike the standard use of FOAM to solve a fixed
SC-SC problem, \TrackedFOAM{} retains the full
FOAM state across successive proximal subproblems. We therefore verify that the Lyapunov contraction
applies to arbitrary states with finite Lyapunov
value and provide a constant-cost implementation
using only feasible first-order saddle-oracle queries.

Fix a proximal center $\z\in\cX$ and
$0<r_{\y}\leq\ell/8$, and define
\[
\widehat F_{\z}(\x;\y)
:=
F_r(\x,\z;\y)
-\frac{\ell}{2}\|\x\|^2
+\frac{r_{\y}}{2}\|\y\|^2.
\]
The function $\widehat F_{\z}$ is convex in $\x$,
concave in $\y$, and independent of $r_{\y}$.

Each FOAM iteration requires a proximal update,
which we compute approximately.
The parameter $\tau$ controls the accuracy
required for this update: a smaller $\tau$
imposes a stronger accuracy requirement.
We use $\tau=1$ for standard FOAM iterations 
and $\tau=1/4$ for the final GS extraction.
Both choices are fixed numerical constants,
independent of the target accuracy $\epsilon$.
Given $(\omegavec_g,\y_g)\in\R^m\times\R^n$,
an inexact proximal output is a tuple
$(\x_f,\y_f^+,\omegavec_f^+,\w_f^+)\in
\cX\times\cY\times\R^m\times\R^n$
satisfying the following conditions.
Define
\begin{align*}
  \bm\Delta_x
  :=\omegavec_f^+
      +\frac{\ell}{2}(\x_f-\ell^{-1}\omegavec_g),\quad\mbox{and} \quad
  \bm\Delta_y
   :=\w_f^++r_{\y}\y_f^+
      +\frac{\ell}{8}(\y_f^+-\y_g).
\end{align*}
We require
\begin{align}
  \omegavec_f^+-\nabla_{\x}F_r(\x_f,\z;\y_f^+)+\ell\x_f
   \in\ & N_{\cX}(\x_f),
  \label{eq:app-rprox-normal-x}\\*
  \w_f^++\nabla_{\y}F_r(\x_f,\z;\y_f^+)+r_{\y}\y_f^+
   \in\ & N_{\cY}(\y_f^+),\qquad\mbox{and}
  \label{eq:app-rprox-normal-y}\\*
   \frac{8}{\ell}\norm*{\bm\Delta_x}^2
   +\frac{8}{\ell}\norm*{\bm\Delta_y}^2
   \le\ & \tau^2\left(
      \frac{\ell}{8}\norm*{\x_f+\ell^{-1}\omegavec_g}^2
      +\frac{\ell}{8}\norm*{\y_f^+-\y_g}^2\right).
  \label{eq:app-rprox-certificate}
\end{align}
For $\tau=1$, these conditions coincide with those in
\citet[Condition~(18)]{kovalev2022first}
under $\mu_x=\ell$, $\mu_y=r_{\y}$,
$\theta_y=8/\ell$, and indicator regularizers
$\iota_{\cX}$ and $\iota_{\cY}$.
 
Set
$\alpha_{r_{\y}}:=\sqrt{8r_{\y}/\ell}\in(0,1]$,
$\eta_{\omega}:=\ell/2$, and
$\eta_{\y}:=4/(\alpha_{r_{\y}}\ell)$.
Algorithm~\ref{alg:app-foam-step} gives the resulting FOAM update,
mapping the current state $S$ to $S^+$.

\begin{algorithm}[H]
\small
\caption{One \FOAM{} step at $(\z,r_{\y})$}
\label{alg:app-foam-step}
\KwData{$S=(\omegavec,\y,\omegavec_f,\y_f)$}
Set $(\omegavec_g,\y_g)
=\alpha_{r_{\y}}(\omegavec,\y)
 +(1-\alpha_{r_{\y}})(\omegavec_f,\y_f)$\;
Compute $(\x_f,\y_f^+,\omegavec_f^+,\w_f^+)$ satisfying
\eqref{eq:app-rprox-normal-x}-\eqref{eq:app-rprox-certificate}
with $\tau=1$\;
Set $\omegavec^+
=\omegavec+\eta_\omega\ell^{-1}(\omegavec_f^+-\omegavec)
 -\eta_\omega(\x_f+\ell^{-1}\omegavec_f^+)$\;
Set $\y^+
=\y+\eta_{\y} r_{\y}(\y_f^+-\y)
 -\eta_{\y}(\w_f^++r_{\y}\y_f^+)$\;
\KwRet{$S^+=(\omegavec^+,\y^+,\omegavec_f^+,\y_f^+)$}\;
\end{algorithm}

We next show that the inexact proximal output required
in Algorithm~\ref{alg:app-foam-step} can be computed using only feasible
first-order saddle-oracle queries.
Set $\gamma:=8/\ell$ and define $ \mathcal{C}':=\gamma^{-1/2}(\cX\times\cY)$.
For $\bm u=\gamma^{-1/2}(\x,\y)\in\mathcal{C}'$, let
\[
  A_g(\bm u):=\sqrt{\gamma}
  \begin{pmatrix}
    \nabla_{\x}\widehat F_{\z}(\x;\y)
    +\frac{\ell}{2}(\x-\ell^{-1}\omegavec_g)\\[1mm]
    -\nabla_{\y}\widehat F_{\z}(\x;\y)
    +r_{\y}\y+\gamma^{-1}(\y-\y_g)
  \end{pmatrix}.
\]

\begin{algorithm}[h]
\small
\caption{Projected implementation of the inexact proximal step}
\label{alg:app-feasible-rprox}
\KwData{$\z$, $0<r_{\y}\le\ell/8$,
$(\omegavec_g,\y_g)\in\R^m\times\R^n$,
and $\tau\in\{1/4,1\}$ (default $\tau=1$)}
% Set $\bar{\bm u}
% =\gamma^{-1/2}(-\ell^{-1}\omegavec_g,\y_g)$
% and $\bm u^0=\proj_{\mathcal{C}'}(\bar{\bm u})$\;
Set the reference point
$\bar{\bm u}
=\gamma^{-1/2}(-\ell^{-1}\omegavec_g,\y_g)$
and initialize
$\bm u^0=\proj_{\mathcal C'}(\bar{\bm u})$\;
Set $M_0=32$, $\eta=M_0^{-2}$, and $s=0$\;
\Repeat{$\norm{A_g(\bm u^s)+\bm b^s}
          \le\tau\norm{\bm u^s-\bar{\bm u}}$}{
  Set $s=s+1$\;
  Set $\bm u^s
  =\proj_{\mathcal{C}'}(\bm u^{s-1}-\eta A_g(\bm u^{s-1}))$\;
  Set $\bm b^s
  =\eta^{-1}(\bm u^{s-1}-\bm u^s)-A_g(\bm u^{s-1})$\;
}
Set $(\x_f,\y_f^+)=\sqrt{\gamma}\bm u^s$
and $(\bm b_x,\bm b_y)=\gamma^{-1/2}\bm b^s$\;
Set $\omegavec_f^+
=\nabla_{\x}\widehat F_{\z}(\x_f;\y_f^+)+\bm b_x$
and $\w_f^+
=-\nabla_{\y}\widehat F_{\z}(\x_f;\y_f^+)+\bm b_y$\;
\KwRet{$(\x_f,\y_f^+,\omegavec_f^+,\w_f^+)$}\;
\end{algorithm}

\begin{lemma}[Feasible implementation of the inexact proximal step]
\label{lem:app-feasible-rprox}
For either fixed tolerance $\tau\in\{1/4,1\}$,
Algorithm~\ref{alg:app-feasible-rprox} terminates after a universal
constant number of iterations and oracle calls, evaluates gradients
only in $\cX\times\cY$, and returns a tuple satisfying
\eqref{eq:app-rprox-normal-x}-\eqref{eq:app-rprox-certificate}.
\end{lemma}

\begin{proof}[Proof of Lemma~\ref{lem:app-feasible-rprox}]
\begin{samepage}
Since $\widehat F_{\z}$ is convex-concave,
its saddle operator is monotone.
Consequently, for any
$\bm u,\bm v\in\mathcal C'$, we have
\begin{align*}
 \ip{A_g(\bm u)-A_g(\bm v)}{\bm u-\bm v} \ge
\frac{\gamma\ell}{2}
\norm{\bm u_x-\bm v_x}^2
+\gamma(r_{\y}+\gamma^{-1})
\norm{\bm u_y-\bm v_y}^2 \ge\norm{\bm u-\bm v}^2,
\end{align*}
where the last inequality uses $\gamma=8/\ell$.
\end{samepage}
Thus $A_g$ is $1$-strongly monotone,
uniformly over $0<r_{\y}\le\ell/8$. 
Since $f$ is jointly $\ell$-smooth, $\nabla\widehat F_{\z}$
is $2\ell$-Lipschitz. After scaling, the saddle operator has
Lipschitz constant at most $16$, while the added affine terms
have Lipschitz constant at most $4$. Thus $M_0=32$ is a valid
Lipschitz bound for $A_g$.

By strong monotonicity, Lipschitz continuity,
and projection nonexpansiveness, the map $ T_g(\bm u):=
  \proj_{\mathcal C'}(\bm u-\eta A_g(\bm u))$
is a contraction with factor
$\chi:=\sqrt{1-M_0^{-2}}<1$.
Let $\bm u^\star$ denote its unique fixed point,
and set $d_{\rm vi}:=\norm{\bar{\bm u}-\bm u^\star}$.
Since $\bm u^0=\proj_{\mathcal C'}(\bar{\bm u})$,
projection nonexpansiveness and the contraction
property give
\[
  \norm*{\bm u^s-\bm u^\star}
  \le\chi^s d_{\rm vi},
  \qquad s\ge0.
\]
% Projection nonexpansiveness and these bounds show that
% $\bm u\mapsto\proj_{\mathcal{C}'}(\bm u-\eta A_g(\bm u))$
% is a contraction with factor
% $\chi:=\sqrt{1-M_0^{-2}}<1$.
% Its unique fixed point $\bm u^\star$ satisfies
% $0\in A_g(\bm u^\star)+N_{\mathcal{C}'}(\bm u^\star)$.
% Let $d_{\rm vi}:=\norm{\bar{\bm u}-\bm u^\star}$.
% Since $\bm u^0=\proj_{\mathcal{C}'}(\bar{\bm u})$, we have
% \[
%   \norm*{\bm u^s-\bm u^\star}\le\chi^s d_{\rm vi},
%   \qquad s\ge0.
% \]
Consequently, for $s\ge1$,
\begin{align*}
      \norm*{A_g(\bm u^s)+\bm b^s}
  \le\ &(M_0+M_0^2)
    \norm*{\bm u^s-\bm u^{s-1}} \le(M_0+M_0^2)(1+\chi)
    \chi^{s-1}d_{\rm vi},\qquad\mbox{and}\\[2pt]
   \norm*{\bm u^s-\bar{\bm u}}
   \ge\ &(1-\chi^s)d_{\rm vi}.
\end{align*}
Since $\chi<1$ is a numerical constant, a universal
integer $s$ makes
$(M_0+M_0^2)(1+\chi)\chi^{s-1}\le\tau(1-\chi^s)$
for both $\tau=1$ and $\tau=1/4$.
If $d_{\rm vi}=0$, the first projected step already satisfies the
stopping condition. Thus either tolerance requires only a universal
constant number of iterations.
Every iterate $\bm u^s$, $s\ge0$, lies in $\mathcal{C}'$,
so all gradient evaluations occur in $\cX\times\cY$.
Each iteration and the final output require only a constant number
of first-order saddle-oracle calls.

It remains to verify that the returned tuple satisfies
\eqref{eq:app-rprox-normal-x}-\eqref{eq:app-rprox-certificate}.
Projection optimality gives
$\bm b^s\in N_{\mathcal{C}'}(\bm u^s)$.
The output therefore satisfies
\begin{align*}
  \omegavec_f^+-\nabla_{\x}F_r(\x_f,\z;\y_f^+)+\ell\x_f
  =\ & \bm b_x\in N_{\cX}(\x_f),\qquad\mbox{and}\\
  \w_f^++\nabla_{\y}F_r(\x_f,\z;\y_f^+)+r_{\y}\y_f^+
  =\ & \bm b_y\in N_{\cY}(\y_f^+).
\end{align*}
To verify the remaining relative-error condition, note that
\begin{align*}
  \norm*{A_g(\bm u^s)+\bm b^s}^2
  =\ & \frac{8}{\ell}\norm*{\bm\Delta_x}^2
    +\frac{8}{\ell}\norm*{\bm\Delta_y}^2,
    \qquad\mbox{and}\\
  \norm*{\bm u^s-\bar{\bm u}}^2
  =\ & \frac{\ell}{8}\norm*{\x_f+\ell^{-1}\omegavec_g}^2
    +\frac{\ell}{8}\norm*{\y_f^+-\y_g}^2.
\end{align*}
Thus the stopping condition gives
\eqref{eq:app-rprox-certificate}.
\end{proof}

For any $\rho\in(0,1)$, set
\[
  K_\rho:=
  \left\lceil
    \frac{2}{\alpha_{r_{\y}}}
    \log\frac{1}{\rho}
  \right\rceil.
\]
We define $\FOAM_{\z,r_{\y}}(S,\rho)$ as the
state obtained after $K_\rho$ applications of
Algorithm~\ref{alg:app-foam-step}, each using
Algorithm~\ref{alg:app-feasible-rprox}
with $\tau=1$.

\begin{proof}[Proof of Lemma~\ref{prop:foam-interface}]
Using $\omegavec_{\z,r_{\y}}^\star=-\ell\x_{\z,r_{\y}}^\star$,
projection nonexpansiveness, and strong convexity of
$P_{\z,r_{\y}}$, we obtain
\begin{align*}
  \ell\norm*{\proj_{\cX}\left(-\frac{\omegavec_f}{\ell}\right)
                 -\x_{\z,r_{\y}}^\star}^2
  \le\ & \frac{1}{\ell}
       \norm*{\omegavec_f-\omegavec_{\z,r_{\y}}^\star}^2\le2\left(P_{\z,r_{\y}}(\omegavec_f,\y_f)
             -P_{\z,r_{\y}}^\star\right)
   \le\mathcal{L}_{\z,r_{\y}}(S),
\end{align*}
where the last inequality follows from \eqref{eq:tracking-energy}.
This proves \eqref{eq:interface-primal-error}.
With $\mu_x=\ell$, $\mu_y=r_{\y}$, and $\theta_y=8/\ell$,
\citet[Appendix~D, Lemma~7]{kovalev2022first} give
\[
  \mathcal{L}_{\z,r_{\y}}(S^+)
  \le\left(1-\frac{\alpha_{r_{\y}}}{2}\right)
       \mathcal{L}_{\z,r_{\y}}(S).
\]
The proof of this recursion does not require the slow and fast
pairs to coincide at initialization. It applies to any state
with finite Lyapunov value when the inexact proximal output satisfies
\eqref{eq:app-rprox-normal-x}-\eqref{eq:app-rprox-certificate}.
Applying it for $K_\rho$ steps gives
\[
  \mathcal{L}_{\z,r_{\y}}\left(\FOAM_{\z,r_{\y}}(S,\rho)\right)
  \le\left(1-\frac{\alpha_{r_{\y}}}{2}\right)^{K_\rho}
       \mathcal{L}_{\z,r_{\y}}(S)
  \le\rho\,\mathcal{L}_{\z,r_{\y}}(S).
\]
By Lemma~\ref{lem:app-feasible-rprox}, each step uses a universal
constant number of first-order saddle-oracle calls. Hence the
total oracle cost is
\[
  \cO(K_\rho)
  =\cO\left(1+\sqrt{\frac{\ell}{r_{\y}}}\log\frac{1}{\rho}\right).
\]
For $\rho\in\{1/8,1/400\}$, as used in
\TrackedFOAM{}, this reduces to
$\cO(\sqrt{\ell/r_{\y}})$.
\end{proof}
 
\subsection{Proof of Lemma~\ref{lem:gs-certificate}}
\label{app:gs-certificate-proof}
\begin{proof}[Proof of Lemma~\ref{lem:gs-certificate}]
Let $(\x_f,\y_f^+,\omegavec_f^+,\w_f^+)$ be the output of
Algorithm~\ref{alg:app-feasible-rprox} with
$(\omegavec_g,\y_g)=(\omegavec_f,\y_f)$ and $\tau=1/4$.
Set $\bm u_0:=(\omegavec_f,\y_f)$,
$\bm u_+:=(\omegavec_f^+,\y_f^+)$, and $\bm s:=(\bm s_\omega,\bm s_y)
  :=(\x_f+\ell^{-1}\omegavec_f^+,
     \w_f^++r_{\y}\y_f^+)$.

We first identify the quantities needed to control the GS residuals.
Let $\bm n_{\x}$ and $\bm n_{\y}$ denote the normal vectors in
\eqref{eq:app-rprox-normal-x}-\eqref{eq:app-rprox-normal-y}.
Using \eqref{eq:moreau1} and
$\omegavec_{\z,r_{\y}}^\star=-\ell\x_{\z,r_{\y}}^\star$,
we obtain
\begin{align*}
  \nabla_{\x}f(\x_f;\y_f^+)+\bm n_{\x}
   =\ & \nabla p_r(\z)
    +2(\omegavec_f^+-\omegavec_{\z,r_{\y}}^\star)
    -\ell\bm s_\omega, \quad \text{and}\\
  -\nabla_{\y}f(\x_f;\y_f^+)+\bm n_{\y}
  =\ & \bm s_y-r_{\y}\y_f^+.
\end{align*}
Consequently, we have
\begin{equation}\label{eq:res}
\begin{aligned}
  R_x
 \le\ & \norm{\nabla p_r(\z)}
    +2\norm{\omegavec_f^+-\omegavec_{\z,r_{\y}}^\star}
    +\ell\norm{\bm s_\omega}, \quad \text{and}\\
  R_y
 \le\ & \norm{\bm s_y}+r_{\y}D_{\cY}.
\end{aligned}
\end{equation}
It remains to bound $\bm s$ and
$\omegavec_f^+-\omegavec_{\z,r_{\y}}^\star$
in terms of the FOAM Lyapunov function.

We first show that $\bm s$ is a subgradient of
$P_{\z,r_{\y}}$ at $\bm u_+$.
The primal normal relation
\eqref{eq:app-rprox-normal-x} gives
$\omegavec_f^+-\nabla_{\x}\widehat F_{\z}(\x_f;\y_f^+)
\in N_{\cX}(\x_f)$.
Since $\widehat F_{\z}(\cdot;\y_f^+)$ is convex,
$\x_f$ attains the supremum defining
$P_{\z,r_{\y}}(\bm u_+)$.
For any $\bm u'=(\omegavec',\y')$ with
$\y'\in\cY$, evaluating the supremum at $\x_f$
and using convexity of
$-\widehat F_{\z}(\x_f;\cdot)$ together with
\eqref{eq:app-rprox-normal-y} gives
\[
 P_{\z,r_{\y}}(\bm u')
  \ge P_{\z,r_{\y}}(\bm u_+)
 +\ip{\bm s_\omega}{\omegavec'-\omegavec_f^+} +\ip{\bm s_y}{\y'-\y_f^+}.
\]
The inequality also holds for $\y'\notin\cY$,
where $P_{\z,r_{\y}}(\bm u')=+\infty$.
Hence $\bm s\in\partial P_{\z,r_{\y}}(\bm u_+)$.

We next bound this subgradient using the
inexact proximal condition.
Set $\mathbf H:=\operatorname{diag}
(\ell^{-1}\bI_m,\ell\bI_n/8)$,
$\bm d:=\bm u_+-\bm u_0$, and
$\bm e:=\bm s+\mathbf H\bm d$.
By the definitions of $\bm\Delta_x$ and
$\bm\Delta_y$, we have
\[
  \bm e_\omega=\frac{2\bm\Delta_x}{\ell},
  \qquad
  \bm e_y=\bm\Delta_y,
  \qquad
  \x_f+\ell^{-1}\omegavec_f
  =\bm e_\omega-\frac{2\bm d_\omega}{\ell}.
\]
Since
$\norm{\bm e}_{\mathbf H^{-1}}^2
=(4/\ell)\norm{\bm\Delta_x}^2
+(8/\ell)\norm{\bm\Delta_y}^2$,
condition~\eqref{eq:app-rprox-certificate}
with $\tau=1/4$ yields
\[
 \norm{\bm e}_{\mathbf H^{-1}}^2
  \le\frac{1}{16}\frac{\ell}{8}
  \left(
    \norm*{\bm e_\omega-\frac{2\bm d_\omega}{\ell}}^2
    +\norm{\bm d_y}^2
  \right) \le\frac{1}{16}
  \left(
    \frac14\norm{\bm e}_{\mathbf H^{-1}}^2
    +\norm{\bm d}_{\mathbf H}^2
  \right).
\]
Consequently,
$\norm{\bm e}_{\mathbf H^{-1}}
\le2\norm{\bm d}_{\mathbf H}/\sqrt{63}
\le\norm{\bm d}_{\mathbf H}/3$.
Since $\bm s=\bm e-\mathbf H\bm d$,
the subgradient inequality gives
\[
  P_{\z,r_{\y}}(\bm u_0)-P_{\z,r_{\y}}(\bm u_+)
  \ge\norm{\bm d}_{\mathbf H}^2-\ip{\bm e}{\bm d}
  \ge\frac23\norm{\bm d}_{\mathbf H}^2.
\]
Together with
$P_{\z,r_{\y}}(\bm u_0)-P_{\z,r_{\y}}^\star
\le\mathcal L_{\z,r_{\y}}(S)/2$
from \eqref{eq:tracking-energy}, this yields
\begin{align}
  \norm{\bm s}_{\mathbf H^{-1}}^2
  \le\frac{16}{9}\norm{\bm d}_{\mathbf H}^2
  \le\frac83
  \bigl(P_{\z,r_{\y}}(\bm u_0)
        -P_{\z,r_{\y}}(\bm u_+)\bigr)
   \le\frac43\mathcal L_{\z,r_{\y}}(S).
  \label{eq:gs-conjugate-subgradient-size}
\end{align}
Moreover, the preceding descent inequality gives
$P_{\z,r_{\y}}(\bm u_+)\le P_{\z,r_{\y}}(\bm u_0)$.
By strong convexity of $P_{\z,r_{\y}}$, we obtain
\[
  \norm{\omegavec_f^+-\omegavec_{\z,r_{\y}}^\star}^2
  \le2\ell\bigl(
    P_{\z,r_{\y}}(\bm u_+)-P_{\z,r_{\y}}^\star
  \bigr)
  \le\ell\mathcal L_{\z,r_{\y}}(S).
\]

Finally, since
$\norm{\bm s}_{\mathbf H^{-1}}^2
=\ell\norm{\bm s_\omega}^2
+(8/\ell)\norm{\bm s_y}^2$,
\eqref{eq:gs-conjugate-subgradient-size} gives
$\ell\norm{\bm s_\omega}^2
\le4\mathcal L_{\z,r_{\y}}(S)/3$
and
$\norm{\bm s_y}^2
\le\ell\mathcal L_{\z,r_{\y}}(S)/6$.
Substituting these bounds into the GS residual estimates~\eqref{eq:res} yields
\begin{align*}
  R_x
 \le\ & \norm{\nabla p_r(\z)}
    +\left(2+\frac{2}{\sqrt3}\right)
     \sqrt{\ell\mathcal L_{\z,r_{\y}}(S)}, \quad\text{and}\\
  R_y
  \le\ & \sqrt{\frac{\ell\mathcal L_{\z,r_{\y}}(S)}{6}}
    +r_{\y}D_{\cY}.
\end{align*}
The constant oracle cost follows from
Lemma~\ref{lem:app-feasible-rprox}.
\end{proof}
 
\subsection{Proof of Lemma~\ref{lem:startup-state}}
\label{app:base-initialization}
\begin{proof}[Proof of Lemma~\ref{lem:startup-state}]
Let $\z^0=\mathbf{0}$ and $r_{\y}^0=\ell/8$.
Run Algorithm~\ref{alg:app-feasible-rprox} at
$(\z^0,r_{\y}^0)$ with
$(\omegavec_g,\y_g)=(\mathbf{0},\mathbf{0})$
and $\tau=1$.
Let $(\x_f,\y_f,\omegavec_f,\w_f)$ denote its output,
and set $\widehat S^0:=(\omegavec_f,\y_f,\omegavec_f,\y_f)$.
Set $\bm u_f:=(\omegavec_f,\y_f)$,
let $\bm u^\star$ minimize $P_{\z^0,r_{\y}^0}$,
and define
$\mathbf H^0:=\operatorname{diag}
(\ell^{-1}\bI_m,r_{\y}^0\bI_n)$.
Since $\alpha_{r_{\y}^0}=1$,
\[
  \mathcal L_{\z^0,r_{\y}^0}(\widehat S^0)
  =2\norm{\bm u_f-\bm u^\star}_{\mathbf H^0}^2
  +2\bigl(P_{\z^0,r_{\y}^0}(\bm u_f)
          -P_{\z^0,r_{\y}^0}^\star\bigr).
\]
We bound both terms using two estimates
derived from the inexact proximal conditions.

Define
$\bm s:=(\x_f+\ell^{-1}\omegavec_f,
\w_f+r_{\y}^0\y_f)$ and
$\bm e:=\bm s+\mathbf H^0\bm u_f$. In the following, we
will show that
\begin{align}\label{eq:aim}
     \ip{\bm s}{\bm u_f-\bm u^\star}
  \ge\norm{\bm u_f-\bm u^\star}_{\mathbf H^0}^2,
  \qquad
  \norm{\bm e}_{(\mathbf H^0)^{-1}}
  \le\norm{\bm u_f}_{\mathbf H^0}.
\end{align}
To see how these bounds control the Lyapunov
function, set
$d_0:=\norm{\bm u^\star}_{\mathbf H^0}$,
$d_f:=\norm{\bm u_f-\bm u^\star}_{\mathbf H^0}$,
and $d_{0f}:=\norm{\bm u_f}_{\mathbf H^0}$.
Since $\bm s=\bm e-\mathbf H^0\bm u_f$,
the two estimates imply
\[
  \ip{\mathbf H^0\bm u_f}{\bm u_f-\bm u^\star}
  \le d_{0f}d_f-d_f^2.
\]
Consequently, we have
\[
 d_0^2 =d_f^2+d_{0f}^2
    -2\ip{\mathbf H^0\bm u_f}
           {\bm u_f-\bm u^\star} \ge3d_f^2-2d_fd_{0f}+d_{0f}^2.
\]
Completing the square gives
$d_f^2\le d_0^2/2$ and
$d_{0f}^2\le3d_0^2/2$.
Moreover, if $\bm s\in\partial P_{\z^0,r_{\y}^0}(\bm u_f)$,
the subgradient inequality yields
\[
P_{\z^0,r_{\y}^0}(\bm u_f)
  -P_{\z^0,r_{\y}^0}^\star
 \le\ip{\bm s}{\bm u_f-\bm u^\star} \le\norm{\bm s}_{(\mathbf H^0)^{-1}}d_f
  \le2d_{0f}d_f
  \le\sqrt{3}\,d_0^2.
\]
Thus, once the two estimates and the
subgradient membership are established,
\[
  \mathcal L_{\z^0,r_{\y}^0}(\widehat S^0)
  \le(1+2\sqrt{3})d_0^2\le5d_0^2.
\]

We now verify these properties.
The normal-cone conditions
\eqref{eq:app-rprox-normal-x}-
\eqref{eq:app-rprox-normal-y}
and convexity of $\widehat F_{\z^0}$ in $\x$ and concavity in $\y$ give
$\bm s\in\partial P_{\z^0,r_{\y}^0}(\bm u_f)$.
Since $P_{\z^0,r_{\y}^0}$ is strongly convex
in the $\mathbf H^0$-metric, we have $ \ip{\bm s}{\bm u_f-\bm u^\star}
  \ge\norm{\bm u_f-\bm u^\star}_{\mathbf H^0}^2$. This proves the first property.
For the second estimate, the definitions of
$\bm\Delta_x$ and $\bm\Delta_y$ give
$\bm e_\omega=2\bm\Delta_x/\ell$
and $\bm e_y=\bm\Delta_y$.
Using $r_{\y}^0=\ell/8$ and $\tau=1$ in
\eqref{eq:app-rprox-certificate}, we obtain
\[
\norm{\bm u_f}_{\mathbf H^0}^2
   -\norm{\bm e}_{(\mathbf H^0)^{-1}}^2  \ge
  \frac{1}{2\ell}
  \norm*{\frac{\ell}{2}\x_f-2\bm\Delta_x}^2
  +\frac{3}{\ell}\norm{\bm\Delta_x}^2
  \ge0,
\]
where we used
$\bm\Delta_x=\omegavec_f+\ell\x_f/2$.
This establishes both estimates and hence
$\mathcal L_{\z^0,r_{\y}^0}(\widehat S^0)\le5d_0^2$.

It remains to bound $d_0^2$ by the initial gap.
Set $r^0:=(2\ell,r_{\y}^0)$.
The optimality of $\x_{\z^0,r_{\y}^0}^\star$ gives
\[
  \Phi_{r^0}(\x_{\z^0,r_{\y}^0}^\star)
  +\ell\norm{\x_{\z^0,r_{\y}^0}^\star-\z^0}^2
  \le\Phi_{r^0}(\z^0).
\]
Using
$\omegavec_{\z^0,r_{\y}^0}^\star
=-\ell\x_{\z^0,r_{\y}^0}^\star$,
$\z^0=\mathbf{0}$,
$\norm{\y_{\z^0,r_{\y}^0}^\star}\le D_{\cY}$,
and \eqref{eq:value-bias}, we obtain
\begin{align*}
  d_0^2
  =\ell\norm{\z^0-\x_{\z^0,r_{\y}^0}^\star}^2
    +r_{\y}^0\norm{\y_{\z^0,r_{\y}^0}^\star}^2 \le\Phi(\z^0)-\inf_{\x\in\cX}\Phi(\x)
    +\frac32r_{\y}^0D_{\cY}^2.
\end{align*}
Combining this bound with the Lyapunov estimate
proves the lemma. The constant oracle cost follows
from Lemma~\ref{lem:app-feasible-rprox}.
\end{proof}
\end{document}